\documentclass{amsart}
\usepackage{amsmath,amsthm,fixltx2e}
\usepackage{color}
\usepackage{amsxtra}
\usepackage{amscd}
\usepackage{amsfonts}
\usepackage{amssymb}
\usepackage{eucal}
\usepackage{marginnote}

\newtheorem{theorem}{Theorem}[section]
\newtheorem{cor}[theorem]{Corollary}
\newtheorem{lem}[theorem]{Lemma}
\newtheorem{prop}[theorem]{Proposition}

\newtheorem{thm}[theorem]{Theorem}

\theoremstyle{definition}
\newtheorem{defn}{Definition}[section]

\theoremstyle{remark}
\newtheorem{rem}{Remark}[section]
\newtheorem{ex}{Example}[section]

\newcommand{\nc}{\newcommand}

\nc\ol{\overline} \nc\ul{\underline} \nc\wt{\widetilde}
\nc{\z}{\zeta}

\nc{\ZZ}{{\mathbb Z}} \nc{\NN}{{\mathbb N}} \nc{\QQ}{{\mathbb Q}}
\nc{\CC}{{\mathbb C}} \nc{\TT}{{\mathbb T}} \nc{\CP}{{\mathbb {CP}}}

\nc{\A}{{\mathcal A}} \nc{\bA}{{\mathbb A}} \nc{\bP}{{\mathbb P}}
\nc\U{{\mathfrak U}}  \nc\FF{{\mathfrak F}} \nc\D{{\mathfrak D}}
\nc\dd{{\mathfrak d}} \nc{\F}{{\mathcal F}} \nc{\N}{{\mathcal N}}
\nc{\Aa}{{\mathcal A}} \nc{\E}{{\mathcal E}} \nc{\sS}{{\mathbb S}}
\nc{\K}{{\mathcal K}} \nc{\Oo}{{\mathcal O}} \nc{\M}{{\mathcal M}}
\nc{\PP}{{\mathcal P}}

\newcommand{\T}{\mathfrak{t}}
\newcommand{\gl}{\mathfrak{gl}}
\newcommand{\ssl}{\mathfrak{sl}}
\newcommand{\g}{\mathfrak{g}}
\newcommand{\Gg}{\mathrm{g}}
\newcommand{\Qq}{\mathrm{q}}
\newcommand{\Sym}{\mathrm{Sym}}
\newcommand{\ch}{\mathrm{ch}}

\newcommand{\Ker}{\mathrm{Ker}}
\newcommand{\spa}{\mathrm{span}}
\newcommand{\Frac}{\mathrm{Frac}}
\newcommand{\End}{\mathrm{End}}
\newcommand{\Hom}{\mathrm{Hom}}

\newcommand{\HH}{\mathrm{H}}
\newcommand\uu{\mathbf u}
\newcommand\vv{\mathbf v}

\nc{\bp}{{\mathbf{p}}} \nc{\bq}{{\mathbf{q}}} \nc{\ba}{{\mathbf{a}}}
\nc{\bb}{{\mathbf{b}}} \nc{\bc}{{\mathbf{c}}} \nc{\bC}{{\mathbf{C}}}
\nc{\hh}{{\mathbf{h}}} \nc{\xx}{{\mathbf{x}}} \nc{\yy}{{\mathbf{y}}}
\nc{\bj}{{\mathbf{j}}} \nc{\bU}{{\mathbf{U}}} \nc{\bd}{{\mathbf{d}}}

\newcommand{\ad}{\mathop{\rm ad}\nolimits}
\nc{\iso}{{\stackrel{\sim}{\longrightarrow}}}

\begin{document}

\title[The affine Yangian of $\gl_1$ revisited]
 {The affine Yangian of $\gl_1$ revisited}

\author[Alexander Tsymbaliuk]{Alexander Tsymbaliuk}
 \address{Department of Mathematics, MIT, 77 Massachusetts Ave., Cambridge, MA  02139, USA}
 \curraddr{Simons Center for Geometry and Physics, Stony Brook, NY 11794, USA}
 \email{otsymbaliuk@scgp.stonybrook.edu}

\begin{abstract}
  The affine Yangian of $\gl_1$ has recently appeared simultaneously in the work
 of Maulik-Okounkov~\cite{MO} and Schiffmann-Vasserot~\cite{SV3} in connection
 with the Alday-Gaiotto-Tachikawa conjecture.
 While the presentation from~\cite{MO} is purely geometric,
 the algebraic presentation in~\cite{SV3} is quite involved.
 In this article, we provide a simple loop realization of this algebra which can be viewed
 as an ``additivization'' of the quantum toroidal algebra of $\gl_1$ in the same way
 as the Yangian $Y_h(\g)$ is an ``additivization'' of the quantum loop algebra
 $U_q(L\g)$ for a simple Lie algebra $\g$.
 We also explain the similarity between the representation theories of the affine Yangian and
 the quantum toroidal algebras of $\gl_1$ by generalizing the main result
 of~\cite{GTL} to the current settings.
\end{abstract}

\maketitle


\section*{Introduction}

 The goal of this note is twofold.
 First, we provide an explicit loop type presentation of the affine Yangian of $\gl_1$,
 which appeared simultaneously in~\cite{MO,SV3}, and discuss its representation theory.
 We also explain its relation to the quantum toroidal algebra of $\gl_1$.
 The first half of the paper has an expository role.
 However, the author feels that it is worth recalling the geometric
 and representation theoretic aspects of the quantum toroidal algebra of $\gl_1$,
 since in many cases they admit parallel counterparts for the case of the affine Yangian of $\gl_1$.

 The other source of motivation comes from the similarity between
 these theories for the $\gl_1$ and $\ssl_n$ cases.
 In particular, most of the results from this paper admit the corresponding analogues
 for the quantum toroidal and affine Yangian algebras of $\ssl_n$.
 As such analogues are generally much more technical,
 they will be addressed separately in the forthcoming publications.
 Let us mention some new results of the current paper which admit $\ssl_n$-generalizations:

\noindent
 - We treat the $\ssl_n$-generalizations of Theorems~\ref{limit_1},~\ref{limit_1.1},~\ref{limit_2},~\ref{limit_2.1} in~\cite{T2}.

\noindent
 - We treat the $\ssl_n$-generalization of Proposition~\ref{Fock_horizontal} in~\cite{T1}.

\noindent
 - We treat the $\ssl_n$-generalization of Proposition~\ref{matrix_realization} in~\cite{FT2}.

\noindent
 - We treat the $\ssl_n$-generalizations of Proposition~\ref{Fock_a}, Theorem~\ref{flatness}, and Corollary~\ref{faithfulness} in~\cite{TB}.

\noindent
 - We construct homomorphisms
  $\ddot{U}'_{q_1,q_2,q_3}(\gl_1)\to \widehat{\ddot{Y}}'_{h_1,h_2,h_3}(\ssl_n)$
 generalizing $\Upsilon$ in~\cite{TB}.

 This paper is organized as follows:

$\bullet$
 In Section 1, we recall the definition and discuss some basic properties
 of the quantum toroidal algebra $\ddot{U}_{q_1,q_2,q_3}(\gl_1)$.
 A similar class of algebras was first considered in~\cite{DI}.
 This algebra was rediscovered later by different groups of people
 in~\cite{M},~\cite{BS},~\cite{FT} (see Remark~\ref{historical1}).

 We also introduce the key object of this paper, the affine Yangian $\ddot{Y}_{h_1,h_2,h_3}(\gl_1)$.
 This algebra was considered by the author and B.~Feigin in an unpublished work,
 where it was viewed as a natural ``additivization'' of $\ddot{U}_{q_1,q_2,q_3}(\gl_1)$
 in the same way as $Y_h(\g)$ is an ``additivization'' of $U_q(L\g)$.
 The only non-trivial relations (Y4$'$) and (Y5$'$) were determined
 from the requirement that this algebra should naturally act on the sum
 of equivariant cohomology groups of the Hilbert schemes of points on $\mathbb{A}^2$.
 As pointed out by the referee, this algebra also appeared in~\cite{AS}, where the authors
 proved that it is isomorphic to the algebra $\textbf{SH}^{\textbf{c}}$ from~\cite{SV3}.

$\bullet$
 In Section 2, we recall the $\ddot{U}_{q_1,q_2,q_3}(\gl_1)$-action on
 the direct sum of equivariant K-groups of the Hilbert scheme of points
 on $\bA^2$, discovered simultaneously in~\cite{FT,SV}.
 We also formulate a similar result about the $\ddot{Y}_{h_1,h_2,h_3}(\gl_1)$-action
 on an analogous sum of equivariant cohomologies, discovered by
 Schiffmann-Vasserot in~\cite{SV3} and Feigin-Tsymbaliuk (unpublished).
 We conclude that section by sketching our proof of this result,
 which is completely parallel to~\cite{FT} with only one
 modification (Lemma~\ref{psi_0,1,2}) required.

$\bullet$
 In Section 3, we present the generalizations of the results from Section 2
 to the Gieseker moduli spaces $M(r,n)$ (also known as the instanton moduli spaces).
 In the $r=1$ case, we have $M(1,n)\simeq (\bA^2)^{[n]}$ and we recover the actions from Section 2.

$\bullet$
 In Section 4, we recall some series of $\ddot{U}_{q_1,q_2,q_3}(\gl_1)$-representations
 discovered in~\cite{FFJMM,FFJMM2}. Those are constructed from the simplest family of
 \emph{vector representations} $V(u)$ by using the formal coproduct on $\ddot{U}_{q_1,q_2,q_3}(\gl_1)$.
 The simplest example of representations from the category $\Oo$ are the
 \emph{Fock representations} $F(u)$, whose basis is labeled by all Young diagrams.

 We introduce the analogous \emph{vector representations} $^aV(u)$ of $\ddot{Y}_{h_1,h_2,h_3}(\gl_1)$,
 as well as the Fock representations $^aF(u)$, which are of particular interest for us.
 We also prove that the representations of geometric origin from Section 3
 are the tensor products of the aforementioned Fock representations.
 We conclude that section by introducing appropriate categories $\Oo$ and
 generalizing the standard result of~\cite{CP} to the current settings.

$\bullet$
 In Section 5, we describe the limits of the appropriately renormalized algebras
 $\ddot{U}'_{q_1,q_2,q_3}(\gl_1)$ and $\ddot{Y}'_{h_1,h_2,h_3}(\gl_1)$
 as $q_3\to 1$ and $h_3\to 0$, respectively. The resulting limit algebras
 are closely related to central extensions of the algebras of
 \emph{difference operators} on $\CC^*$ and $\CC$.

$\bullet$
 In Section 6, we construct a homomorphism
   $\Upsilon:\ddot{U}'_{q_1,q_2,q_3}(\gl_1)\to \widehat{\ddot{Y}}'_{h_1,h_2,h_3}(\gl_1)$.
 In the limit $h_3\to 0$, this homomorphism is induced by a natural
 isomorphism between the completions of the aforementioned algebras of difference operators.
 This construction is motivated by the corresponding homomorphism
 $U_q(L\g)\to \widehat{Y_h(\g)}$ from~\cite{GTL}.

 We also prove that the formal algebras
 $\ddot{U}'_{q_1,q_2,q_3}(\gl_1)$ and $\ddot{Y}'_{h_1,h_2,h_3}(\gl_1)$
 are flat deformations of their limits
 $\ddot{U}'_{q_1,q_2,q_3}(\gl_1)/(q_3-1)$ and $\ddot{Y}'_{h_1,h_2,h_3}(\gl_1)/(h_3)$.
 In particular, this implies that $\Upsilon$ is injective.
 We also establish the faithfulness of the action of the two
 algebras in interest on the sum of the representations from Section 3.

$\bullet$
 In Section 7, we recall the definition of the small shuffle algebra $S^m$ and
 its commutative subalgebra $\A^m$, which played a crucial role in~\cite{FT}.
 We introduce its additive analogue $S^a$ and the corresponding commutative subalgebra $\A^a$.

$\bullet$
 In Section 8, we discuss a \emph{horizontal} realization of $\ddot{U}_{q_1,q_2,q_3}(\gl_1)$,
 under which the Fock representations $F(u)$ correspond to the vertex type
 representations $\rho_c$ from~\cite{FHHSY}.
 The representations $\rho_c$ provide a new viewpoint towards
 the commutative algebra $\A^m$ (see~\cite{FT2} for more details).
 We conclude that section by introducing and discussing properties of
 the \emph{Whittaker} vectors (see also~\cite{SV3} for the cohomological case).

$\bullet$
 In Appendix, we present main computations.

\subsection*{Acknowledgments:}
 I want to thank Pavel Etingof and Boris Feigin for numerous stimulating discussions.
 I am also grateful to Michael Finkelberg, Sachin Gautam, Valerio Toledano Laredo,
 Michael McBreen, and Andrei Negut for their comments.

 The author is indebted to the anonymous referee for useful comments
 on the first version of the paper, which led to a better exposition of the material.
 The author gratefully acknowledges support from the Simons Center for Geometry and Physics,
 Stony Brook University, at which the final version of this paper was completed.
 The final revision of this paper was carried out
 while the author was partially supported by the NSF grant DMS--1502497.


\section{Basic definitions}

 In this section, we introduce two associative algebras
  $\ddot{U}_{q_1,q_2,q_3}(\gl_1)$ and $\ddot{Y}_{h_1,h_2,h_3}(\gl_1)$,
 which are the key objects studied in this paper.

\subsection{Quantum toroidal algebra of $\gl_1$}
$\ $

 Let $q_1,q_2,q_3=q^2\in \CC^*$ satisfy $q_1q_2q_3=1$,
 while $q_i^m\ne 1$ for any $m\in \NN,i\in \{1,2,3\}$.
 The quantum toroidal algebra of $\gl_1$, denoted by $\ddot{U}_{q_1,q_2,q_3}(\gl_1)$,
 is the unital associative $\CC$-algebra generated by $\{e_i,f_i,\psi_i, \psi_0^{-1}|i\in \ZZ\}$
 with the following defining relations:
\begin{equation}\tag{T0} \label{T0}
  \psi_0\cdot \psi_0^{-1}=\psi_0^{-1}\cdot \psi_0=1,\ \
  [\psi^\pm(z),\psi^\pm(w)]=0,\ \  [\psi^+(z),\psi^-(w)]=0,
\end{equation}
\begin{equation}\tag{T1} \label{T1}
  e(z)e(w)(z-q_1w)(z-q_2w)(z-q_3w)=-e(w)e(z)(w-q_1z)(w-q_2z)(w-q_3z),
\end{equation}
\begin{equation}\tag{T2}\label{T2}
  f(z)f(w)(w-q_1z)(w-q_2z)(w-q_3z)=-f(w)f(z)(z-q_1w)(z-q_2w)(z-q_3w),
\end{equation}
\begin{equation}\tag{T3}\label{T3}
  [e(z),f(w)]=\frac{\delta(z/w)}{(1-q_1)(1-q_2)(1-q_3)}(\psi^{+}(w)-\psi^{-}(z)),
\end{equation}
\begin{equation}\tag{T4}\label{T4}
  \psi^{\pm}(z)e(w)(z-q_1w)(z-q_2w)(z-q_3w)=-e(w)\psi^{\pm}(z)(w-q_1z)(w-q_2z)(w-q_3z),
\end{equation}
\begin{equation}\tag{T5}\label{T5}
  \psi^{\pm}(z)f(w)(w-q_1z)(w-q_2z)(w-q_3z)=-f(w)\psi^{\pm}(z)(z-q_1w)(z-q_2w)(z-q_3w),
\end{equation}
\begin{equation}\tag{T6}\label{T6}
  \mathrm{Sym}_{\mathfrak{S}_3} [e_{i_1},[e_{i_2+1},e_{i_3-1}]]=0,\ \
  \mathrm{Sym}_{\mathfrak{S}_3} [f_{i_1},[f_{i_2+1},f_{i_3-1}]]=0,
\end{equation}
 where these generating series are defined as follows:
  $$e(z):=\sum_{i=-\infty}^{\infty}{e_iz^{-i}},\
    f(z):=\sum_{i=-\infty}^{\infty}{f_iz^{-i}},\
    \psi^{\pm}(z):=\psi_0^{\pm 1}+\sum_{j>0}{\psi_{\pm j}z^{\mp j}},\
    \delta(z):=\sum_{i=-\infty}^{\infty}{z^i}.$$
 The relations (T0--T5) should be viewed as collections of termwise relations which can be
 recovered by evaluating the coefficients of $z^kw^l\ (k,l\in \ZZ)$ on both sides of the equalities.

\begin{rem}\label{historical1}
 The first reference of such algebras (but without the \emph{Serre} relation (T6)) goes back to~\cite{DI}.
 To distinguish, we refer to the algebra with the same collection of generators
 and the defining relations (T0--T5) as the \emph{Ding-Iohara} algebra, see~\cite{FT}.
 These algebras also appeared independently in the work of Burban and Schiffmann (see~\cite[Section 6]{BS}) as
 a direct specialization of Kapranov's theorem to the case of
 elliptic curves, and later with the Serre relation in the work of
 Schiffmann (see~\cite{S}) on the Drinfeld realization of elliptic Hall algebras.
 The algebra $\ddot{U}_{q_1,q_2,q_3}(\gl_1)$ was studied in~\cite{FFJMM,FFJMM2}
 under the name \emph{quantum continuous $\gl_\infty$}.
 We would like to thank the referee for pointing out that the
 algebra $\ddot{U}_{q_1,q_2,q_3}(\gl_1)$ also appeared independently
 in the work of Miki on quantum deformations of $W_{1+\infty}$, see~\cite{M}.
\end{rem}

\subsection{Some properties of $\ddot{U}_{q_1,q_2,q_3}(\gl_1)$}
$\ $

 Let $\ddot{U}^0$ be the subalgebra of $\ddot{U}_{q_1,q_2,q_3}(\gl_1)$
 generated by $\{\psi_i, \psi_0^{-1}\}_{i\in \ZZ}$.
 It is often more convenient to use the generators $\{\psi_0^{\pm 1}, t_i\}_{i\in \ZZ^*}$
 of $\ddot{U}^0$ (here $\ZZ^*:=\ZZ\backslash\{0\}$), defined via
  $$\psi^\pm(z)=\psi^{\pm 1}_0\cdot \exp\left(\mp\sum_{m>0}\frac{\beta_m}{m}t_{\pm m}z^{\mp m}\right),\
    \mathrm{where}\ \beta_m:=(1-q_1^m)(1-q_2^m)(1-q_3^m).$$
 This choice of $t_i$ is motivated by the following two basic results.

\begin{prop}\label{triang}
 The relations \emph{(T4,T5)} are equivalent to

\medskip
\noindent
 \emph{(T4t)} $[\psi_0, e_j]=0,\ [t_i, e_j]=e_{i+j}\ \mathrm{for}\ i\in \ZZ^*, j\in \ZZ$.

\medskip
\noindent
 \emph{(T5t)} $[\psi_0, f_j]=0,\ [t_i, f_j]=-f_{i+j}\ \mathrm{for}\ i\in \ZZ^*, j\in \ZZ$.
\end{prop}

\begin{proof}
$\ $

 The proof of this proposition follows formally from the identity
   $$\log\left(\frac{(z-q_1^{-1}w)(z-q_2^{-1}w)(z-q_3^{-1}w)}{(z-q_1w)(z-q_2w)(z-q_3w)}\right)=
     \sum_{m>0}{-\frac{\beta_m}{m}\cdot \frac{w^m}{z^m}},$$
 where the left-hand side is expressed as a Taylor series in $w/z$.
\end{proof}

\begin{prop}\label{Serre_0}
 If the relations \emph{(T4t,T5t)} hold, then \emph{(T6)} is equivalent to
\begin{equation}\tag{T6t}\label{T6t}
   [e_0,[e_1,e_{-1}]]=0,\ \ [f_0,[f_1,f_{-1}]]=0.
\end{equation}
\end{prop}

\begin{proof}
$\ $

 Setting $i_1=i_2=i_3=0$ in (T6), we get (T6t).
 Let us now deduce (T4t),(T5t),(T6t)$\Rightarrow$(T6).
 We consider only the case with $\{e_i\}$ (the case with $\{f_i\}$ is completely analogous).
 For any $i\in \ZZ^*$, define the operator
  $T_i:\ddot{U}_{q_1,q_2,q_3}(\gl_1)\to \ddot{U}_{q_1,q_2,q_3}(\gl_1)$
 via $T_i:X\mapsto [t_i,X]=t_iX-Xt_i$.
 Combining the relation (T4t) with the algebraic equality
   $$[t,[a,[b,c]]]=[[t,a],[b,c]]+[a,[[t,b],c]]+[a,[b,[t,c]]],$$
 we get
   $$T_{i_1}T_{i_2}T_{i_3}-T_{i_1+i_2}T_{i_3}-T_{i_1+i_3}T_{i_2}-T_{i_2+i_3}T_{i_1}+2T_{i_1+i_2+i_3}:
     [e_0,[e_1,e_{-1}]]\mapsto \mathrm{Sym}_{\mathfrak{S}_3} [e_{i_1},[e_{i_2+1},e_{i_3-1}]]$$
 where we set $T_0:=3\mathrm{Id}$. This completes our proof.
\end{proof}

 Let $\ddot{U}^-$ and $\ddot{U}^+$ be the subalgebras of $\ddot{U}_{q_1,q_2,q_3}(\gl_1)$
 generated by $\{f_i\}$ and $\{e_i\}$, respectively.
 We conclude this section by the standard result (see Appendix A for a proof):

\begin{prop}\label{triangular_m}
 (a) (Triangular decomposition for $\ddot{U}_{q_1,q_2,q_3}(\gl_1)$)
 The multiplication map
  $m:\ddot{U}^-\otimes \ddot{U}^0\otimes \ddot{U}^+\to \ddot{U}_{q_1,q_2,q_3}(\gl_1)$
 is an isomorphism of vector spaces.

\noindent
 (b) The subalgebras $\ddot{U}^-, \ddot{U}^+, \ddot{U}^0$ are generated by
    $\{f_i\},\ \{e_i\},\ \{\psi_i, \psi_0^{-1}\}$
  with the defining relations (T2,T6), (T1,T6), and (T0), respectively.
\end{prop}

\subsection{Affine Yangian of $\gl_1$}
$\ $

 Let $h_1, h_2, h_3\in \CC$ satisfy $h_1+h_2+h_3=0$.
 The affine Yangian of $\gl_1$, denoted by $\ddot{Y}_{h_1,h_2,h_3}(\gl_1)$,
 is the unital associative $\CC$-algebra generated by $\{e_j,f_j,\psi_j|j\in \ZZ_+\}$
 (here $\ZZ_+:=\NN\cup\{0\}$) with the following defining relations:

\medskip
\noindent
 (Y0) $[\psi_i,\psi_j]=0,$

\medskip
\noindent
 (Y1) $[e_{i+3},e_j]-3[e_{i+2},e_{j+1}]+3[e_{i+1},e_{j+2}]-[e_i,e_{j+3}]+\sigma_2([e_{i+1},e_j]-[e_i,e_{j+1}])=\sigma_3\{e_i,e_j\},$

\medskip
\noindent
 (Y2) $[f_{i+3},f_j]-3[f_{i+2},f_{j+1}]+3[f_{i+1},f_{j+2}]-[f_i,f_{j+3}]+\sigma_2([f_{i+1},f_j]-[f_i,f_{j+1}])=-\sigma_3\{f_i,f_j\},$

\medskip
\noindent
 (Y3) $[e_i,f_j]=\psi_{i+j},$

\medskip
\noindent
 (Y4) $[\psi_{i+3},e_j]-3[\psi_{i+2},e_{j+1}]+3[\psi_{i+1},e_{j+2}]-[\psi_i,e_{j+3}]+\sigma_2([\psi_{i+1},e_j]-[\psi_i,e_{j+1}])=\sigma_3\{\psi_i,e_j\},$

\medskip
\noindent
 (Y4$'$) $[\psi_0,e_j]=0,\ \ [\psi_1,e_j]=0,\ \ [\psi_2,e_j]=2e_j,$

\medskip
\noindent
 (Y5) $[\psi_{i+3},f_j]-3[\psi_{i+2},f_{j+1}]+3[\psi_{i+1},f_{j+2}]-[\psi_i,f_{j+3}]+\sigma_2([\psi_{i+1},f_j]-[\psi_i,f_{j+1}])=-\sigma_3\{\psi_i,f_j\},$

\medskip
\noindent
 (Y5$'$) $[\psi_0,f_j]=0,\ \ [\psi_1,f_j]=0,\ \ [\psi_2,f_j]=-2f_j,$

\medskip
\noindent
 (Y6) $\mathrm{Sym}_{\mathfrak{S}_3} [e_{i_1},[e_{i_2},e_{i_3+1}]]=0,\ \ \mathrm{Sym}_{\mathfrak{S}_3} [f_{i_1},[f_{i_2},f_{i_3+1}]]=0$,

\medskip
\noindent
 where $i,j\in \ZZ_+$ and we set $\{a,b\}:=ab+ba,\ \sigma_2:=h_1h_2+h_1h_3+h_2h_3,\ \sigma_3:=h_1h_2h_3$.

\begin{rem}
 This algebra can be considered as a natural ``additivization'' of
 $\ddot{U}_{q_1,q_2,q_3}(\gl_1)$ in the same way as the classical Yangian
 $Y_h(\g)$ is an ``additivization'' of the quantum loop algebra $U_q(L\g)$.
 The relations (Y0--Y6) were obtained jointly with B.~Feigin in an unpublished work.
 To the best of our knowledge, the first written reference goes back to~\cite[Section 3.2]{AS}
 (with the particular choice $\{h_1,h_2,h_3\}=\{1,-\kappa,\kappa-1\}$),
 where it was shown to be isomorphic to the algebra ${\bf{SH}^{\bf{c}}}$ from~\cite{SV3}.
 The same algebra was also implicitly considered in~\cite{MO}.
\end{rem}

\subsection{Some properties of $\ddot{Y}_{h_1,h_2,h_3}(\gl_1)$}
$\ $

 Define the generating series
   $$e(z):=\sum_{j\geq 0}{e_jz^{-j-1}},\ \
     f(z):=\sum_{j\geq 0}{f_jz^{-j-1}},\ \
     \psi(z):=1+\sigma_3\sum_{j\geq 0}{\psi_jz^{-j-1}}.$$

 Let $\ddot{Y}^-,\ddot{Y}^0,\ddot{Y}^+$ be the subalgebras of $\ddot{Y}_{h_1,h_2,h_3}(\gl_1)$
 generated by $\{f_j\},\ \{\psi_j\}$, and $\{e_j\}$, respectively.
 Let $\ddot{Y}^{\geq}$ and $\ddot{Y}^{\leq}$ be the subalgebras of $\ddot{Y}_{h_1,h_2,h_3}(\gl_1)$
 generated by $\ddot{Y}^0,\ddot{Y}^+$ and $\ddot{Y}^0,\ddot{Y}^-$, respectively.
 The following is analogous to Proposition~\ref{triangular_m}:

\begin{prop}\label{triangular_a}
 (a) $\ddot{Y}^0$ is a polynomial algebra in the generators $\{\psi_j\}$.

\noindent
 (b) $\ddot{Y}^-$ and $\ddot{Y}^+$ are the algebras generated by
 $\{f_j\}$ and $\{e_j\}$ with the defining relations $(Y2,Y6)$ and $(Y1,Y6)$, respectively.

\noindent
 (c) $\ddot{Y}^{\leq}$ and $\ddot{Y}^{\geq}$ are the algebras generated by
 $\{\psi_j,f_j\}$ and $\{\psi_j, e_j\}$ with the defining relations
 $(Y0,Y2,Y5,Y5',Y6)$ and $(Y0,Y1,Y4,Y4',Y6)$, respectively.

\noindent
 (d) Multiplication induces an isomorphism of vector spaces
    $$m:\ddot{Y}^-\otimes \ddot{Y}^0\otimes \ddot{Y}^+\iso \ddot{Y}_{h_1,h_2,h_3}(\gl_1).$$
\end{prop}

 Let us consider the algebra homomorphisms
   $$\sigma^+:\ddot{Y}^{\geq}\to \ddot{Y}^{\geq}\
     \mathrm{determined\ by}\ \psi_j\mapsto \psi_j,\ e_j\mapsto e_{j+1}$$
 and
   $$\sigma^-:\ddot{Y}^{\leq}\to \ddot{Y}^{\leq}\
     \mathrm{determined\ by}\ \psi_j\mapsto \psi_j,\ f_j\mapsto f_{j+1}.$$
 These are well defined due to Proposition~\ref{triangular_a}. Let
   $$\mu: \ddot{Y}_{h_1,h_2,h_3}(\gl_1)^{\otimes 2}\to \ddot{Y}_{h_1,h_2,h_3}(\gl_1)$$
 be the multiplication. The following result is straightforward:

\begin{prop}\label{yangian_generating}
 Define
  $\sigma^\pm_{(1)}(a\otimes b):=\sigma^\pm(a)\otimes b$
 and
  $\sigma^\pm_{(2)}(a\otimes b):=a\otimes \sigma^\pm(b)$.
 We also set $P(z,w):=(z-w-h_1)(z-w-h_2)(z-w-h_3)$. Then:

(a) The relation (Y0) is equivalent to $[\psi(z),\psi(w)]=0.$

\medskip
(b) The relation (Y1) is equivalent to
     $$\partial_z^3 \mu(P(z,\sigma_{(2)}^+)e(z)\otimes e_j + P(\sigma_{(1)}^+,z)e_j\otimes e(z))=0\ \ \mathrm{for}\ j\in \ZZ_+.$$

(c) The relation (Y2) is equivalent to
     $$\partial_z^3 \mu(P(\sigma_{(2)}^-,z)f(z)\otimes f_j+P(z,\sigma_{(1)}^-)f_j\otimes f(z))=0\ \ \mathrm{for}\ j\in \ZZ_+.$$

(d) The relation (Y3) is equivalent to
     $$\sigma_3\cdot(w-z)[e(z),f(w)]=\psi(z)-\psi(w).$$

(e) The relations (Y4,Y4$'$) are equivalent to
     $$P(z,\sigma^+)\psi(z)e_j+P(\sigma^+,z)e_j\psi(z)=0\ \ \mathrm{for}\ j\in \ZZ_+.$$

(f) The relations (Y5,Y5$'$) are equivalent to
     $$P(\sigma^-,z)\psi(z)f_j+P(z,\sigma^-)f_j\psi(z)=0\ \ \mathrm{for}\ j\in \ZZ_+.$$
\end{prop}


\section{Representation theory via the Hilbert scheme}

\subsection{Correspondences and fixed points for $(\bA^2)^{[n]}$} \label{Hilb1}
$\ $

 Throughout this section $X=\bA^2$.
 Let $X^{[n]}$ be the Hilbert scheme of $n$ points on $X$.
 Its $\CC$-points are the codimension $n$ ideals $J\subset\CC[x,y]$.
 Let $P[i]\subset \coprod_{n}{X^{[n]}\times X^{[n+i]}}$ be the Nakajima-Grojnowski correspondence.
 For $i>0$, the correspondence $P[i]\subset \coprod_{n}{X^{[n]}\times X^{[n+i]}}$ consists
 of all pairs of ideals $(J_1, J_2)$ of $\CC[x,y]$ of codimension $n,\ n+i$ respectively,
 such that $J_2\subset J_1$ and the factor $J_1/J_2$ is supported at a single point.
 It is known that $P[1]$ is a smooth variety.
 Let $L$ be the \emph{tautological} line bundle on $P[1]$ whose fiber
 at a point $(J_1, J_2) \in P[1]$ equals $J_1/J_2$.
 There are natural projections $\bp, \bq$ from $P[1]$ to
 $X^{[n]}$ and $X^{[n+1]}$, correspondingly.

 Consider a natural action of $\TT=\CC^{*}\times \CC^{*}$ on each $X^{[n]}$,
 induced from the one on $X$ given by the formula $(t_1,t_2)\cdot (x,y)=(t_1 x, t_2 y)$.
 The set $(X^{[n]})^{\TT}$ of $\TT$-fixed points in $X^{[n]}$ is
 finite and is in bijection with size $n$ Young diagrams.
 For such a Young diagram $\lambda=(\lambda_1,\ldots, \lambda_k)$,
 the corresponding ideal $J_{\lambda}\in (X^{[n]})^{\TT}$ is given by
   $J_\lambda=\CC[x,y]\cdot (\CC x^{\lambda_1}y^0\oplus\cdots\oplus\CC x^{\lambda_k}y^{k-1}\oplus\CC y^k)$.

 \textbf{Notation:} For a Young diagram $\lambda$, let $\lambda^*$
 be the conjugate diagram and define $|\lambda|:=\sum \lambda_i$.
 For a box $\square$ with the coordinates $(i,j)$, we define
   $a_\lambda(\square):=\lambda_j-i,\ l_\lambda(\square):=\lambda^*_i-j.$
 The diagram obtained from $\lambda$ by adding a box to its $j$th row
 is denoted by $\lambda+\square_j$, or simply by $\lambda+j$.

\subsection{Geometric $\ddot{U}_{q_1,q_2,q_3}(\gl_1)$-action I}\label{rep_M}
$\ $

 Let ${}'M$ be the direct sum of equivariant (complexified) $K$-groups:
   ${}'M=\bigoplus_{n}K^{\TT}(X^{[n]})$.
 It is a module over $K^{\TT}(\textrm{pt})=\CC[t_1^{\pm 1},t_2^{\pm 1}]$.
 Set $t_3:=t_1^{-1}t_2^{-1}\in \CC(t_1,t_2)=\Frac(K^{\TT}(\textrm{pt}))$ and
 consider a quadratic extension $\mathbb{F}=\CC(t_1,t_2)[t]/(t^2-t_3)$.
 We define
   $$M:=\ {}'M\otimes_{K^{\TT}(\textrm{pt})} \mathbb{F}.$$

 It has a natural grading:
  $M=\bigoplus_{n}M_{n},\ M_{n}=K^{\TT}(X^{[n]})\otimes_{K^{\TT}(\textrm{pt})} \mathbb{F}.$
 According to the localization theorem,
 restriction to the $\TT$-fixed point set induces an isomorphism
   $K^{\TT}(X^{[n]})\otimes_{K^{\TT}(\textrm{pt})}\mathbb{F}\overset{\sim}\longrightarrow
    K^{\TT}((X^{[n]})^{\TT})\otimes_{K^{\TT}(\textrm{pt})}\mathbb{F}.$
 The structure sheaves $\{\lambda\}$ of the $\TT$-fixed points $J_\lambda$
 (defined in Section~\ref{Hilb1}) form a basis in
   $\bigoplus_{n}K^{\TT} ((X^{[n]})^{\TT})\otimes_{K^{\TT}(\textrm{pt})}\mathbb{F}$.
 Since embedding of a point $J_\lambda$ into $X^{[|\lambda|]}$ is a proper morphism,
 the direct image in the equivariant $K$-theory is well defined, and we denote by
 $[\lambda]\in M_{|\lambda|}$ the direct image of the structure sheaf $\{\lambda\}$.
 The set $\{[\lambda]\}$ forms a basis of $M$.

 Let $\mathfrak{F}$ be the \textit{tautological} vector bundle on $X^{[n]}$
 whose fiber at a point $J\in X^{[n]}$ equals the quotient $\CC[x,y]/J$.
 Consider the following generating series $\ba(z), \bc(z)\in M(z)$:
  $$\ba(z):=\Lambda^{\bullet}_{-1/z}(\mathfrak{F})=\sum_{i\geq 0}{[\Lambda^i(\mathfrak{F})](-1/z)^i},$$
  $$\bc(z):=\ba(zt_1)\ba(zt_2)\ba(zt_3)\ba(zt_1^{-1})^{-1}\ba(zt_2^{-1})^{-1}\ba(zt_3^{-1})^{-1}.$$

 Finally, we define the linear operators $e_i, f_i, \psi_i, \psi_0^{-1}\ (i\in \ZZ)$ acting on $M$:
\begin{equation}\label{raz}
  e_i=\bq_*(L^{\otimes i}\otimes \bp^*):\ M_{n}\to M_{n+1},
\end{equation}
\begin{equation}\label{dva}
  f_i=-t^{-1}\cdot \bp_*(L^{\otimes (i-1)}\otimes \bq^*):\ M_{n}\to M_{n-1},
\end{equation}
\begin{equation}\label{tree}
  \psi^{\pm}(z)=\psi_0^{\pm 1}+\sum_{j=1}^{\infty} \psi_{\pm j}z^{\mp j}:=
  \left(\frac{t^{-1}-tz^{-1}}{1-z^{-1}}\bc(z)\right)^{\pm} \in \prod_n \End(M_{n})[[z^{\mp 1}]],
\end{equation}
 where $\gamma(z)^\pm$ denotes the expansion of a rational function $\gamma(z)$ in $z^{\mp 1}$, respectively.

\begin{theorem}\label{DIm}
 The operators $e_i, f_i, \psi_i, \psi_0^{-1}$, defined in (\ref{raz}--\ref{tree}),
 satisfy the relations~(T0--T6) with the parameters $q_i=t_i$.
 This endows $M$ with the structure of $\ddot{U}_{q_1,q_2,q_3}(\gl_1)$-representation.
\end{theorem}

\begin{rem}
 This theorem was proved simultaneously and independently in~\cite{FT} and~\cite{SV}.
\end{rem}

\subsection{Geometric $\ddot{Y}_{h_1,h_2,h_3}(\gl_1)$-action I} \label{rep_V}
$\ $

 Let ${}'V$ be the direct sum of equivariant (complexified) cohomology:
   ${}'V=\bigoplus_{n}\HH^\bullet_{\TT}(X^{[n]}).$
 It is a module over $\HH^\bullet_{\TT}(\textrm{pt})=\CC[\T]= \CC[s_1,s_2]$,
 where $\T$ is the Lie algebra of $\TT$.
 We define
   $$V:=\ {}'V\otimes_{\HH^\bullet_{\TT}(\textrm{pt})} \Frac(\HH^\bullet_{\TT}(\textrm{pt}))=
     {}'V\otimes_{\CC [s_1,s_2]} \CC (s_1,s_2).$$

 It has a natural grading:
   $V=\bigoplus_{n}V_n,\ V_n=\HH^\bullet_{\TT}(X^{[n]})\otimes_{\HH^\bullet_{\TT}(\textrm{pt})} \Frac(\HH^\bullet_{\TT}(\textrm{pt})).$
 According to the localization theorem,
 restriction to the $\TT$-fixed point set induces an isomorphism
  $$\HH^\bullet_{\TT}(X^{[n]})\otimes_{\HH^\bullet_{\TT}(\textrm{pt})} \Frac(\HH^\bullet_{\TT}(\textrm{pt}))\overset{\sim}\longrightarrow
    \HH^\bullet_{\TT}((X^{[n]})^{\TT})\otimes_{\HH^\bullet_{\TT}(\textrm{pt})}\Frac(\HH^\bullet_{\TT}(\textrm{pt})).$$

 The fundamental cycles of the $\TT$-fixed points $J_\lambda$ form a basis in
   $\bigoplus_{n}\HH^\bullet_{\TT} ((X^{[n]})^{\TT})\otimes_{\HH^\bullet_{\TT}(\textrm{pt})}\Frac (\HH^\bullet_{\TT}(\textrm{pt}))$.
 Since embedding of a point $J_\lambda$ into $X^{[|\lambda|]}$ is a proper morphism,
 the direct image in the equivariant cohomology is well defined, and we denote
 by $[\lambda]\in V_{|\lambda|}$ the direct image of the fundamental cycle of
 the point $J_\lambda$. The set $\{[\lambda]\}$ forms a basis of $V$.

 Consider the following generating series $\bC(z)\in V[[z^{-1}]]$:
  $$\bC(z):=\left(\frac{\ch(\FF t_1^{-1},-z^{-1})\ch(\FF t_2^{-1},-z^{-1})\ch(\FF t_3^{-1},-z^{-1})}
    {\ch(\FF t_1,-z^{-1})\ch(\FF t_2,-z^{-1})\ch(\FF t_3,-z^{-1})}\right)^+,$$
 where $\ch(F,\bullet)$ denotes the Chern polynomial of $F$. We also set $s_3:=-s_1-s_2$.

 Finally, we define the linear operators $e_j, f_j, \psi_j\ (j\in \ZZ_+)$ acting on $V$:
\begin{equation}\tag{1$'$}\label{raz'}
  e_j=\bq_*(c_1(L)^j\cdot \bp^*):\ V_n\to V_{n+1},
\end{equation}
\begin{equation}\tag{2$'$}\label{dva'}
  f_j=\bp_*(c_1(L)^j\cdot \bq^*):\ V_n\to V_{n-1},
\end{equation}
\begin{equation}\tag{3$'$}\label{tree'}
  \psi(z)=1+s_1s_2s_3\sum_{j=0}^{\infty} \psi_jz^{-j-1}:=(1-s_3/z)\bC(z) \in \prod_n \End(V_n)[[z^{-1}]].
\end{equation}

\begin{theorem}\label{DIa}
 The operators $e_j, f_j, \psi_j$, defined in (\ref{raz'}--\ref{tree'}), satisfy
 the relations~(Y0--Y6) with the parameters $h_i=s_i$.
 This endows $V$ with the structure of $\ddot{Y}_{h_1,h_2,h_3}(\gl_1)$-representation.
\end{theorem}

\begin{rem}
 This result is a natural ``cohomological'' analogue of Theorem~\ref{DIm}.
 It was obtained jointly with B.~Feigin in an unpublished work.
 The first written reference goes back to~\cite{SV3} (to be precise,
 in~\cite{SV3} an action of an algebra ${\bf{SH}^{\bf{c}}}$ on the space $V$ was constructed,
 and in~\cite{AS} it was shown that ${\bf{SH}^{\bf{c}}}$ is isomorphic to the affine Yangian of $\gl_1$).
\end{rem}

 In the remaining part of this section, we explain how the proof of
 Theorem~\ref{DIm} from~\cite{FT} can be adapted almost automatically
 to the cohomological case of Theorem~\ref{DIa}.
 We start with an explicit computation of the matrix coefficients of
 $e_p,f_p,\psi(z)$ in the fixed point basis $\{[\lambda]\}$ of $V$.
 For a linear operator $A\in \End(V)$, we use
 $A_{\mid{[\lambda,\mu]}}$ to denote the coefficient of $[\mu]$ in $A([\lambda])$.
 We also set $\chi(\square_{i,j}):=(i-1)s_1+(j-1)s_2$
 for a box $\square_{i,j}$ located in the $j$th row and $i$th column.

\begin{lem}\label{matrix_elements_a}
\noindent
 (a) The only nonzero matrix coefficients of the operators $e_p, f_p$ are as follows:
 $$e_{p\mid{[\lambda-i,\lambda]}}=\frac{((\lambda_i-1)s_1+(i-1)s_2)^p}{(s_1+s_2)((\lambda_1-\lambda_i+1)s_1+(1-i)s_2)}
    \cdot \prod_{j\geq 1}\frac{(\lambda_j-\lambda_i+1)s_1+(j-i+1)s_2}{(\lambda_{j+1}-\lambda_i+1)s_1+(j-i+1)s_2},$$
 $$f_{p\mid{[\lambda+i,\lambda]}}=\frac{(\lambda_is_1+(i-1)s_2)^p((\lambda_i-\lambda_1+1)s_1+is_2)}{s_1+s_2}
    \cdot \prod_{j\geq 1}\frac{(\lambda_i-\lambda_{j+1}+1)s_1+(i-j)s_2}{(\lambda_i-\lambda_j+1)s_1+(i-j)s_2}.$$
\noindent
 (b) $\psi(z)$ is diagonal in the fixed point basis and the eigenvalue of $\psi(z)$ on $[\lambda]$ equals
  $${\psi(z)}_{\mid_\lambda}=
    \left(\left(1-\frac{s_3}{z}\right)
    \prod_{\square \in\lambda}\frac{(z-\chi(\square)+s_1)(z-\chi(\square)+s_2)(z-\chi(\square)+s_3)}
    {(z-\chi(\square)-s_1)(z-\chi(\square)-s_2)(z-\chi(\square)-s_3)}\right)^+.$$
\end{lem}

 This is a ``cohomological'' analogue of~\cite[Lemma 3.1, Proposition 3.1]{FT}.
 Using this result, proof of Theorem~\ref{DIa} is reduced to
 a routine verification of the relations (Y0--Y6) in the fixed point basis.
 The only non-trivial relation is (Y3).
 A similar issue in the case of $K$-theory was resolved by~\cite[Lemma 4.1]{FT}.
 We conclude this section by proving an analogous result:

\begin{lem}\label{psi_0,1,2}
 Consider the linear operator $\phi_{i,j}:=[e_i,f_j]$ acting on $V$.

\noindent
 (a) The operator $\phi_{i,j}$ is diagonal in the fixed point basis $\{[\lambda]\}$ of $V$.

\noindent
 (b) For any Young diagram $\lambda$, we have
 $\phi_{i,j}([\lambda])={\gamma_{i+j}}_{\mid_\lambda}\cdot [\lambda]$, where
\begin{multline}\tag{$\sharp$}\label{gamma}
  {\gamma_m}_{\mid_\lambda}=
    s_1^{-2}\sum_{i=1}^k{y_i^m\prod_{1\leq j\leq k}^{j\ne i}
    {\frac{(y_i-y_j+s_2)(y_j-y_i+s_1+s_2)}{(y_i-y_j)(y_j-y_i+s_1)}}
     \cdot \frac{y_i+s_1+(1-k)s_2}{-y_i+ks_2}}-\\
   s_1^{-2}\sum_{i=1}^k{(y_i+s_1)^m\prod_{1\leq j\leq k}^{j\ne i}
   {\frac{(y_j-y_i+s_2)(y_i-y_j+s_1+s_2)}{(y_j-y_i)(y_i-y_j+s_1)}}
    \cdot \frac{y_i+2s_1+(1-k)s_2}{-y_i-s_1+ks_2}}.
\end{multline}

 Here $y_i:=(\lambda_i-1)s_1+(i-1)s_2$ and $k$ is a positive integer such that $\lambda_k=0$.

\noindent
 (c) For any Young diagram $\lambda$, we have
      $${\gamma_0}_{\mid_\lambda}=-1/s_1s_2,\ \
        {\gamma_1}_{\mid_\lambda}=0,\ \
        {\gamma_2}_{\mid_\lambda}=2|\lambda|.$$
\end{lem}

\begin{proof}$\ $

   Parts (a) and (b) follow from Lemma~\ref{matrix_elements_a}(a) by straightforward calculations.

 Let us now prove part (c).
 For $m\geq 0$, the expression in the right-hand side of ($\sharp$) is a rational function in $y_i$
 with the only possible (simple) poles at $y_i=y_j,\ y_i=y_j+s_1,\ y_i=ks_2,\ y_i=-s_1+ks_2$.
 A straightforward computation shows that the residues at these points are in fact zero.
 Therefore, ${\gamma_m}_{\mid_\lambda}\in \CC(s_1,s_2)[y_1,y_2,\ldots]$ for $m\geq 0$.

\noindent
 $\circ$ \textit{Case 1: $m=0$}.

 Since ${\gamma_0}_{\mid_\lambda}$ is a polynomial in $y_i$ of degree $\leq 0$,
 it must be an element of $\CC(s_1,s_2)$ independent of $\lambda$.
 Evaluating at the empty diagram, we find
  ${\gamma_0}_{\mid_\lambda}={\gamma_0}_{\mid_\emptyset}=-1/s_1s_2$.

\noindent
 $\circ$ \textit{Case 2: $m=1$}.

 The eigenvalue ${\gamma_1}_{\mid_\lambda}$ is a polynomial in $y_i$ of degree $\leq 1$.
 However, the limit of the right-hand side of ($\sharp$) with $m=1$ as $y_{i_0}\to \infty$
 while $y_j$ are fixed for all $j\ne i_0$ is finite for any index $i_0$.
 Therefore, ${\gamma_1}_{\mid_\lambda}$ must be a degree $0$ polynomial.
 Hence, ${\gamma_1}_{\mid_\lambda}={\gamma_1}_{\mid_\emptyset}=0$.

\noindent
 $\circ$ \textit{Case 3: $m=2$}.

 Recall that ${\gamma_2}_{\mid_\lambda}$ is a polynomial in $y_i$ of degree $\leq 2$.
 Arguments similar to the above show that ${\gamma_2}_{\mid_\lambda}$ is actually
 a degree $\leq 1$ polynomial in $y_i$. Let us now compute its principal linear part.
 The coefficient of $y_{i_0}$ equals the limit
  $\underset{\xi\to \infty}\lim \frac{1}{\xi}{\gamma_2}_{\mid_\lambda}$
 as $y_j$ is fixed for $j\ne i_0$ and $y_{i_0}=\xi\to \infty$.
 By $(\sharp)$ this is just $\frac{2}{s_1}$.
 Therefore, there exists $F(s_1,s_2)\in \CC(s_1,s_2)$ such that
  ${\gamma_2}_{\mid_\lambda}=\frac{2}{s_1}(\wt{y}_1+\wt{y}_2+\ldots)+F(s_1,s_2)=2|\lambda|+F(s_1,s_2)$,
 where $\wt{y}_i:=y_i-((i-1)s_2-s_1)$
 (the sequence $\{\wt{y}_i\}$ stabilizes to 0 as $i\to \infty$, unlike $\{y_i\}$).
 Evaluating at the empty Young diagram, we find $F(s_1,s_2)={\gamma_2}_{\mid_\emptyset}=0$.
 The equality ${\gamma_2}_{\mid_\lambda}=2|\lambda|$ follows.
\end{proof}

  Arguments similar to those from~\cite{FT} imply ${\gamma_m}_{\mid_\lambda}={\psi_m}_{\mid_\lambda}$ (see also Appendix B).

\begin{rem}
 Due to Lemma 2.3(b), the next $\psi$-coefficient acts in the following way:
   $${\psi_3}_{\mid_\lambda}=6\sum_{\square\in \lambda}\chi(\square)+2(s_1+s_2)|\lambda|.$$
 In particular, $\frac{1}{6}(\psi_3+s_3\psi_2)$ corresponds to the cup product with $c_1(\FF)$.
 This operator was first studied by M.~Lehn.
 It is also related to the Laplace-Beltrami operator (see~\cite[Section 4]{Na2}).
\end{rem}


\section{Representation theory via the Gieseker space}

 The purpose of this section is to provide generalizations
 of the results from Section 2 to the higher rank $r$ cases, that is,
 replacing $(\bA^2)^{[n]}$ by the Gieseker moduli spaces $M(r,n)$.


\subsection{Correspondences and fixed points for $M(r,n)$}\label{Gieseker1}
$\ $

 We recall some basics on the Gieseker framed moduli spaces  $M(r,n)$
 of torsion free sheaves on $\bP^2$ of rank $r$ with $c_2=n$.
 Its $\CC$-points are the isomorphism classes of pairs $\{(E,\Phi)\}$,
 where $E$ is a torsion free sheaf on $\bP^2$ of rank $r$ with $c_2(E)=n$, and which is locally free
 in a neighborhood of the line $l_\infty=\{(0:z_1:z_2)\}\subset \bP^2$, while
 $\Phi:E_{\mid_{l_\infty}}\iso \Oo_{l_\infty}^{\oplus r}$ (called a \emph{framing at infinity}).

 This space has an alternative quiver description (see~\cite[Ch. 2]{Na1} for details):
   $$M(r,n)=\M(r,n)/GL_n(\CC),\ \M(r,n)=\{(B_1,B_2,i,j)|[B_1,B_2]+ij=0\}^s,$$
 where $B_1,B_2\in \End(\CC^n), i\in \Hom(\CC^r,\CC^n), j\in \Hom(\CC^n,\CC^r)$,
 the $GL_n(\CC)$-action is given by
   $g\cdot(B_1,B_2,i,j)=(gB_1g^{-1},gB_2g^{-1}, gi, jg^{-1}),$
 while the superscript $s$ symbolizes the stability condition
   $\mathrm{``there\ exists\ no\ subspace}\ S\subsetneq \CC^n\
    \mathrm{such\ that}\ B_\alpha(S)\subset S\ (\alpha=1,2)\
    \mathrm{and}\ \mathrm{Im}(i)\subset S \mathrm{"}.$
 Let $\FF_r$ be the \emph{tautological} rank $n$ vector bundle on $M(r,n)$.

 Consider a natural action of $\TT_r=(\CC^{*})^2\times (\CC^{*})^r$ on $M(r,n)$, where
 $(\CC^*)^2$ acts on $\bP^2$ via $(t_1,t_2)\cdot([z_0:z_1:z_2])=[z_0:t_1z_1:t_2z_2]$, while
 $(\CC^*)^r$ acts by rescaling the framing isomorphism.
 The set $M(r,n)^{\TT_r}$ of $\TT_r$-fixed points in $M(r,n)$ is
 finite and is in bijection with $r$-tuples of Young diagrams $\bar{\lambda}=(\lambda^1,\ldots,\lambda^r)$
 satisfying $|\bar{\lambda}|:=|\lambda^1|+\ldots+|\lambda^r|=n$,
 denoted by $\bar{\lambda}\vdash n$ (see~\cite[Proposition 2.9]{NY}).
 For such $\bar{\lambda}$, the corresponding point $\xi_{\bar{\lambda}}\in M(r,n)^{\TT_r}$
 is represented by $(E_{\bar{\lambda}},\Phi_{\bar{\lambda}})$, where
 $E_{\bar{\lambda}}=J_{\lambda^1}\oplus\cdots\oplus J_{\lambda^r}$
 and $\Phi_{\bar{\lambda}}$ is a sum of natural inclusions
 $J_{\lambda^j \mid_{l_\infty}}\hookrightarrow \Oo_{l_\infty}$.

 Following~\cite[Section 5]{Na3}, we recall the \emph{Hecke correspondences},
 which generalize $P[1]$ from Section 2.1 to higher ranks.
 Consider $\M(r;n,n+1)\subset\M(r,n)\times \M(r,n+1)$ consisting of pairs of tuples
 $\{(B^{(k)}_1,B^{(k)}_2,i^{(k)},j^{(k)})\}$ for $k=n,n+1$, such that there exists
 $\xi:\CC^{n+1}\to \CC^{n}$ satisfying
   $$\xi B^{(n+1)}_1=B^{(n)}_1\xi,\ \xi B^{(n+1)}_2=B^{(n)}_2\xi,\ \xi i^{(n+1)}=i^{(n)},\ j^{(n+1)}=j^{(n)}\xi.$$
 The stability condition implies $\xi$ is surjective.
 Therefore $S:=\Ker\ \xi\subset\CC^{n+1}$ is a 1-dimensional subspace of $\Ker\ j^{(n+1)}$
 invariant with respect to $B^{(n+1)}_1, B^{(n+1)}_2$.
 This provides an identification of $\M(r;n,n+1)$
 with pairs of $(B^{(n+1)}_1,B^{(n+1)}_2,i^{(n+1)},j^{(n+1)})\in \M(r,n+1)$
 and a 1-dimensional subspace $S\subset \CC^{n+1}$ satisfying the above conditions.
 Taking the latter viewpoint, we define the Hecke correspondence $M(r;n,n+1)$
 as the quotient $M(r;n,n+1)=\M(r;n,n+1)/GL_{n+1}(\CC)$.
 Let $L_r$ be the \emph{tautological} line bundle on $M(r;n,n+1)$,
 while $\bp_r, \bq_r$ be the natural projections from $M(r;n,n+1)$ to $M(r,n)$ and $M(r,n+1)$, correspondingly.
 The set $M(r;n,n+1)^{\TT_r}$ of $\TT_r$-fixed points in $M(r;n,n+1)$ is in
 bijection with pairs of $r$-tuples of diagrams $\bar{\lambda}\vdash n, \bar{\mu}\vdash n+1$
 such that $\lambda^j\subseteq\mu^j$ for $1\leq j\leq r$;
 the corresponding fixed point will be denoted by $\xi_{\bar{\lambda},\bar{\mu}}$.

 Our computations are based on the following well-known result (see~\cite{Na3,NY}):

\begin{prop}
 (a) The variety $M(r;n,n+1)$ is smooth of complex dimension $2rn+r+1$.

\noindent
 (b) The $\TT_r$-character of the tangent space to $M(r,n)$ at the $\TT_r$-fixed point $\xi_{\bar{\lambda}}$ equals
     $$T_{\bar{\lambda}}=\sum_{a,b=1}^r \frac{\chi_b}{\chi_a}
       \left(\sum_{\square\in \lambda^a}t_1^{-a_{\lambda^b}(\square)}t_2^{l_{\lambda^a}(\square)+1}+
             \sum_{\square\in \lambda^b}t_1^{a_{\lambda^a}(\square)+1}t_2^{-l_{\lambda^b}(\square)}
       \right).$$
\noindent
 (c) The map $(\bp_r,\bq_r):M(r;n,n+1)\to M(r,n)\times M(r,n+1)$ is a closed immersion and
 the $\TT_r$-character of the fiber of the normal bundle of
 $M(r;n,n+1)$ at $\xi_{\bar{\lambda},\bar{\mu}}$ equals
     $$N_{\bar{\lambda},\bar{\mu}}=-t_1t_2+\sum_{a,b=1}^r \frac{\chi_b}{\chi_a}
       \left(\sum_{\square\in \lambda^a}t_1^{-a_{\lambda^b}(\square)}t_2^{l_{\mu^a}(\square)+1}+
             \sum_{\square\in \mu^b}t_1^{a_{\mu^a}(\square)+1}t_2^{-l_{\lambda^b}(\square)}
       \right).$$
\end{prop}

\subsection{Geometric $\ddot{U}_{q_1,q_2,q_3}(\gl_1)$-action II}
$\ $

 Let ${}'M^r$ be the direct sum of equivariant (complexified) $K$-groups:
  ${}'M^r=\bigoplus_{n}K^{\TT_r}(M(r,n))$.
 It is a module over
  $K^{\TT_r}(\textrm{pt})=\CC[\TT_r]= \CC[t_1^{\pm 1},t_2^{\pm 1},\chi_1^{\pm 1},\ldots,\chi_r^{\pm 1}]$.
 Consider a quadratic extension $\mathbb{F}_r=\Frac(K^{\TT_r}(\textrm{pt}))[t]/(t^2-t_3)$, where
  $t_3:=t_1^{-1}t_2^{-1}\in \Frac(K^{\TT_r}(\textrm{pt}))$.
 We define
  $$M^r:=\ {}'M^r\otimes_{K^{\TT_r}(\textrm{pt})} \mathbb{F}_r.$$

 It has a natural grading:
  $M^r=\bigoplus_{n}M^r_{n},\ M^r_{n}=K^{\TT_r}(M(r,n))\otimes_{K^{\TT_r}(\textrm{pt})} \mathbb{F}_r.$
 By the localization theorem, restriction to the $\TT_r$-fixed point set induces an isomorphism
  $K^{\TT_r}(M(r,n))\otimes_{K^{\TT_r}(\textrm{pt})} \mathbb{F}_r\overset{\sim}\longrightarrow
   K^{\TT_r}(M(r,n)^{\TT_r})\otimes_{K^{\TT_r}(\textrm{pt})}\mathbb{F}_r.$
 The structure sheaves $\{\bar{\lambda}\}$ of the $\TT_r$-fixed points $\xi_{\bar{\lambda}}$
 form a basis in $\bigoplus_{n}K^{\TT_r} (M(r,n)^{\TT_r})\otimes_{K^{\TT_r}(\textrm{pt})}\mathbb{F}_r$.
 Since embedding of a point $\xi_{\bar{\lambda}}$ into $M(r,|\bar{\lambda}|)$ is a proper morphism,
 the direct image in the equivariant $K$-theory is well defined,
 and we denote by $[\bar{\lambda}]\in M^r_{|\bar{\lambda}|}$
 the direct image of the structure sheaf $\{\bar{\lambda}\}$.
 The set $\{[\bar{\lambda}]\}$ forms a basis of $M^r$.

 Consider the following generating series $\ba_r(z), \bc_r(z)\in M^r(z)$:
   $$\ba_r(z):=\Lambda^{\bullet}_{-1/z}(\FF_r)=\sum_{i\geq 0}{[\Lambda^i(\FF_r)](-1/z)^i},$$
   $$\bc_r(z):=\ba_r(zt_1)\ba_r(zt_2)\ba_r(zt_3)\ba_r(zt_1^{-1})^{-1}\ba_r(zt_2^{-1})^{-1}\ba_r(zt_3^{-1})^{-1}.$$

 Finally, we define the linear operators $e_i, f_i, \psi_i, \psi_0^{-1}\ (i\in \ZZ)$ acting on $M^r$:
\begin{equation}\label{raz_r}
  e_i=\bq_{r*}(L_r^{\otimes i}\otimes \bp_r^*):\ M^r_{n}\to M^r_{n+1},
\end{equation}
\begin{equation}\label{dva_r}
  f_i=(-t)^{r-2}\chi_1^{-1}\cdots \chi_r^{-1}\cdot \bp_{r*}(L_r^{\otimes (i-r)}\otimes \bq_r^*):\ M^r_{n}\to M^r_{n-1},
\end{equation}
\begin{equation}\label{tree_r}
  \psi^{\pm}(z)=\psi_0^{\pm 1}+\sum_{j=1}^{\infty} \psi_{\pm j}z^{\mp j}:=
  \left(\prod_{a=1}^r\frac{t^{-1}-t\chi_a^{-1}/z}{1-\chi_a^{-1}/z}\bc_r(z)\right)^{\pm} \in \prod_n \End(M^r_{n})[[z^{\mp 1}]].
\end{equation}

\begin{theorem}\label{Gieseker_m}
 The operators $e_i, f_i, \psi_i, \psi_0^{-1}$, defined in (\ref{raz_r}--\ref{tree_r}),
 satisfy the relations~(T0--T6) with the parameters $q_i=t_i$.
 This endows $M^r$ with the structure of $\ddot{U}_{q_1,q_2,q_3}(\gl_1)$-representation.
\end{theorem}

\begin{rem}
 This higher rank generalization of Theorem~\ref{DIm} first appeared in~\cite{SV}.
\end{rem}

 We refer the interested reader to Appendix B, where it is explained how
 the proof of Theorem~\ref{DIm} from~\cite{FT} can be easily adapted to the general rank $r$ case.
 We conclude this section by computing explicitly the matrix coefficients of
 $e_p,f_p,\psi^\pm (z)$ in the fixed point basis of $M^r$.

\begin{lem}\label{matrix_coeff_Gis_m}
  Consider the fixed point basis $\{[\bar{\lambda}]\}$ of $M^r$.
  Define $\chi^{(a)}_k:=t_1^{\lambda^a_k-1}t_2^{k-1}\chi_a^{-1}$.

\noindent
 (a) The only nonzero matrix coefficients of the operators $e_p, f_p$ are as follows:
   $$e_{p\mid{[\bar{\lambda}-\square^{l}_j,\bar{\lambda}]}}=\frac{(\chi^{(l)}_j)^p}{1-t_1t_2}\cdot
     \prod_{a=1}^r\prod_{k=1}^\infty \frac{1-t_1t_2\chi^{(a)}_k/\chi^{(l)}_j}{1-t_1\chi^{(a)}_k/\chi^{(l)}_j},$$
   $$f_{p\mid{[\bar{\lambda}+\square^{l}_j,\bar{\lambda}]}}=(-t)^{r-2}\chi_1^{-1}\cdots \chi_r^{-1}\cdot \frac{(t_1\chi^{(l)}_j)^{p-r}}{1-t_1t_2}\cdot
     \prod_{a=1}^r\prod_{k=1}^\infty \frac{1-t_1t_2\chi^{(l)}_j/\chi^{(a)}_k}{1-t_1\chi^{(l)}_j/\chi^{(a)}_k},$$
  where $\bar{\lambda}\pm \square^{l}_j$ denotes the $r$-tuple of diagrams
   $(\lambda^1,\ldots,\lambda^{l-1}, \lambda^l\pm j, \lambda^{l+1},\ldots,\lambda^r)$.

\noindent
 (b) $\psi^\pm(z)$ is diagonal in the fixed point basis and the eigenvalue of $\psi^\pm(z)$ on $[\bar{\lambda}]$ equals
  $${\psi^\pm(z)}_{\mid_{\bar{\lambda}}}=
    \left(\prod_{a=1}^r\frac{t^{-1}-t\chi_a^{-1}/z}{1-\chi_a^{-1}/z}
           \prod_{a=1}^r\prod_{\square \in\lambda^a}\frac{(1-t_1^{-1}\chi(\square)/z)(1-t_2^{-1}\chi(\square)/z)(1-t_3^{-1}\chi(\square)/z)}
           {(1-t_1\chi(\square)/z)(1-t_2\chi(\square)/z)(1-t_3\chi(\square)/z)}\right)^\pm,$$
 where $\chi(\square^a_{i,j}):=t_1^{i-1}t_2^{j-1}\chi_a^{-1}$ for a box $\square^a_{i,j}$
 located in the $j$th row and $i$th column of $\lambda^a$.
\end{lem}

\subsection{Geometric $\ddot{Y}_{h_1,h_2,h_3}(\gl_1)$-action II}
$\ $

 Let ${}'V^r$ be the direct sum of equivariant (complexified) cohomology:
  ${}'V^r=\bigoplus_{n}\HH^\bullet_{\TT_r}(M(r,n)).$
 It is a module over
  $\HH^\bullet_{\TT_r}(\textrm{pt})=\CC[\T_r]= \CC[s_1,s_2,x_1,\ldots,x_r]$,
 where $\T_r=\mathrm{Lie}(\TT_r)$. We define
  $$V^r:=\ {}'V^r\otimes_{\HH^\bullet_{\TT_r}(\textrm{pt})} \Frac(\HH^\bullet_{\TT_r}(\textrm{pt}))=
    {}'V^r\otimes_{\CC [s_1,s_2,x_1,\ldots,x_r]} \CC (s_1,s_2,x_1,\ldots,x_r).$$

 It has a natural grading:
   $V^r=\bigoplus_{n}V^r_n,\
    V^r_n=\HH^\bullet_{\TT_r}(M(r,n))\otimes_{\HH^\bullet_{\TT_r}(\textrm{pt})} \Frac(\HH^\bullet_{\TT_r}(\textrm{pt})).$
 According to the localization theorem, restriction to
 the $\TT_r$-fixed point set induces an isomorphism
  $$\HH^\bullet_{\TT_r}(M(r,n))\otimes_{\HH^\bullet_{\TT_r}(\textrm{pt})} \Frac(\HH^\bullet_{\TT_r}(\textrm{pt}))\overset{\sim}\longrightarrow
    \HH^\bullet_{\TT_r}(M(r,n)^{\TT_r})\otimes_{\HH^\bullet_{\TT_r}(\textrm{pt})}\Frac(\HH^\bullet_{\TT_r}(\textrm{pt})).$$

 The fundamental cycles of the $\TT_r$-fixed points $\xi_{\bar{\lambda}}$ form a basis in
   $\bigoplus_{n}\HH^\bullet_{\TT_r} (M(r,n)^{\TT_r})\otimes_{\HH^\bullet_{\TT_r}(\textrm{pt})}\Frac (\HH^\bullet_{\TT_r}(\textrm{pt}))$.
 Since embedding of a point $\xi_{\bar{\lambda}}$ into $M(r,|\bar{\lambda}|)$ is a proper morphism,
 the direct image in the equivariant cohomology is well defined, and we denote by
 $[\bar{\lambda}]\in V^r_{|\bar{\lambda}|}$ the direct image of the fundamental cycle
 of the point $\xi_{\bar{\lambda}}$. The set $\{[\bar{\lambda}]\}$ forms a basis of $V^r$.

 Set $s_3:=-s_1-s_2$. Consider the following generating series $\bC_r(z)\in V^r[[z^{-1}]]$:
  $$\bC_r(z):=\left(\frac{\ch(\FF_r t_1^{-1},-z^{-1})\ch(\FF_r t_2^{-1},-z^{-1})\ch(\FF_r t_3^{-1},-z^{-1})}
    {\ch(\FF_r t_1,-z^{-1})\ch(\FF_r t_2,-z^{-1})\ch(\FF_r t_3,-z^{-1})}\right)^+.$$

 Finally, we define the linear operators $e_j, f_j, \psi_j\ (j\in \ZZ_+)$ acting on $V^r$:
\begin{equation}\tag{4$'$}\label{raz_r'}
  e_j=\bq_{r*}(c_1(L_r)^j\cdot \bp_r^*):\ V^r_n\to V^r_{n+1},
\end{equation}
\begin{equation}\tag{5$'$}\label{dva_r'}
  f_j=(-1)^{r-1}\bp_{r*}(c_1(L_r)^j\cdot \bq_r^*):\ V^r_n\to V^r_{n-1},
\end{equation}
\begin{equation}\tag{6$'$}\label{tree_r'}
  \psi(z)=1+s_1s_2s_3\sum_{j=0}^{\infty} \psi_jz^{-j-1}:=
  \left(\prod_{a=1}^r \frac{z+x_a-s_3}{z+x_a}\right)^+\cdot \bC_r(z) \in \prod_n \End(V^r_n)[[z^{-1}]].
\end{equation}

\begin{theorem}\label{Gieseker_a}
 The operators $e_j, f_j, \psi_j$, defined in (\ref{raz_r'}--\ref{tree_r'}), satisfy
 the relations~(Y0--Y6) with the parameters $h_i=s_i$.
 This endows $V^r$ with the structure of $\ddot{Y}_{h_1,h_2,h_3}(\gl_1)$-representation.
\end{theorem}

\begin{rem}
 This result is a natural ``cohomological'' analogue of Theorem~\ref{Gieseker_m}.
 It was obtained jointly with B.~Feigin in an unpublished work
 (motivated by its K-theoretical version from~\cite{SV}): we sketch our proof in Appendix B.
 The first written reference goes back to~\cite{SV3}.
\end{rem}

 We conclude by computing the action of $e_p,f_p, \psi(z)$ in the fixed point basis $\{[\bar{\lambda}]\}$ of $V^r$.

\begin{lem}\label{matrix_coeff_Gis_a}
 Define $x^{(a)}_k:=(\lambda^a_k-1)s_1+(k-1)s_2-x_a,\ \chi(\square^a_{i,j}):=(i-1)s_1+(j-1)s_2-x_a$.

\noindent
 (a) The only nonzero matrix coefficients of the operators $e_p, f_p$ are as follows:
  $$e_{p\mid{[\bar{\lambda}-\square^{l}_j,\bar{\lambda}]}}=\frac{(x^{(l)}_j)^p}{s_1+s_2}\cdot
     \prod_{a=1}^r\prod_{k=1}^\infty \frac{s_1+s_2+x^{(a)}_k-x^{(l)}_j}{s_1+x^{(a)}_k-x^{(l)}_j},$$
  $$f_{p\mid{[\bar{\lambda}+\square^{l}_j,\bar{\lambda}]}}=(-1)^{r-1}\frac{(s_1+x^{(l)}_j)^p}{s_1+s_2}\cdot
     \prod_{a=1}^r\prod_{k=1}^\infty \frac{s_1+s_2+x^{(l)}_j-x^{(a)}_k}{s_1+x^{(l)}_j-x^{(a)}_k}.$$

\noindent
 (b) $\psi(z)$ is diagonal in the fixed point basis and the eigenvalue of $\psi(z)$ on $[\bar{\lambda}]$ equals
  $${\psi(z)}_{\mid_{\bar{\lambda}}}=
    \left(\prod_{a=1}^r \frac{z+x_a-s_3}{z+x_a}\cdot
    \prod_{a=1}^r\prod_{\square\in \lambda^a}\frac{(z-\chi(\square)+s_1)(z-\chi(\square)+s_2)(z-\chi(\square)+s_3)}
    {(z-\chi(\square)-s_1)(z-\chi(\square)-s_2)(z-\chi(\square)-s_3)}\right)^+.$$
\end{lem}

\begin{cor}
 We have
  $$\psi(z)_{\mid_{\bar{\lambda}}}=
   1-\frac{rs_3}{z}+\frac{s_3\sum x_a+\binom{r}{2}s_3^2}{z^2}+
   \frac{2\sigma_3|\bar{\lambda}|-s_3\sum x_a^2-(r-1)s_3^2\sum x_a-\binom{r}{3}s_3^3}{z^3}+o(z^{-3}).$$
\end{cor}


\section{Some representations of $\ddot{U}_{q_1,q_2,q_3}(\gl_1)$ and $\ddot{Y}_{h_1,h_2,h_3}(\gl_1)$}

 In this section, we recall several families of
 $\ddot{U}_{q_1,q_2,q_3}(\gl_1)$-representations from~\cite{FFJMM, FFJMM2}
 and introduce their analogues for the case of $\ddot{Y}_{h_1,h_2,h_3}(\gl_1)$.
 We also establish the relation to the representations from Sections 2--3
 and introduce the appropriate categories $\mathcal{O}$.

\subsection{Vector representations $V(u)$ and $^aV(u)$}
$\ $

 The main `building block' of all known $\ddot{U}_{q_1,q_2,q_3}(\gl_1)$-representations
 is the family of \emph{vector representations} $\{V(u)\}_{u\in \CC^*}$,
 whose basis is parametrized by $\ZZ$ (see~\cite[Proposition 3.1]{FFJMM}).

\begin{prop}[Vector representation of $\ddot{U}_{q_1,q_2,q_3}(\gl_1)$]\label{Vector_m}
 For $u\in \CC^*$, let $V(u)$ be a $\CC$-vector space with the basis $\{[u]_j\}_{j\in \ZZ}$.
 The following formulas define a $\ddot{U}_{q_1,q_2,q_3}(\gl_1)$-action on $V(u)$:
  $$e(z)[u]_i=(1-q_1)^{-1}\delta(q_1^iu/z)\cdot[u]_{i+1},$$
  $$f(z)[u]_i=(q_1^{-1}-1)^{-1}\delta(q_1^{i-1}u/z)\cdot[u]_{i-1},$$
  $$\psi^\pm(z)[u]_i=\left(\frac{(z-q_1^iq_2u)(z-q_1^iq_3u)}{(z-q_1^iu)(z-q_1^{i-1}u)}\right)^\pm\cdot[u]_i.$$
\end{prop}

 Define $\delta^+(w):=1+w+w^2+\ldots=(\frac{1}{1-w})^+$.
 Our next result provides an analogous construction of \emph{vector representations}
 $\{^aV(u)\}_{u\in \CC}$ for the case of $\ddot{Y}_{h_1,h_2,h_3}(\gl_1)$
 (the proof is straightforward and is left to the interested reader):

\begin{prop}[Vector representation of $\ddot{Y}_{h_1,h_2,h_3}(\gl_1)$]
 For $u\in \CC$, let $^aV(u)$ be a $\CC$-vector space with the basis $\{[u]_j\}_{j\in \ZZ}$.
 The following formulas define a $\ddot{Y}_{h_1,h_2,h_3}(\gl_1)$-action on $^aV(u)$:
  $$e(z)[u]_i=\frac{1}{h_1z}\delta^+((ih_1+u)/z)[u]_{i+1}=\left(\frac{1}{h_1(z-ih_1-u)}\right)^+\cdot[u]_{i+1},$$
  $$f(z)[u]_i=-\frac{1}{h_1z}\delta^+(((i-1)h_1+u)/z)[u]_{i-1}=\left(\frac{-1}{h_1(z-(i-1)h_1-u)}\right)^+\cdot[u]_{i-1},$$
  $$\psi(z)[u]_i=\left(\frac{(z-(ih_1+h_2+u))(z-(ih_1+h_3+u))}{(z-(ih_1+u))(z-((i-1)h_1+u))}\right)^+\cdot [u]_i.$$
\end{prop}

\subsection{Fock representations $F(u)$ and $^aF(u)$}
$\ $

 A more interesting family of $\ddot{U}_{q_1,q_2,q_3}(\gl_1)$-representations,
 whose basis is labeled by Young diagrams $\{\lambda\}$,
 was established in~\cite[Theorem 4.3]{FFJMM}.

\begin{prop}[Fock representation of $\ddot{U}_{q_1,q_2,q_3}(\gl_1)$]
 For $u\in \CC^*$, let $F(u)$ be a $\CC$-vector space with the basis $\{|\lambda\rangle\}$.
 The following formulas define a $\ddot{U}_{q_1,q_2,q_3}(\gl_1)$-action on $F(u)$:
  $$e(z)|\lambda\rangle=
    \sum_{i\geq 1}\prod_{j=1}^{i-1}
    \frac{(1-q_1^{\lambda_i-\lambda_j}q_2^{i-j-1})(1-q_1^{\lambda_i-\lambda_j+1}q_2^{i-j+1})}
         {(1-q_1^{\lambda_i-\lambda_j}q_2^{i-j})(1-q_1^{\lambda_i-\lambda_j+1}q_2^{i-j})}
    \cdot \frac{\delta(q_1^{\lambda_i}q_2^{i-1}u/z)}{1-q_1}\cdot|\lambda+i\rangle,$$
  $$f(z)|\lambda\rangle=
    q^{-1}\sum_{i\geq 1}\prod_{j=i+1}^{\infty}
    \frac{(1-q_1^{\lambda_j-\lambda_i+1}q_2^{j-i+1})(1-q_1^{\lambda_{j}-\lambda_i}q_2^{j-i-1})}
         {(1-q_1^{\lambda_{j}-\lambda_i+1}q_2^{j-i})(1-q_1^{\lambda_j-\lambda_i}q_2^{j-i})}
    \cdot\frac{\delta(q_1^{\lambda_i-1}q_2^{i-1}u/z)}{q_1^{-1}-1}\cdot|\lambda-i\rangle,$$
  $$\psi^{\pm}(z)|\lambda\rangle=
    q^{-1}\cdot \left(\frac{z-q_1^{\lambda_1-1}q_2^{-1}u}{z-q_1^{\lambda_1}u}
    \prod_{i=1}^{\infty}\frac{(z-q_1^{\lambda_i}q_2^iu)(z-q_1^{\lambda_{i+1}-1}q_2^{i-1}u)}
    {(z-q_1^{\lambda_{i+1}}q_2^iu)(z-q_1^{\lambda_i-1}q_2^{i-1}u)}\right)^{\pm}\cdot|\lambda\rangle.$$
\end{prop}

\begin{rem}\label{coproduct_m}
 The Fock module $F(u)$ was originally constructed from $V(u)$ by using the semi-infinite wedge construction
 and the formal coproduct structure on $\ddot{U}_{q_1,q_2,q_3}(\gl_1)$ defined by
  $$\Delta(e(z))=e(z)\otimes 1+\psi^-(z)\otimes e(z),\
    \Delta(f(z))=f(z)\otimes \psi^+(z)+1\otimes f(z),\
    \Delta(\psi^\pm(z))=\psi^\pm(z)\otimes \psi^\pm(z).$$
\end{rem}

\begin{rem}\label{Fock_vs_K-theory}
 (a) According to~\cite[Corollary 4.5]{FFJMM}, there exist constants $\{c_\lambda\}$ such that
  the map $[\lambda]\mapsto c_\lambda|\lambda\rangle$ induces an
  isomorphism $M\iso F(1)$ of $\ddot{U}_{q_1,q_2,q_3}(\gl_1)$-representations,
  where $M$ is the representation from Theorem~\ref{DIm} and $q_i=t_i$.

\noindent
 (b) For $u\in \CC^*$, let $\phi_u$ be the \emph{shift automorphism}
 of $\ddot{U}_{q_1,q_2,q_3}(\gl_1)$ defined on the generators by
  $$\psi_0^{-1}\mapsto \psi_0^{-1},\ \
    \psi_i\mapsto u^{-i}\cdot \psi_i,\ \
    e_i\mapsto u^{-i}\cdot e_i, \ \
    f_i\mapsto u^{-i}\cdot f_i\ \  \mathrm{for}\ i\in \ZZ.$$
 Then the modules $F(u)$ and $V(u)$ are obtained from $F(1)$ and $V(1)$ via
 a twist by $\phi_{1/u}$.\footnote{For $\sigma\in \mathrm{Aut}(A)$, the twist of an
 $A$-representation $\rho:A\to \End(V)$ is $\rho^\sigma: A\to \End(V)$ with $\rho^\sigma(a)=\rho(\sigma(a))$.}
\end{rem}

 Let us analogously define \emph{Fock representations} $\{^aF(u)\}_{u\in \CC}$ for the case of $\ddot{Y}_{h_1,h_2,h_3}(\gl_1)$.

\begin{prop}[Fock representation of $\ddot{Y}_{h_1,h_2,h_3}(\gl_1)$]\label{Fock_a}
 For $u\in \CC$, let $^aF(u)$ be a $\CC$-vector space with the basis $\{|\lambda\rangle\}$.
 The following formulas define a $\ddot{Y}_{h_1,h_2,h_3}(\gl_1)$-action on $^aF(u)$:
  $$e(z)|\lambda\rangle=
    \frac{1}{h_1z}\sum_{i\geq 1}\prod_{j=1}^{i-1}
    \frac{((\lambda_i-\lambda_j)h_1+(i-j-1)h_2)((\lambda_i-\lambda_j+1)h_1+(i-j+1)h_2)}
         {((\lambda_i-\lambda_j)h_1+(i-j)h_2)((\lambda_i-\lambda_j+1)h_1+(i-j)h_2)}\times$$
    $$\delta^+\left(\frac{\lambda_ih_1+(i-1)h_2+u}{z}\right)\cdot|\lambda+i\rangle,$$
  $$f(z)|\lambda\rangle=
    -\frac{1}{h_1z}\sum_{i\geq 1}\prod_{j=i+1}^{\infty}
    \frac{((\lambda_j-\lambda_i+1)h_1+(j-i+1)h_2)((\lambda_{j+1}-\lambda_i)h_1+(j-i)h_2)}
         {((\lambda_{j+1}-\lambda_i+1)h_1+(j-i+1)h_2)((\lambda_j-\lambda_i)h_1+(j-i)h_2)}\times$$
    $$\frac{(\lambda_{i+1}-\lambda_i)h_1}{(\lambda_{i+1}-\lambda_i+1)h_1+h_2}
      \delta^+\left(\frac{(\lambda_i-1)h_1+(i-1)h_2+u}{z}\right)\cdot|\lambda-i\rangle,$$
  $$\psi(z)|\lambda\rangle=
    \left(\prod_{i=1}^{\infty}\frac{(z-(\lambda_ih_1+ih_2+u))(z-((\lambda_{i+1}-1)h_1+(i-1)h_2+u))}
    {(z-(\lambda_{i+1}h_1+ih_2+u))(z-((\lambda_i-1)h_1+(i-1)h_2+u))}\right)^+\times$$
    $$\left(\frac{z-((\lambda_1-1)h_1-h_2+u)}{z-(\lambda_1h_1+u)}\right)^+\cdot |\lambda\rangle.$$
\end{prop}

 The proof of this proposition follows from the following lemma:

\begin{lem}\label{iso}
 (a) For $u\in \CC$, there exists the \emph{shift automorphism} $\phi^a_u$ of $\ddot{Y}_{h_1,h_2,h_3}(\gl_1)$ such that
      $$\phi^a_u:\ e(z)\mapsto e(z+u), \ f(z)\mapsto f(z+u),\ \psi(z)\mapsto \psi(z+u).$$

\noindent
 (b) The Fock representation $^aF(u)$ is obtained from $^aF(0)$ via a twist by $\phi^a_{-u}$.

\noindent
 (c) There exist constants $\{c^a_\lambda\}$ such that the map
   $[\lambda]\mapsto c^a_\lambda|\lambda\rangle$
 induces an isomorphism $V\iso {^aF(0)}$ of $\ddot{Y}_{h_1,h_2,h_3}(\gl_1)$-representations,
 where $V$ is the representation from Theorem~\ref{DIa} and $h_i=s_i$.
\end{lem}

\begin{proof}
$\ $

 Parts (a) and (b) are straightforward.

 We define $c^a_\lambda$ by the following formula:
  $$c^a_\lambda=\prod_{i\geq 1} \prod_{p=0}^{\lambda_i-1}(-ph_1+h_2)\cdot
     \prod_{i\geq 2} \prod_{j=1}^{i-1} \prod_{p=1}^{\lambda_i}\frac{(p-\lambda_j)h_1+(i-j+1)h_2}{(p-\lambda_j)h_1+(i-j)h_2}.$$
 It is a routine verification to check that the map $[\lambda]\mapsto c^a_\lambda|\lambda\rangle$
 intertwines the formulas for the matrix coefficients of $e_k, f_k, \psi_k\ (k\in \ZZ_+)$
 from Lemma~\ref{matrix_elements_a} and Proposition~\ref{Fock_a}.
\end{proof}

\begin{defn}
 We say that a $\ddot{Y}_{h_1,h_2,h_3}(\gl_1)$-representation $U$ has a \emph{central charge} $(c_0,c_1)\in \CC^2$ if
 $\psi_0$ acts on $U$ via $c_0\cdot \mathrm{Id}_U$ and $\psi_1$ acts on $U$ via $c_1\cdot \mathrm{Id}_U$.
\end{defn}

 Thus $^aV(u)$ has a central charge $\left(0,\frac{1}{h_1}\right)$,
 while $^aF(u)$ has a central charge $\left(-\frac{1}{h_1h_2},-\frac{u}{h_1h_2}\right)$.

\subsection{Tensor products $\otimes F(u_i)$}
$\ $

 In this section, we relate the representation $M^r$ from Theorem~\ref{Gieseker_m}
 to the Fock modules $F(u)$. Let $\Delta$ be the \emph{formal} comultiplication on
 $\ddot{U}_{q_1,q_2,q_3}(\gl_1)$ from Remark~\ref{coproduct_m}.
 This is not a comultiplication in the usual sense, since $\Delta(e_i)$ and $\Delta(f_i)$
 contain infinite sums. However, for all modules of interest
 ``with general spectral parameters'' these formulas make sense
 (this is explained below in the particular case of Fock modules).
 The following theorem can be viewed as a higher rank generalization of Remark~\ref{Fock_vs_K-theory}.

\begin{thm}\label{Tensor_Fock_m}
 There exists a unique collection of constants $\{c_{\bar{\lambda}}\}$ from
 the field $\mathbb{F}_r$ with $c_{\bar{\emptyset}}=1$ such that the map
   $[\bar{\lambda}]=[(\lambda^1,\ldots,\lambda^r)]\mapsto c_{\bar{\lambda}}\cdot |\lambda^1\rangle\otimes\cdots\otimes |\lambda^r\rangle$
 induces an isomorphism $M^r\iso F(\chi_1^{-1})\otimes \cdots \otimes F(\chi_r^{-1})$
 of $\ddot{U}_{q_1,q_2,q_3}(\gl_1)$-representations with $q_i=t_i$.
\end{thm}

 Let us first make sense of the tensor product
 $F(\chi_1)\otimes F(\chi_2)$ (the case $r>2$ is completely analogous).
 The action of $e(z)$ and $f(z)$ in the basis $\{|\lambda\rangle\}$ of $F(\chi)$ can be written as follows:
   $$e(z)|\lambda\rangle=\sum_\square  a_{\lambda,\square}\delta\left(\frac{\chi(\square)}{z}\right)|\lambda+\square\rangle,\
     f(z)|\lambda\rangle=\sum_\square  b_{\lambda,\square}\delta\left(\frac{\chi(\square)}{z}\right)|\lambda-\square\rangle,$$
 where $a_{\lambda,\square}, b_{\lambda,\square}\in \mathbb{F}_r$,
 the first sum is over $\square\notin \lambda$ such that $\lambda+\square$ is a Young diagram,
 while the second sum is over $\square\in \lambda$ such that $\lambda-\square$ is a Young diagram.

 According to the comultiplication formula, we have
   $$\Delta(e(z))(|\lambda^1\rangle \otimes |\lambda^2\rangle)=
     e(z)(|\lambda^1\rangle)\otimes |\lambda^2\rangle+\psi^-(z)(|\lambda^1\rangle)\otimes e(z)(|\lambda^2\rangle).$$
 The first summand is well defined.
 To make sense of the second summand, we use the formula
\begin{equation}\tag{*}
   g(z)\delta(a/z)=g(a)\delta(a/z)\ \mathrm{for\ any\ rational\ function}\ g(z).
\end{equation}
 Recall that $\psi^{\pm}(z)(|\lambda\rangle)=\gamma_{\lambda}(z)^\pm\cdot |\lambda\rangle$,
 where $\gamma_{\lambda}(z)$ is a rational function in $z$ depending on $\lambda$.
 Combining this with (*), we find
   $$\psi^-(z)(|\lambda^1\rangle)\otimes e(z)(|\lambda^2\rangle)=
     \sum_\square a_{\lambda^2,\square}\gamma_{\lambda^1}(\chi(\square))
     \delta\left(\frac{\chi(\square)}{z}\right)\cdot|\lambda^1\rangle\otimes |\lambda^2+\square\rangle.$$
 The action of $f_i$ on $F(\chi_1)\otimes F(\chi_2)$ is defined analogously.
 Finally, the formula $\Delta(\psi^{\pm}(z))=\psi^{\pm}(z)\otimes \psi^{\pm}(z)$
 provides a well defined action of $\psi_i$ on $F(\chi_1)\otimes F(\chi_2)$.

\begin{proof}[Proof of Theorem~\ref{Tensor_Fock_m}]$\ $

 Due to Remark~\ref{Fock_vs_K-theory}, we can identify $F(\chi_k^{-1})\simeq M^{\phi_{\chi_k}}$,
 the twist of $M$ by $\phi_{\chi_k}$. For any $r$-tuple of diagrams
 $\bar{\lambda}=(\lambda^1,\ldots,\lambda^r)$, Lemma~\ref{matrix_coeff_Gis_m}(b)
 implies that the eigenvalue of $\psi^\pm(z)$ on $[\bar{\lambda}]\in M^r$
 equals the eigenvalue of $\psi^\pm(z)$ on
  $|\lambda^1\rangle \otimes \cdots \otimes |\lambda^r\rangle\in F(\chi_1^{-1})\otimes \cdots\otimes F(\chi_r^{-1})$.
 Hence, the map
  $[\bar{\lambda}]\mapsto c_{\bar{\lambda}}\cdot |\lambda^1\rangle\otimes\cdots\otimes |\lambda^r\rangle$
 intertwines actions of $\psi_i$ and $\psi_0^{-1}$ for any constants $c_{\bar{\lambda}}$.

 Consider constants $c_{\bar{\lambda}}$ such that $c_{\bar{\emptyset}}=1$ and
 $c_{\bar{\lambda}+\square^k_i}/c_{\bar{\lambda}}=d_{\bar{\lambda},\square^k_i}$ with
\begin{equation}\label{compatibility}
    d_{\bar{\lambda},\square^k_i}:=
    \frac{t^{k-1}t_2(-t_1\chi^{(k)}_i)^{-r}}{\chi_1\cdots \chi_r}\cdot
    \prod_{(l,j)\prec (k,i)}  \frac{1-t_1t_2\chi^{(k)}_i/\chi^{(l)}_j}{1-t_1\chi^{(k)}_i/\chi^{(l)}_j}\cdot
    \prod_{(l,j)\succ (k,i)}  \frac{1-\chi^{(k)}_i/\chi^{(l)}_j}{1-t_2^{-1}\chi^{(k)}_i/\chi^{(l)}_j},
\end{equation}
 where $\chi^{(m)}_p=t_1^{\lambda^m_p-1}t_2^{p-1}\chi_m^{-1}$ as before and
 $(l,j)\prec (k,i)$ if and only if $l<k$ or $l=k\ \mathrm{and}\ j<i$.
 Since $\chi^{(m)}_{p+1}=t_2\chi^{(m)}_p$ for $p\geq |\lambda^m|$, it is easy to see
 that the infinite products in (\ref{compatibility}) are actually finite.
 The existence of $c_{\bar{\lambda}}$ satisfying
  $c_{\bar{\lambda}+\square^k_i}/c_{\bar{\lambda}}=d_{\bar{\lambda},\square^k_i}$
 is equivalent to
   $$d_{\bar{\lambda}+\square^k_i, \square^l_j}\cdot d_{\bar{\lambda},\square^k_i}=
     d_{\bar{\lambda}+\square^l_j, \square^k_i}\cdot d_{\bar{\lambda},\square^l_j}\
     \mathrm{for\ all\ possible}\ \square^k_i,\square^l_j.$$
 The proof of these equalities as well as verification that the map
  $[\bar{\lambda}]\mapsto c_{\bar{\lambda}}\cdot |\lambda^1\rangle\otimes\cdots\otimes |\lambda^r\rangle$
 intertwines actions of $e_i$ and $f_i$ are left to the interested reader.
 The result follows.
\end{proof}

\subsection{Tensor products $\otimes\ ^aF(u_i)$}
$\ $

 In this section, we relate the representation $V^r$ from Theorem~\ref{Gieseker_a}
 to the Fock modules $^aF(u)$. First we need to define the tensor product
 $W_1\otimes W_2$ of $\ddot{Y}_{h_1,h_2,h_3}(\gl_1)$-representations $W_1, W_2$.

\begin{defn}
 A $\ddot{Y}_{h_1,h_2,h_3}(\gl_1)$-representation $W$ is called \emph{admissible}
 if there exists a basis $\{w_\alpha\}_{\alpha\in I}$ of $W$ such that

$\circ$
   $e(z)(w_\alpha)=\sum_{\alpha'\in I} \frac{c_{\alpha,\alpha'}}{z}\delta^+(\lambda_{\alpha,\alpha'}/z)w_{\alpha'},
    f(z)(w_\alpha)=\sum_{\alpha''\in I} \frac{d_{\alpha,\alpha''}}{z}\delta^+(\lambda_{\alpha'',\alpha}/z)w_{\alpha''}$
 for some $c_{\alpha,\alpha'}, d_{\alpha,\alpha''}, \lambda_{\alpha,\alpha'}\in \CC$,
 so that both sums are finite for every $\alpha\in I$.

$\circ$
   $\psi(z)(w_\alpha)=\gamma_W(\alpha,z)^+ \cdot w_\alpha$ for a rational function
   $\gamma_W(\alpha,z)$ defined by
   $$\gamma_W(\alpha,z)=
     1+\sigma_3\sum_{\alpha''\in I}\frac{d_{\alpha,\alpha''}c_{\alpha'',\alpha}}{z-\lambda_{\alpha'',\alpha}}-
        \sigma_3\sum_{\alpha'\in I}\frac{c_{\alpha,\alpha'}d_{\alpha',\alpha}}{z-\lambda_{\alpha,\alpha'}}.$$

$\circ$
 For any $\alpha_1\ne \alpha_2\in I$, there is a bijection between
  $\{\alpha'\in I|c_{\alpha_1,\alpha'}d_{\alpha',\alpha_2}\ne 0\}$
 and
  $\{\alpha''\in I|d_{\alpha_1,\alpha''}c_{\alpha'',\alpha_2}\ne 0\}$
 such that
  $\lambda_{\alpha_1,\alpha'}=\lambda_{\alpha'',\alpha_2}$ and
  $\lambda_{\alpha_2,\alpha'}=\lambda_{\alpha'',\alpha_1}$.
\end{defn}

\begin{ex}\label{admissible}
  The modules $^aV(u)$ and $^aF(u)$ are admissible.
\end{ex}

 Let $W_1$ and $W_2$ be admissible $\ddot{Y}_{h_1,h_2,h_3}(\gl_1)$-representations with the
 corresponding bases $\{w^1_\alpha\}_{\alpha\in I}$ and $\{w^2_\beta\}_{\beta\in J}$.
 Define the operator series $\psi(z),e(z),f(z)\in \mathrm{End}(W_1\otimes W_2)[[z^{-1}]]$ by
  $$\psi(z)(w^1_\alpha\otimes w^2_\beta)=\psi(z)(w^1_\alpha)\otimes \psi(z)(w^2_\beta),$$
  $$e(z)(w^1_\alpha\otimes w^2_\beta)=
    \sum_{\alpha'\in I}\frac{c^1_{\alpha,\alpha'}}{z}\delta^+(\lambda^1_{\alpha,\alpha'}/z)w^1_{\alpha'}\otimes w^2_\beta+
    \sum_{\beta'\in J}\frac{c^2_{\beta,\beta'}\gamma_{W_1}(\alpha, \lambda^2_{\beta,\beta'})}{z}\delta^+(\lambda^2_{\beta,\beta'}/z)w^1_\alpha\otimes w^2_{\beta'},$$
  $$f(z)(w^1_\alpha\otimes w^2_\beta)=
    \sum_{\beta'\in J}\frac{d^2_{\beta,\beta''}}{z}\delta^+(\lambda^2_{\beta'',\beta}/z)w^1_{\alpha}\otimes w^2_{\beta''}+
    \sum_{\alpha''\in I}\frac{d^1_{\alpha,\alpha''}\gamma_{W_2}(\beta, \lambda^1_{\alpha'',\alpha})}{z}\delta^+(\lambda^1_{\alpha'',\alpha}/z)w^1_{\alpha''}\otimes w^2_{\beta}.$$

\begin{rem}\label{well-defined}
 These formulas are well defined only if
 $\gamma_{W_1}(\alpha,z)$ is regular at $\{\lambda^2_{\beta,\beta'}|c^2_{\beta,\beta'}\ne 0\}$ and
 $\gamma_{W_2}(\beta,z)$ is regular at $\{\lambda^1_{\alpha'',\alpha}|d^1_{\alpha,\alpha''}\ne 0\}$
 for any $\alpha\in I, \beta\in J$.
\end{rem}

 The following is straightforward:

\begin{lem}\label{tensor}
 If $W_1$ and $W_2$ are admissible $\ddot{Y}_{h_1,h_2,h_3}(\gl_1)$-representations
 and the assumptions of Remark~\ref{well-defined} hold, then the above formulas
 define an action of $\ddot{Y}_{h_1,h_2,h_3}(\gl_1)$ on $W_1\otimes W_2$.
\end{lem}

\begin{rem}
 We refer the interested reader to~\cite[Section 1]{TB} for an alternative viewpoint.
\end{rem}

 It might be still possible to define an action of $\ddot{Y}_{h_1,h_2,h_3}(\gl_1)$ on a
 submodule or a factor-module of $W_1\otimes W_2$ if the assumptions of Remark~\ref{well-defined} fail,
 due to our next result:

\begin{lem}\label{subspace}
 Let $S$ be a subset of $I\times J$ such that
 $e(z)(w^1_\alpha\otimes w^2_\beta), f(z)(w^1_\alpha\otimes w^2_\beta)$ are well defined
 (in the sense of Remark~\ref{well-defined}) for any $(\alpha,\beta)\in S$ and satisfy one of the following conditions:

\noindent
 (a) For any $(\alpha,\beta)\in S, (\alpha',\beta')\notin S$, $w^1_{\alpha'}\otimes w^2_{\beta'}$ doesn't appear in
  $e(z)(w^1_\alpha\otimes w^2_\beta), f(z)(w^1_\alpha\otimes w^2_\beta)$.

\noindent
 (b) For any $(\alpha,\beta)\in S, (\alpha',\beta')\notin S$, $w^1_{\alpha}\otimes w^2_{\beta}$ doesn't appear in
  $e(z)(w^1_{\alpha'}\otimes w^2_{\beta'}), f(z)(w^1_{\alpha'}\otimes w^2_{\beta'})$.

\noindent
 Then the above formulas define a $\ddot{Y}_{h_1,h_2,h_3}(\gl_1)$-action on
 $\spa\{w^1_\alpha\otimes w^2_\beta\}_{(\alpha,\beta)\in S}$.
\end{lem}

 The key result of this section is proved completely analogously to Theorem~\ref{Tensor_Fock_m}:

\begin{thm}\label{Tensor_Fock_a}
 There exists a unique collection of constants $\{c^a_{\bar{\lambda}}\}$ from
 $\CC(s_1,s_2,x_1,\ldots,x_r)$ with $c^a_{\bar{\emptyset}}=1$ such that the map
   $[\bar{\lambda}]=[(\lambda^1,\ldots,\lambda^r)]\mapsto c^a_{\bar{\lambda}}\cdot |\lambda^1\rangle\otimes\cdots\otimes |\lambda^r\rangle$
 induces an isomorphism $V^r\iso\ ^aF(-x_1)\otimes \cdots \otimes {^aF(-x_r)}$
 of $\ddot{Y}_{h_1,h_2,h_3}(\gl_1)$-representations with $h_i=s_i$.
\end{thm}

\begin{rem}
 As $^aF(0)\simeq V$ by Lemma~\ref{iso}(c), we get
   $V^r\simeq V^{\phi^a_{x_1}}\otimes\cdots\otimes V^{\phi^a_{x_r}}$.
 In other words, the representation of $\ddot{Y}_{h_1,h_2,h_3}(\gl_1)$
 on the sum of equivariant cohomology groups of $M(r,n)$ is a tensor product of
 $r$ copies of such representations for $M(1,m)$.
\end{rem}

\subsection{Other families of representations}
$\ $

 We recall some other series of $\ddot{U}_{q_1,q_2,q_3}(\gl_1)$-representations
 from~\cite{FFJMM, FFJMM2}. All of them admit straightforward generalizations
 for the case of $\ddot{Y}_{h_1,h_2,h_3}(\gl_1)$. These have the same bases,
 while the matrix coefficients of $e(z),f(z),\psi(z)$ in these bases are modified as follows:
   $$1-q_1^iq_2^jq_3^ku/z\rightsquigarrow z-ih_1-jh_2-kh_3-u,\
     \delta(q_1^iq_2^jq_3^ku/z)\rightsquigarrow \pm \frac{1}{z}\delta^+((ih_1+jh_2+kh_3+u)/z),$$
 where the latter sign is ``$+$'' for $e(z)$ and ``$-$'' for $f(z)$.

\medskip
\noindent
 $\bullet$ \emph{Representations $W^N(u)$.}

 Consider the tensor product
   $V^N(u):=V(u)\otimes V(uq_3^{-1})\otimes V(uq_3^{-2})\otimes \cdots\otimes V(uq_3^{1-N})$.
 Define
   $\PP^N:=\{\lambda=(\lambda_1,\ldots,\lambda_N)\in \ZZ^N|\lambda_1\geq \cdots \geq \lambda_N\},
    \PP^{N,+}:=\{\lambda\in \PP^N|\lambda_N\geq 0\}$.
 Let $W^N(u)\subset V^N(u)$ be the subspace spanned by
   $[u]_\lambda:=[u]_{\lambda_1}\otimes [uq_3^{-1}]_{\lambda_2-1} \otimes \cdots \otimes [uq_3^{1-N}]_{\lambda_N-N+1}$
 for $\lambda\in \PP^N$.

 According to~\cite[Lemma 3.7]{FFJMM},
 $W^N(u)$ is a $\ddot{U}_{q_1,q_2,q_3}(\gl_1)$-submodule of $V^N(u)$.
 The subspace $W^{N,+}(u)\subset W^N(u)$ corresponding to $\PP^{N,+}$ is not a submodule.
 However, its limit as $N\to \infty$ is well defined (after a renormalization)
 and coincides with the Fock module $F(u)$.

\medskip
\noindent
 $\bullet$  \emph{Representations $G_{\ba}^{k,r}$.}

 Let $q_1,q_2$ be in the $(k,r)$-\emph{resonance relation}: $q_1^aq_2^b=1$
 iff $a=(1-r)c, b=(k+1)c$ for some $c\in \ZZ$ (assume $k\geq 1, r\geq 2$).
 In this case the action of $\ddot{U}_{q_1,q_2,q_3}(\gl_1)$ on $W^N(u)$ is ill defined.
 Consider the set of \emph{$(k,r)$-admissible} partitions
   $S^{k,r,N}:=\{\lambda\in \PP^N| \lambda_i-\lambda_{i+k}\geq r\ \forall\ i\leq N-k\}$.

 Let $W^{k,r,N}(u)$ be the subspace of $W^N(u)$ spanned by the vectors
 $[u]_\lambda$ for $\lambda\in S^{k,r,N}$. According to~\cite[Lemma 6.2]{FFJMM},
 the comultiplication rule makes $W^{k,r,N}(u)$ into a $\ddot{U}_{q_1,q_2,q_3}(\gl_1)$-module.
 We think of it as ``a submodule of $W^N(u)$ or $V^N(u)$'' even
 though none of them has a $\ddot{U}_{q_1,q_2,q_3}(\gl_1)$-module structure
(we use an analogue of Lemma~\ref{subspace} there).

 Let us fix a sequence of non-negative integers $\ba=(a_1,\ldots,a_k)$ satisfying $\sum_{i=1}^k a_i=r$.
 Define
  $\PP_{\ba}^{k,r}:=\{(\lambda_1\geq \lambda_2\geq \cdots)|
                       \lambda_j-\lambda_{j+k}\geq r\ \forall\ j\geq 1\ \mathrm{and}\ \lambda_j=\lambda_j^0\ \forall\ j\gg0 \}$,
 where we set $\lambda^{0}_{\nu k+i+1}:=-\nu r-\sum_{j=1}^i a_j$ for $\nu\in \ZZ_+, 0\leq i\leq k-1$.
 One can define an action of $\ddot{U}_{q_1,q_2,q_3}(\gl_1)$ on the $N\to \infty$ limit of $W^{k,r,N}(u)$,
 yielding an action of $\ddot{U}_{q_1,q_2,q_3}(\gl_1)$ on the space
 $G_{\ba}^{k,r}$ whose basis is labeled by $\lambda\in \PP_{\ba}^{k,r}$, see~\cite[Theorem 6.5]{FFJMM}.

\medskip
\noindent
 $\bullet$ \emph{Representations $\M_{\ba,\bb}(u)$.}

 Consider the tensor product $F(u_1)\otimes \cdots \otimes F(u_n)$.
 It is well defined as a $\ddot{U}_{q_1,q_2,q_3}(\gl_1)$-module if
 $q_1,q_2,u_1,\ldots,u_n$ are \emph{generic}, that is,
   $q_1^aq_2^bu_1^{c_1}\cdots u_n^{c_n}=1 \Longleftrightarrow a=b=c_1=\ldots=c_n=0$.

 Consider the \emph{resonance case}
 $u_i=u_{i+1}q_1^{a_i+1}q_2^{b_i+1}$ for $1\leq i\leq n-1$ and some $a_i,b_i\in \ZZ_+$.
 Set $u:=u_1$. Let $\M_{\ba,\bb}(u)\subset F(u_1)\otimes \cdots\otimes F(u_n)$ be the subspace spanned by
 $|\lambda^1,\ldots,\lambda^n\rangle:=[u_1]_{\lambda^1}\otimes\cdots\otimes [u_n]_{\lambda^n}$,
 where Young diagrams $\lambda^1,\ldots,\lambda^n$ satisfy
 $\lambda^i_s\geq \lambda^{i+1}_{s+b_i}-a_i$ for $1\leq i\leq n-1, s\in \NN$.
 According to~\cite[Proposition 3.3]{FFJMM2}, the comultiplication rule makes $\M_{\ba,\bb}(u)$
 into a $\ddot{U}_{q_1,q_2,q_3}(\gl_1)$-module for \emph{generic} $q_1,q_2,u$.

\medskip
\noindent
 $\bullet$ \emph{Representations $\M^{p',p}_{\ba,\bb}(u)$.}

 Assume further that $q_1,q_2$ are not generic: there exist
 $p\ne p'\in \NN$ such that $q_1^aq_2^b=1$ iff $a=p'c, b=pc$ for some $c\in \ZZ$.
 We require that $a_n:=p'-1-\sum_{i=1}^{n-1}(a_i+1), b_n:=p-1-\sum_{i=1}^{n-1}(b_i+1)$ are non-negative.
 In this case, the action of $\ddot{U}_{q_1,q_2,q_3}(\gl_1)$ on $\M_{\ba,\bb}(u)$ is ill defined.

 Consider a subspace $\M^{p',p}_{\ba,\bb}(u)\subset F(u_1)\otimes \cdots\otimes F(u_n)$
 spanned by $|\lambda^1,\ldots,\lambda^n\rangle$ satisfying the same conditions
 $\lambda^i_s\geq \lambda^{i+1}_{s+b_i}-a_i$ but for $1\leq i\leq n$, where we set $\lambda^{n+1}:=\lambda^1$.
 The comultiplication rule makes it into a $\ddot{U}_{q_1,q_2,q_3}(\gl_1)$-module,
 due to~\cite[Proposition 3.7]{FFJMM2}.
 We think of $\M^{p',p}_{\ba,\bb}(u)$ as ``a subquotient of $F(u_1)\otimes \cdots\otimes  F(u_n)$''.
 Their characters coincide with the characters from the $\mathcal{W}_n$-minimal series,
 according to the main result of~\cite{FFJMM2}.

\subsection{Categories $\Oo$}
$\ $

  We conclude this section by introducing the appropriate categories $\Oo$ for both algebras.

\medskip
\noindent
  $\bullet$ Category $\Oo$ for $\ddot{U}_{q_1,q_2,q_3}(\gl_1)$.

\noindent
  We equip $\ddot{U}_{q_1,q_2,q_3}(\gl_1)$ with the $\ZZ$-grading by assigning
    $\deg(e_i)=-1, \deg(f_i)=1, \deg(\psi_i)=0$.

\begin{defn}\label{Cat O for toroidal}
  We say that a $\ZZ$-graded $\ddot{U}_{q_1,q_2,q_3}(\gl_1)$-module $L$ is in the category $\Oo$ if

\noindent
 (i) for any $v\in L$ there exists $N\in \ZZ$ such that $\ddot{U}_{q_1,q_2,q_3}(\gl_1)_{\geq N}(v)=0$,

\noindent
 (ii) $L$ is of \emph{finite type}, that is, the graded components $L_k$ are finite dimensional for all $k\in \ZZ$.
\end{defn}

 We say that $L$ is a \emph{highest weight} module if there exists $v_0\in L$
 generating $L$ and such that $f_i(v_0)=0, \psi_i(v_0)=p_i\cdot v_0, \psi_0^{-1}(v_0)=p_0^{-1}\cdot v_0$ for all $i\in \ZZ$
 and some $p_i\in \CC$ with $p_0\ne 0$.
 To such a collection $\{p_i\}$, we associate two series
   $p^\pm(z):=p_0^{\pm 1}+\sum_{m=1}^{\infty}p_{\pm m} z^{\mp m}\in \CC[[z^{\mp 1}]]$.
 For any such series $p^{\pm}(z)\in \CC[[z^{\mp 1}]]$,
 there is a universal highest weight representation $M_{p^+,p^-}$, which
 may be defined as the quotient of $\ddot{U}_{q_1,q_2,q_3}(\gl_1)$ by
 the left ideal generated by $\{f_i,\psi_i-p_i,\psi_0^{-1}-p_0^{-1}\}_{i\in \ZZ}$.
 Standard arguments show that $M_{p^+,p^-}$ has a unique irreducible quotient $V_{p^+,p^-}$.

 Module $V_{p^+,p^-}$ obviously satisfies the condition (i) of Definition~\ref{Cat O for toroidal}.
 Our next result provides a criterion for $V_{p^+,p^-}$ to satisfy the condition (ii),
 or equivalently to be in the category $\Oo$.

\begin{prop}\label{O_Tor}
 The module $V_{p^+,p^-}$ is of finite type if and only if there exists a
 rational function $P(z)$ such that $p^\pm(z)=P(z)^\pm$ and $P(0)P(\infty)=1$.
\end{prop}

\begin{proof}
$\ $

 Our proof is standard and is based on the arguments from~\cite{CP}.
 Define constants $\{\bar{p}_i\}_{i\in \ZZ}$ as
 $p_i$ (for $i>0$), $-p_i$ (for $i<0$), and $p_0-p_0^{-1}$ (for $i=0$).
 To prove the ``only if'' part, we choose indices $k\in \ZZ, l\in \ZZ_+$ such
 that $\{e_k(v_0),\ldots,e_{k+l}(v_0)\}$ span the degree $-1$ component $(V_{p^+,p^-})_{-1}$
 and $a_0e_k(v_0)+a_1e_{k+1}(v_0)+\ldots+a_le_{k+l}(v_0)=0$
 for some complex numbers $a_0,\ldots,a_l\in \CC$ with $a_l\ne 0$.
 Applying $f_{r-k}$ to the above equality and using
 the relation (T3) in the form $f_ie_j(v_0)=-\beta_1^{-1}\bar{p}_{i+j}\cdot v_0$,
 we get $a_0\bar{p}_r+a_1\bar{p}_{r+1}+\ldots+a_l\bar{p}_{r+l}=0$ for all $r\in \ZZ$.
 Therefore, the collection $\{\bar{p}_i\}_{i\in \ZZ}$ satisfies a simple recurrence relation.
 Solving this recurrence relation and using the condition $\bar{p}_0=p_0-p_0^{-1}$,
 we immediately see that $p^\pm(z)$ are extensions in $z^{\mp 1}$ of the same rational function.

 To prove the ``if'' direction, let us assume that $p^\pm(z)=P(z)^\pm$ for a rational function $P(z)$.
 Reversing the arguments from above, we get $\dim (V_{p^+,p^-})_{-1}<\infty$.
 Combining this with the relation (T1), a simple induction
 argument implies that $\dim (V_{p^+,p^-})_{-l}<\infty$ for all $l\in \NN$.
\end{proof}

\noindent
  $\bullet$ Category $\Oo$ for $\ddot{Y}_{h_1,h_2,h_3}(\gl_1)$.

\noindent
 We equip $\ddot{Y}_{h_1,h_2,h_3}(\gl_1)$ with the $\ZZ$-grading by assigning
   $\deg(e_j)=-1, \deg(f_j)=1, \deg(\psi_j)=0$.

\begin{defn}\label{Cat O for Yangian}
  We say that a $\ZZ$-graded $\ddot{Y}_{h_1,h_2,h_3}(\gl_1)$-module $L$ is in the category $\Oo$ if

\noindent
 (i) for any $v\in L$ there exists $N\in \ZZ$ such that $\ddot{Y}_{h_1,h_2,h_3}(\gl_1)_{\geq N}(v)=0$,

\noindent
 (ii) $L$ is of \emph{finite type}, that is, the graded components $L_k$ are finite dimensional for all $k\in \ZZ$.
\end{defn}

 We say that $L$ is a \emph{highest weight} module if there exists $v_0\in L$ generating $L$ and
 such that $f_j(v_0)=0,\ \psi_j(v_0)=p_j\cdot v_0$ for all $j\in \ZZ_+$ and some $p_j\in \CC$.
 Set $p(z):=1+\sum_{j\geq 0}p_j z^{-j-1}\in \CC[[z^{-1}]]$.
 For any $\{p_j\}$, there is a universal highest weight representation $M_{p}$,
 which may be defined as the quotient of $\ddot{Y}_{h_1,h_2,h_3}(\gl_1)$ by
 the left ideal generated by $\{f_j,\psi_j-p_j\}_{j\in \ZZ_+}$.
 It has a unique irreducible quotient $V_p$.

 Module $V_p$ obviously satisfies the condition (i) of Definition~\ref{Cat O for Yangian}.
 The following criterion for $V_p$ to be in the category $\Oo$ is completely analogous
 to the one for $V_{p^+,p^-}$:

\begin{prop}\label{O_Yan}
 The module $V_p$ is in the category $\Oo$ if and only if there exists a
 rational function $P(z)$ such that $p(z)=P(z)^+$ and $P(\infty)=1$.
\end{prop}


\section{Limit algebras}

 The goal of this section is to relate certain limits of our two algebras of interest
 to the well-known algebras of difference operators on $\CC$ and $\CC^*$.
 In this section, $h$ is a formal variable.

\subsection{Difference operators on $\CC^*$}
$\ $

 Set $q=\exp(h)\in \CC[[h]]$.
 The algebra of \emph{$q$-difference operators on $\CC^*$}, denoted by $\dd_q$, is the
 unital associative $\CC[[h]]$-algebra topologically generated by $Z^{\pm 1}, D^{\pm 1}$
 subject to the relations:
   $$Z\cdot Z^{-1}=Z^{-1}\cdot Z=1,\ D\cdot D^{-1}=D^{-1}\cdot D=1,\ D\cdot Z=qZ\cdot D.$$
 We will view $\dd_q$ as a Lie algebra with the natural commutator-Lie bracket $[\cdot,\cdot]$.
 It is easy to check that the following formula defines a 2-cocycle $\phi_{\dd}\in C^2(\dd_q,\CC[[h]])$:
  $$\phi_{\dd}(Z^aD^j,Z^bD^{-j'})=
    \left\{
     \begin{array}{llr}
        \sum_{i=-j}^{-1}q^{ai+b(i+j)} & \mathrm{if}\ \ j=j'>0\\
        -\sum_{i=j}^{-1}q^{bi+a(i-j)} & \mathrm{if}\ \ j=j'<0\\
        0 & \mathrm{otherwise}
     \end{array}
    \right. .$$

 This endows $\bar{\dd}_q=\dd_q\oplus \CC[[h]]\cdot c_\dd$ with the Lie algebra structure.

\subsection{Difference operators on $\CC$}
$\ $

 The algebra of \emph{$h$-difference operators on $\CC$}, denoted by $\D_h$, is the
 unital associative $\CC[[h]]$-algebra topologically generated by $x, \partial^{\pm 1}$
 subject to the following defining relations:
   $$\partial\cdot \partial^{-1}=\partial^{-1}\cdot \partial=1,\ \partial\cdot x=(x+h)\cdot \partial.$$
 We will view $\D_h$ as a Lie algebra with the natural commutator-Lie bracket $[\cdot,\cdot]$.
 It is easy to check that the following formula defines a 2-cocycle $\phi_\D\in C^2(\D_h,\CC[[h]])$:
  $$\phi_{\D}(f(x)\partial^r,g(x)\partial^{-s})=
    \left\{
     \begin{array}{llr}
        \sum_{l=-r}^{-1}f(lh)g((l+r)h) & \mathrm{if}\ \ r=s>0\\
        -\sum_{l=r}^{-1}g(lh)f((l-r)h) & \mathrm{if}\ \ r=s<0\\
        0 & \mathrm{otherwise}
     \end{array}
    \right. .$$

 This endows $\bar{\D}_h=\D_h\oplus \CC[[h]]\cdot c_\D$ with the Lie algebra structure.

\subsection{Isomorphism $\Upsilon_0$}
$\ $

 Let us introduce the appropriate completions of the algebras $\bar{\dd}_q$ and $\bar{\D}_h$:

\noindent
 $\circ$
 $\widehat{\bar{\dd}_q}$ is the completion of $\bar{\dd}_q$
 with respect to the powers of the two-sided ideal $J_{\dd}=(Z-1,q-1)$;

\noindent
 $\circ$
 $\widehat{\bar{\D}_h}$ is the completion of $\bar{\D}_h$
 with respect to the powers of the two-sided ideal $J_{\D}=(x,h)$.

 In other words, we have:
  $$\widehat{\bar{\dd}_q}:=\underset{\longleftarrow}\lim\ \bar{\dd}_q/\bar{\dd}_q\cdot (Z-1,q-1)^j,\
    \widehat{\bar{\D}_h}:=\underset{\longleftarrow}\lim\ \bar{\D}_h/\bar{\D}_h\cdot (x,h)^j.$$

\begin{rem}
 Taking completions of $\dd_q$ and $\D_h$ with respect to the ideals $J_{\dd}$ and $J_\D$ commutes
 with taking central extensions with respect to the 2-cocycles $\phi_{\dd}$ and $\phi_{\D}$.
\end{rem}

 The following result is straightforward:

\begin{prop}\label{Upsilon_0}
 The assignment
   $$D^{\pm 1}\mapsto \partial^{\pm 1},\ Z^{\pm 1}\mapsto e^{\pm x},\ c_\dd\mapsto c_\D$$
 extends to an isomorphism
   $\Upsilon_0:\widehat{\bar{\dd}_q}\iso  \widehat{\bar{\D}_h}$ of $\CC[[h]]$-algebras.
\end{prop}

\begin{rem}
 Specializing $h$ to a complex parameter $h_0\in \CC$, we get the classical
 $\CC$-algebras of difference operators $\dd_{q}$ and $\D_{h_0}$ as well as
 their one-dimensional central extensions $\bar{\dd}_{q}$ and $\bar{\D}_{h_0}$,
 where $q=\exp(h_0)\in \CC^*$. In other words, we consider the $\CC$-algebras
 given by the same collections of the generators and the defining relations.
 However, one can not define their completions as above.
 This explains our preference to work over $\CC[[h]]$.
\end{rem}

\subsection{Algebras $\ddot{U}'_{h_1,h_2,h_3}(\gl_1)$ and $\ddot{U}'_h(\gl_1)$}
$\ $

 Throughout this section, we let $h_2, h_3$ be formal variables and set $h_1:=-h_2-h_3$.
 We define $q_i:=\exp(h_i)\in \CC[[h_2,h_3]]$ for $i=1,2,3$.
 In order to consider a \emph{formal version} of $\ddot{U}_{q_1,q_2,q_3}(\gl_1)$,
 that is, the $\CC[[h_2,h_3]]$-algebra with the same generators and defining relations,
 we need to modify (T3) in an appropriate way.
 First, we \emph{renormalize} (T3) to the following form:
\begin{equation}\tag{T3$'$}\label{T3'}
  [e(z),f(w)]=\delta(z/w)(\psi^{+}(w)-\psi^{-}(z))/(1-q_3).
\end{equation}
 In the case of specialized values $q_i\in \CC\backslash\{0,1\}$,
 this corresponds to \emph{rescaling} $e_i, f_i$ by $1-q_1,1-q_2$.
 Next, we present $\psi^\pm(z)$ in the form
   $$\psi^\pm(z)=\exp\left(\mp \frac{h_3}{2}H_0\right)\cdot \exp\left(\pm(1-q_3)\sum_{m>0} H_{\pm m}z^{\mp m}\right).$$
 Switching from the generators $\{\psi_i,\psi_0^{-1}\}$ to $\{H_i\}$, the relations (T4,T5) get modified to:
\begin{equation}\tag{T4H}
  [H_0,e_j]=0,\ \ [H_i,e_j]=-\frac{(1-q_1^i)(1-q_2^i)(1-q_3^i)}{i(1-q_3)}e_{i+j}\ \mathrm{for}\ i\in \ZZ^*, j\in \ZZ,
\end{equation}
\begin{equation}\tag{T5H}
  [H_0,f_j]=0,\ \ [H_i,f_i]=\frac{(1-q_1^i)(1-q_2^i)(1-q_3^i)}{i(1-q_3)}f_{i+j}\ \mathrm{for}\ i\in \ZZ^*, j\in \ZZ.
\end{equation}
 These relations are well defined in the formal setting since
  $\frac{(1-q_1^i)(1-q_2^i)(1-q_3^i)}{1-q_3}\in \CC[[h_2,h_3]]$.
 Note that the right-hand side of (T3$'$) is also a polynomial in $H_i$ with coefficients in $\CC[[h_2,h_3]]$.

\begin{defn}
 $\ddot{U}'_{h_1,h_2,h_3}(\gl_1)$ is the unital associative $\CC[[h_2,h_3]]$-algebra topologically generated by
 $\{e_i,f_i,H_i|i\in \ZZ\}$ with the defining relations (T0,T1,T2,T3$'$,T4H,T5H,T6).
\end{defn}

 Finally, we define $\ddot{U}'_h(\gl_1)$ by
   $$\ddot{U}'_h(\gl_1):=\ddot{U}'_{-h-h_3,h,h_3}(\gl_1)/(h_3).$$
 It is a unital associative $\CC[[h]]$-algebra topologically generated by $\{e_i,f_i,H_i|i\in \ZZ\}$
 subject to the relations (T1,T2,T6) and
\begin{equation}\tag{T0L}\label{T40L}
  [H_i, H_j]=0,
\end{equation}
\begin{equation}\tag{T3L}
  [e_i,f_j]=H_{i+j},
\end{equation}
\begin{equation}\tag{T4tL}\label{T4tL}
  [H_i, e_j]=-(1-q^i)(1-q^{-i})e_{i+j},
\end{equation}
\begin{equation}\tag{T5tL}\label{T5tL}
  [H_i, f_j]=(1-q^i)(1-q^{-i})f_{i+j},
\end{equation}
 where $i,j\in \ZZ$ and $q=\exp(h)\in \CC[[h]]$.

\begin{rem}
 For $h_0\in \CC$, define $\ddot{U}_{h_0}(\gl_1)$ as the $\CC$-algebra generated by
 $\{e_i,f_i,H_i|i\in \ZZ\}$ with the same defining relations (T0L,T1,T2,T3L,T4tL,T5tL,T6),
 where $q=e^{h_0}\in \CC^*$.
\end{rem}

 The following result is straightforward:

\begin{prop}\label{limit_m}
 The assignment
   $$e_i\mapsto Z^iD,\ f_i\mapsto -D^{-1}Z^i,\ H_i\mapsto -(1-q^{-i})Z^i-q^{-i}c_\dd$$
 extends to a homomorphism $\theta:\ddot{U}'_h(\gl_1)\to U(\bar{\dd}_q)$.
\end{prop}

\begin{proof}
$\ $

 It suffices to show that all the defining relations of $\ddot{U}'_h(\gl_1)$ are preserved under
 the above assignment. This is a simple exercise, which we leave to the interested reader.
\end{proof}

 Let $\bar{\dd}^0_q\subset \bar{\dd}_q$ be the free $\CC[[h]]$-submodule spanned by
   $$\{c_\dd, hZ^kD^0, h^{j-1}Z^iD^{\pm j}|i\in \ZZ, k\in \ZZ^*, j\in \NN\}.$$

\begin{lem}
 $\bar{\dd}^0_q$ is a Lie subalgebra of $\bar{\dd}_q$ and $\mathrm{Im}(\theta)\subset U(\bar{\dd}^0_q)$.
\end{lem}

 In fact, we have the following result:

\begin{thm}\label{limit_1}
 The homomorphism $\theta$ provides an isomorphism $\theta:\ddot{U}'_h(\gl_1)\iso U(\bar{\dd}^0_q)$.
\end{thm}

 Note that all the defining relations of $\ddot{U}'_h(\gl_1)$ are of Lie-type.
 Hence, $\ddot{U}'_h(\gl_1)$ is an enveloping algebra of the Lie algebra generated by
 $e_i,f_i,H_i$ with the aforementioned defining relations.
 Thus, Theorem~\ref{limit_1} provides a presentation
 of the Lie algebra $\bar{\dd}^0_q$ by generators and relations.

 Actually, we will prove a more general result in Appendix C:

\begin{thm}\label{limit_1.1}
 If $h_0\in \CC\backslash \{\mathbb{Q}\cdot \pi \sqrt{-1}\}$,
 then $\theta$ induces an isomorphism of the $\CC$-algebras:
   $\theta:\ddot{U}_{h_0}(\gl_1)\iso U(\bar{\dd}^0_{q})$,
 where $\bar{\dd}^0_{q}\subset \bar{\dd}_{q}$ is a $\CC$-Lie subalgebra
 spanned by $\{c_\dd\}\cup\{Z^iD^j\}_{(i,j)\ne (0,0)}$.
\end{thm}

\subsection{Algebras $\ddot{Y}'_{h_1,h_2,h_3}(\gl_1)$ and $\ddot{Y}'_h(\gl_1)$}
$\ $

 Analogously to the previous section, let $h_2,h_3$ be formal
 variables and set $h_1:=-h_2-h_3$. We view $\ddot{Y}_{h_1,h_2,h_3}(\gl_1)$
 as a \emph{formal version} of the corresponding algebra introduced in Section 1.3.
 In other words, $\ddot{Y}_{h_1,h_2,h_3}(\gl_1)$ is the unital associative $\CC[[h_2,h_3]]$-algebra
 topologically generated by $\{e_j,f_j,\psi_j|j\in \ZZ_+\}$ subject to the relations (Y0--Y6).

 Let us renormalize the relations (Y4$'$) and (Y5$'$) to make them homogeneous:
\begin{equation}\tag{Y4$'$H}
  [\psi_0,e_j]=0,\ \ [\psi_1,e_j]=0,\ \ [\psi_2,e_j]=-2h_1h_2e_j,
\end{equation}
\begin{equation}\tag{Y5$'$H}
  [\psi_0,f_j]=0,\ \ [\psi_1,f_j]=0,\ \ [\psi_2,f_j]=2h_1h_2f_j.
\end{equation}

\begin{defn}
 The algebra $\ddot{Y}'_{h_1,h_2,h_3}(\gl_1)$ is the unital associative $\CC[[h_2,h_3]]$-algebra
 topologically generated by $\{e_j,f_j,\psi_j|j\in \ZZ_+\}$ subject to the relations (Y0--Y3,Y4,Y4$'$H,Y5,Y5$'$H,Y6).
\end{defn}

 We equip the algebra $\ddot{Y}'_{h_1,h_2,h_3}(\gl_1)$ with the $\ZZ_+$-grading by assigning
   $$\deg(e_j):=j,\ \deg(f_j):=j,\ \deg(\psi_j):=j,\ \deg(h_k):=1\ \ \mathrm{for}\ j\in \ZZ_+, k\in \{1,2,3\}.$$
 Finally, we define $\ddot{Y}'_h(\gl_1)$ by
    $$\ddot{Y}'_{h}(\gl_1):=\ddot{Y}'_{-h-h_3,h,h_3}(\gl_1)/(h_3).$$
 It is an associative algebra over $\CC[[h]]$.
 Its specialization at $h_0\in\CC$ is denoted by $\ddot{Y}_{h_0}(\gl_1)$.

 The following result is straightforward:

\begin{prop}\label{limit_a}
  The assignment
    $$e_j\mapsto x^j\partial,\ f_j\mapsto -\partial^{-1}x^j,\ \psi_j\mapsto (x-h)^j-x^j-(-h)^jc_\D$$
  extends to an algebra homomorphism $\vartheta:\ddot{Y}'_h(\gl_1)\to U(\bar{\D}_h)$.
\end{prop}

 Let $\bar{\D}^0_h\subset \bar{\D}_h$ be the free $\CC[[h]]$-submodule spanned by
   $\{c_\D, hx^i\partial^0, h^{j-1}x^i\partial^{\pm j}|i\in \ZZ_+, j\in \NN\}.$

\begin{lem}
 $\bar{\D}^0_h$ is a Lie subalgebra of $\bar{\D}_h$ and $\mathrm{Im}(\vartheta)\subset U(\bar{\D}^0_h)$.
\end{lem}

 In fact, we have the following result:

\begin{thm}\label{limit_2}
  The homomorphism $\vartheta$ provides an isomorphism $\vartheta:\ddot{Y}'_h(\gl_1)\iso U(\bar{\D}^0_h)$.
\end{thm}

 Note that all the defining relations of $\ddot{Y}'_h(\gl_1)$ are of Lie-type.
 Hence, $\ddot{Y}'_h(\gl_1)$ is an enveloping algebra of the Lie algebra generated by
 $e_j,f_j,\psi_j$ with the aforementioned defining relations.
 Thus, Theorem~\ref{limit_2} provides a presentation of the Lie algebra $\bar{\D}^0_h$ by generators and relations.

 Actually, we will prove a more general result in Appendix C:

\begin{thm}\label{limit_2.1}
 For $h_0\in \CC^*$, $\vartheta$ induces an isomorphism of $\CC$-algebras
 $\vartheta:\ddot{Y}_{h_0}(\gl_1)\iso U(\bar{\D}_{h_0})$.
\end{thm}


\section{Key homomorphism and flatness of deformations}

 Following~\cite{GTL}, we construct an algebra homomorphism
   $\Upsilon: \ddot{U}'_{h_1,h_2,h_3}(\gl_1)\to \widehat{\ddot{Y}}'_{h_1,h_2,h_3}(\gl_1)$
 and establish a \emph{compatible} isomorphism of their faithful representations.
 We also prove the flatness result (Theorem~\ref{flatness})
 for both $\ddot{U}'_{h_1,h_2,h_3}(\gl_1)$ and $\ddot{Y}'_{h_1,h_2,h_3}(\gl_1)$.

\subsection{Homomorphism $\Upsilon$}
$\ $

 Let $\widehat{\ddot{Y}}'_{h_1,h_2,h_3}(\gl_1)$ be the completion of
 $\ddot{Y}'_{h_1,h_2,h_3}(\gl_1)$ with respect to the $\ZZ_+$-grading from Section 5.5.
 To state our main result, we introduce the following notation (compare with~\cite{GTL}):

\noindent
 $\bullet$
  Define $\psi(z)$ as in Section 1.4:
   $\psi(z):=1-h_3\sum_{j\geq 0}\psi_jz^{-j-1}\in \ddot{Y}'_{h_1,h_2,h_3}(\gl_1)[[z^{-1}]]$.

\noindent
 $\bullet$
 Define $k_j\in \CC[\psi_0,\psi_1,\psi_2,\ldots]$ via
   $\sum_{j\geq  0}k_jz^{-j-1}=k(z):=\log(\psi(z))$.

\noindent
 $\bullet$
 Define the \emph{inverse Borel transform} $B:z^{-1}\CC[[z^{-1}]]\to \CC[[w]]$ by
   $\sum_{j=0}^\infty \frac{a_j}{z^{j+1}}\mapsto \sum_{j=0}^\infty \frac{a_j}{j!}w^j$.

\noindent
 $\bullet$
 Define $B(w)\in h_3\ddot{Y}'_{h_1,h_2,h_3}(\gl_1)[[w]]$ to be the inverse Borel transform of $k(z)$.

\noindent
 $\bullet$
 Define a function
   $G(v):=\log\left(\frac{v}{e^{v/2}-e^{-v/2}}\right)\in v\mathbb{Q}[[v]]$.

\noindent
 $\bullet$
 Define $\gamma(v):=-B(-\partial_v)G'(v)\in \widehat{\ddot{Y}}'_{h_1,h_2,h_3}(\gl_1)[[v]]$.

\noindent
 $\bullet$
 Define $\Gg(v)=\sum_{j\geq 0} \Gg_j v^j\in \widehat{\ddot{Y}}'_{h_1,h_2,h_3}(\gl_1)[[v]]$ by
   $\Gg(v):=\left(\frac{h_3}{q_3-1}\right)^{1/2} \cdot \exp\left(\frac{\gamma(v)}{2}\right)$.

\medskip
 The identity $B(\log(1-s/z))=(1-e^{sw})/w$ immediately implies the following result:

\begin{cor}\label{Borel}
 The equalities from Proposition~\ref{yangian_generating}(e,f) are equivalent to
  $$[B(w),e_j]=\frac{\sum_{i=1}^3(e^{h_iw}-e^{-h_iw})}{w}\ e^{w\sigma^+}e_j,\
    [B(w),f_j]=\frac{\sum_{i=1}^3(e^{-h_iw}-e^{h_iw})}{w}\ e^{w\sigma^-}f_j.$$
\end{cor}

 Now we are ready to state the main result of this section.

\begin{thm}\label{homomorphism}
  The assignment
\begin{equation}\label{Upsilon}\tag{$\ddag$}
  H_0 \mapsto \psi_0,\ H_m\mapsto \frac{B(m)}{1-q_3},\
  e_k\mapsto e^{k\sigma^+}\Gg(\sigma^+)e_0,\
  f_k\mapsto e^{k\sigma^-}\Gg(\sigma^-)f_0\ \
  \mathrm{for}\ k\in \ZZ, m\in \ZZ^*
\end{equation}
  extends to an algebra homomorphism
    $$\Upsilon:\ddot{U}'_{h_1,h_2,h_3}(\gl_1)\to \widehat{\ddot{Y}}'_{h_1,h_2,h_3}(\gl_1).$$
\end{thm}

 The proof of this theorem is presented in Section~\ref{Proof}.

\subsection{Limit of $\Upsilon$}
$\ $

 Recall the isomorphisms from Theorems~\ref{limit_1} and~\ref{limit_2}:
  $$\theta:\ddot{U}'_{-h-h_3,h,h_3}(\gl_1)/(h_3)\iso U(\bar{\dd}^0_q),$$
  $$\vartheta:\ddot{Y}'_{-h-h_3,h,h_3}(\gl_1)/(h_3)\iso U(\bar{\D}^0_h).$$
 The homomorphism $\Upsilon$ factors through the factors by $(h_3)$, inducing
 $\Upsilon_{\mid_{h_3=0}}:U(\bar{\dd}^0_q)\to \widehat{U(\bar{\D}^0_h)}$.

\begin{prop}\label{limit_Upsilon}
 The limit homomorphism $\Upsilon_{\mid_{h_3=0}}$ is induced by $\Upsilon_0$.
\end{prop}

\begin{proof}$\ $

 To prove this result, we explicitly compute the images of the generators under $\Upsilon_{\mid_{h_3=0}}$:

\noindent
 $\circ$
   $\Upsilon_{\mid_{h_3=0}}(c_\dd)=c_\D$.

\noindent
 $\circ$
   $\Upsilon_{\mid_{h_3=0}}((q^{-i}-1)Z^i-q^{-i}c_\dd)=
    \sum_{k\geq 0}((x-h)^k-x^k-(-h)^kc_\D)\frac{i^k}{k!}=(q^{-i}-1)e^{ix}-q^{-i}c_\D$.

\noindent
 $\circ$
  $\Upsilon_{\mid_{h_3=0}}(Z^iD)=\sum_{k\geq 0} \frac{i^k}{k!}\cdot x^k\partial=e^{ix}\partial$.

\noindent
 $\circ$
  $\Upsilon_{\mid_{h_3=0}}(-D^{-1}Z^i)=-\sum_{k\geq 0} \frac{i^k}{k!}\partial^{-1}\cdot x^k=-\partial^{-1}e^{ix}$.
\end{proof}

\subsection{Elliptic Hall algebra}
$\ $

 We recall the elliptic Hall algebra studied in~\cite{BS}.
 First, we introduce the following notation:

\noindent
 $\bullet$
 Define $(\ZZ^2)^*:=\ZZ^2\backslash\{(0,0)\}$
 and $(\ZZ^2)^{\pm}:=\{(a,b)|\pm a>0\ \mathrm{or}\ a=0, \pm b>0\}.$

\noindent
 $\bullet$
 For $\xx=(a,b)\in (\ZZ^2)^*$, we define $\deg(\xx):=\mathrm{gcd}(a,b)\in \NN$.

\noindent
 $\bullet$
 For $\xx\in (\ZZ^2)^*$, we define
 $\epsilon_\xx:=1$ if $\xx\in (\ZZ^2)^+$ and $\epsilon_\xx:=-1$ if $\xx\in (\ZZ^2)^-$.

\noindent
 $\bullet$
 For non-collinear $\xx,\yy\in (\ZZ^2)^*$, we define
  $\epsilon_{\xx,\yy}:=\mathrm{sign}(\det(\xx,\yy))\in \{\pm 1\}$.

\noindent
 $\bullet$
 For non-collinear $\xx,\yy\in (\ZZ^2)^*$,
 we denote the triangle with vertices $\{(0,0),\xx,\xx+\yy\}$ by $\triangle_{\xx,\yy}$.

\noindent
 $\bullet$
 We say that $\triangle_{\xx,\yy}$ is \emph{empty} if there are no lattice points inside this triangle.

\noindent
 $\bullet$
 For $n\in \ZZ^*$, we define
   $\alpha_n:=-\frac{\beta_n}{n}=\frac{(1-q_1^{-n})(1-q_2^{-n})(1-q_3^{-n})}{n}$.

\begin{defn}[\cite{BS}]
 The (central extension of) elliptic Hall algebra $\widetilde{\E}$ is the unital
 associative algebra generated by $\{u_\xx, \kappa_\yy| \xx\in (\ZZ^2)^*, \yy\in \ZZ^2\}$
 with the following defining relations:
\begin{equation}\tag{E0}\label{E0}
 \kappa_\xx \kappa_\yy=\kappa_{\xx+\yy},\ \kappa_{0,0}=1,
\end{equation}
\begin{equation}\tag{E1}\label{E1}
 [u_\yy,u_\xx]=\delta_{\xx,-\yy}\cdot\frac{\kappa_\xx-\kappa_\xx^{-1}}{\alpha_{\deg(\xx)}}\
 \mathrm{if}\ \xx,\yy\ \mathrm{are\ collinear},
\end{equation}
\begin{equation}\tag{E2}\label{E2}
 [u_\yy,u_\xx]=\epsilon_{\xx,\yy}\kappa_{\alpha(\xx,\yy)}\frac{\theta_{\xx+\yy}}{\alpha_1}\
 \mathrm{if}\ \triangle_{\xx,\yy}\ \mathrm{is\ empty\ and}\ \deg(\xx)=1,
\end{equation}
 where the elements $\{\theta_\xx|\xx\in (\ZZ^2)^*\}$ are determined from the equality
\begin{equation}\tag{E3}\label{E3}
 \sum_{n\geq 0}\theta_{n\xx_0}x^n=\exp\left(\sum_{r>0}\alpha_r u_{r\xx_0} x^r\right)
 \mathrm{for}\ \xx_0\in (\ZZ^2)^*\ \mathrm{with}\ \deg(\xx_0)=1,
\end{equation}
 while $\alpha(\xx,\yy)$ is defined by
\begin{equation}\tag{E4}\label{E4}
 \alpha(\xx,\yy)=\left\{
     \begin{array}{llr}
        \epsilon_\xx(\epsilon_\xx \xx+\epsilon_\yy \yy-\epsilon_{\xx+\yy}(\xx+\yy))/2\ \ \mathrm{if}\  \epsilon_{\xx,\yy}=1\\
        \epsilon_\yy(\epsilon_\xx \xx+\epsilon_\yy \yy-\epsilon_{\xx+\yy}(\xx+\yy))/2\ \ \mathrm{if}\  \epsilon_{\xx,\yy}=-1
     \end{array}
    \right..
\end{equation}
\end{defn}

 The relation of this algebra to the quantum toroidal of $\gl_1$ is given in the following theorem:

\begin{thm}\cite{S}\label{Hall}
  The assignment
   $$u_{1,i}\mapsto e_i,\ u_{-1,i}\mapsto f_i,\ \theta_{0,\pm j}\mapsto \psi_{\pm j}\cdot \psi_0^{\mp 1},\
     \kappa_{a,b}\mapsto \psi_0^a\ \ \mathrm{for}\ i,a,b\in\ZZ, j\in \NN$$
  extends to an isomorphism of algebras $\Xi: \widetilde{\E}/(\kappa_{0,1}-1)\iso \ddot{U}_{q_1,q_2,q_3}(\gl_1)$.
\end{thm}

\begin{rem}
 This theorem was proved in~\cite{S} only for $\E:=\widetilde{\E}/(\kappa_\yy-1)_{\yy\in \ZZ^2}$,
 but the above generalization is straightforward.
 According to~\cite{BS}, $\E$ is also isomorphic to the Drinfeld double of
 the spherical Hall algebra of an elliptic curve over a finite field.
\end{rem}

 This result provides distinguished elements $\{u_{\xx}|\xx\in (\ZZ^2)^*\}$
 of $\ddot{U}'_{h_1,h_2,h_3}(\gl_1)$. Their images in
 $\ddot{U}'_{h_2}(\gl_1)=\ddot{U}'_{-h_2-h_3,h_2,h_3}(\gl_1)/(h_3)$ will be denoted by $\bar{u}_\xx$.

\begin{lem}
 The isomorphism $\theta$ maps these elements $\bar{u}_\xx$ as follows:
\begin{equation}\label{Pick1}
 \bar{u}_{0,r}\mapsto
 \mathrm{sign}(r)\frac{(1-q_2^{-1})(1-q_2)}{1-q_2^r}\left(Z^r-\frac{1}{1-q_2^r}c_\dd\right)\
 \mathrm{for}\ r\in \ZZ^*,
\end{equation}
\begin{equation}\label{Pick2}
 \bar{u}_{\pm k,\pm l}\mapsto
 \pm q_2^{\pm f(k,l)}\frac{1-q_2^{\mp 1}}{1-q_2^{\mp d}}(1-q_2^{\pm 1})^{k-1}Z^{\pm l}D^{\pm k}q_2^{\frac{kl}{2}\mp \frac{kl}{2}}\
 \mathrm{for}\ k\in \NN,l\in \ZZ,\
\end{equation}
 where $d:=\mathrm{gcd}(k,l)\in \NN$ and $f(k,l):=\frac{kl-k-l-d+2}{2}\in \ZZ$
 (note that $f(k,l)$ equals the number of lattice points inside
 the triangle with vertices $\{(0,0), (0,l), (k,l)\}$ if $k,l>0$).
\end{lem}

\begin{proof}
$\ $

 Considering the ``$h_3\to 0$ limit'' of the relation~(\ref{E2}), we find
\begin{equation}\tag{E2$'$}\label{E2'}
  [\bar{u}_\yy,\bar{u}_\xx]=\epsilon_{\xx,\yy}\frac{\alpha_{\deg(\xx+\yy)}}{\alpha_1}{\bar{u}_{\xx+\yy}}\
  \mathrm{if}\ \triangle_{\xx,\yy}\ \mathrm{is\ empty\ and}\ \deg(\xx)=1.
\end{equation}
 In particular, we get
  $\bar{u}_{0,r}=\mathrm{sign}(r)\frac{\alpha_1}{\alpha_r}[\bar{u}_{-1,0},\bar{u}_{1,r}]$.
 Applying $\theta$ to both sides, we recover~(\ref{Pick1}).

 We prove~(\ref{Pick2}) by an induction on $k$; we will consider only the sign ``+'' case.
 Case $k=1$ is trivial. Given $(k,l)\in \ZZ_{>1}\times \ZZ$, choose unique
 $\xx=(k_1,l_1),\yy=(k_2,l_2), 0< k_1,k_2<k,$ such that
 $\xx+\yy=(k,l)$, $\epsilon_{\xx,\yy}=1$, $\deg(\xx)=\deg(\yy)=1$, and $\triangle_{\xx,\yy}$ is empty.
 Combining the formula~(\ref{E2'}) with the induction assumption on
 $\theta(\bar{u}_\xx)$ and $\theta(\bar{u}_\yy)$, we find
   $$\theta(\bar{u}_{k,l})=
     \frac{(1-q_2)(1-q_2^{-1})}{(1-q_2^d)(1-q_2^{-d})}q_2^{f(k_1,l_1)+f(k_2,l_2)}(q_2^{k_2l_1}-q_2^{k_1l_2})(1-q_2)^{k_1+k_2-2}Z^{l_1+l_2}D^{k_1+k_2}.$$
 Our choice of $\xx,\yy$ and the Pick's formula imply that $k_1l_2-k_2l_1=d$.
 As a result, we have
  $q_2^{k_2l_1}-q_2^{k_1l_2}=q_2^{k_2l_1}(1-q_2^d)$ and
  $f(k_1,l_1)+f(k_2,l_2)+k_2l_1=f(k_1+k_2,l_1+l_2)=f(k,l)$.
 This completes the induction step.
\end{proof}

\subsection{Flatness and faithfulness}
$\ $

 The main result of this section is:

\begin{thm}\label{flatness}
 (a) The algebra $\ddot{U}'_{h_1,h_2,h_3}(\gl_1)$ is a flat $\CC[[h_3]]$-deformation of
 $\ddot{U}'_{h_2}(\gl_1)\simeq U(\bar{\dd}^0_{q_2})$.

\noindent
 (b) The algebra $\ddot{Y}'_{h_1,h_2,h_3}(\gl_1)$ is a flat $\CC[[h_3]]$-deformation of
 $\ddot{Y}'_{h_2}(\gl_1)\simeq U(\bar{\D}^0_{h_2})$.
\end{thm}

\begin{proof}
$\ $

 To prove Theorem~\ref{flatness}, it suffices to provide a faithful
 $U(\bar{\dd}^0_{q_2})$-representation (respectively $U(\bar{\D}^0_{h_2})$-representation)
 which admits a flat deformation to a representation of $\ddot{U}'_{h_1,h_2,h_3}(\gl_1)$
 (respectively $\ddot{Y}'_{h_1,h_2,h_3}(\gl_1)$).
 Let $R$ be a localization of  $\CC[[h_2,h_3]]$ by the multiplicative set
 $\{(h_2-\nu_1h_3)\cdots (h_2-\nu_sh_3)|s\in \NN, \nu_j\in  \CC\}$.
 Note that $\bar{R}:=R/(h_3)\simeq \CC((h_2))$.
 The ring $R$ is needed to make use of the representations from Sections 2--4, therefore, we define
   $$\ddot{U}'_R(\gl_1):=\ddot{U}'_{h_1,h_2,h_3}(\gl_1)\otimes_{\CC[[h_2,h_3]]} R,\ \ \
     \ddot{Y}'_R(\gl_1):=\ddot{Y}'_{h_1,h_2,h_3}(\gl_1)\otimes_{\CC[[h_2,h_3]]} R.$$

 Consider the Lie algebra
   $$\gl_{\infty}=\left\{\sum_{i,j\in \ZZ}a_{i,j} E_{i,j}| a_{i,j}\in \CC[[h_2]]\
     \mathrm{and}\ a_{i,j}=0\ \mathrm{for}\ |i-j|\gg 0\right\}.$$
 Let $\bar{\gl}_{\infty}=\gl_\infty\oplus \CC[[h_2]]\cdot \kappa$ be the central extension
 of this Lie algebra via the 2-cocycle
   $$\phi_\gl\left(\sum a_{i,j}E_{i,j}, \sum b_{i,j}E_{i,j}\right)=
     \sum_{i\leq 0<j} a_{i,j}b_{j,i}-\sum_{j\leq 0<i}a_{i,j}b_{j,i}.$$
 For any $u\in 1+h_2\CC[[h_2]]$, consider the homomorphism
 $\tau_u: U_{\bar{R}}(\bar{\dd}^0_{q_2})\to U_{\bar{R}}(\bar{\gl}_{\infty})$ induced by
   $$c_{\dd}\mapsto -\kappa\ \ \mathrm{and}\ \
    Z^kD^l\mapsto \sum_{i\in \ZZ} u^kq_2^{(1-i)k}E_{i,i-l}-\delta_{0,l}\frac{1-q_2^{k}u^k}{1-q_2^{k}}\kappa\
    \ \mathrm{for}\ \ (k,l)\in (\ZZ^2)^*.$$
 Let $\varpi_u:\ddot{U}'_{\bar{R}}(\gl_1)\to U_{\bar{R}}(\bar{\gl}_{\infty})$ be the
 composition of $\theta:\ddot{U}'_{\bar{R}}(\gl_1)\to U_{\bar{R}}(\bar{\dd}^0_{q_2})$ and $\tau_u$.
 Then
   $$\varpi_u(e(z))=\sum_{i\in \ZZ} \delta(q_2^{-i}u/z)E_{i+1,i},\
     \varpi_u(f(z))=-\sum_{i\in \ZZ} \delta(q_2^{-i}u/z)E_{i,i+1}.$$
 Let $F_\infty$ be the fundamental representation of $\bar{\gl}_{\infty}$.
 It is realized on $\wedge^{\infty/2} \CC^\infty$ with the highest
 weight vector $w_0\wedge w_{-1}\wedge w_{-2}\wedge \cdots$
 (here $\CC^\infty$ is a $\CC$-vector spaces with the basis $\{w_i\}_{i\in \ZZ}$).
 Comparing the formulas for the Fock $\ddot{U}'_{R}(\gl_1)$-module $F'_R(u)$ with those for
 the $\bar{\gl}_{\infty}$-action on $F_\infty$, we see that $F'_R(u)$ degenerates to
 $\varpi_u^*(F_\infty)$ (the intertwining linear map is given by
 $|\lambda\rangle \mapsto w_{\lambda_1}\wedge w_{\lambda_2-1}\wedge w_{\lambda_3-2}\wedge\cdots $).
 Moreover, it is easy to see that any legible finite tensor product $F'_R(u_1)\otimes \cdots\otimes F'_R(u_n)$
 degenerates to $\varpi_{u_1}^*(F_\infty)\otimes \cdots \otimes \varpi_{u_n}^*(F_\infty)$.

 It remains to prove that the module
   $\bigoplus_n\bigoplus_{u_1,\ldots,u_n}\tau_{u_1}^*(F_\infty)\otimes \cdots \otimes \tau_{u_n}^* (F_\infty)$
 is a faithful representation of $U(\bar{\dd}^0_{q_2})$, where the sum
 is over all collections  $u_1,\ldots,u_n\in 1+h_2\CC[[h_2]]$ which are not in resonance.
 This follows from the faithfulness of the $\bar{\dd}^0_{q_2}$-action on
 each $\tau_u^*(F_\infty)$  for any $u\in 1+h_2\CC[[h_2]]$,
 which is a simple exercise left to the interested reader.

 For the $\ddot{Y}'_{\bar{R}}(\gl_1)$ case, we consider the homomorphism
   $\varsigma_v:U_{\bar{R}}(\bar{\D}^0_{h_2})\to U_{\bar{R}}(\bar{\gl}_{\infty})$
 induced by
   $$c_\D\mapsto -\kappa \ \ \mathrm{and}\ \
     x^n\partial^l\mapsto \sum_{i\in \ZZ} (v+(1-i)h_2)^n E_{i,i-l}+\delta_{0,l}c_n\kappa
     \ \ \mathrm{for}\ \ n\in \ZZ_+, l\in \ZZ,$$
 where $c_n\in {\bar{R}}$ are determined recursively from
   $$\binom{n}{1}h_2c_{n-1}-\binom{n}{2}h_2^2c_{n-2}+\ldots+(-1)^{n+1}h_2^nc_0-(-h_2)^n+v^n=0.$$
 The rest of the arguments are the same.
 This completes our proof of Theorem~\ref{flatness}.
\end{proof}

\begin{cor}\label{faithfulness}
 (a) The following is a faithful $\ddot{U}'_R(\gl_1)$-representation:
    $$\mathbf{F}'_R:=
      \bigoplus_{n\in \NN}\bigoplus_{u_1,\ldots,u_n\in 1+h_2\CC[[h_2]]-\mathrm{not\ in\ resonance}}
      F'_R(u_1)\otimes \cdots\otimes F'_R(u_n).$$
\noindent
 (b) The following is a faithful $\ddot{Y}'_R(\gl_1)$-representation:
   $$^{a}\mathbf{F}'_R:=
     \bigoplus_{n\in \NN}\bigoplus_{v_1,\ldots,v_n\in h_2\CC[[h_2]]-\mathrm{not\ in\ resonance}}
     {^{a}F}'_R(v_1)\otimes \cdots\otimes {^{a}F}'_R(v_n).$$
\end{cor}

 As another consequence of Theorem~\ref{flatness} and Proposition~\ref{limit_Upsilon}, we have:

\begin{cor}
 The homomorphism $\Upsilon$ is injective.
\end{cor}

\begin{rem}
 In contrast to~\cite{GTL}, $\Upsilon$ does not induce an isomorphism of appropriate completions,
 since the homomorphism $\Upsilon_{\mid_{h_3=0}}$ does not induce an isomorphism of
 $\widehat{\bar{\dd}^0_{q_2}}$ and $\widehat{\bar{\D}^0_{h_2}}$.
\end{rem}

\subsection{Compatible isomorphisms of representations}
$\ $

 Given $n\in \NN$, consider two $n$-tuples
   $$\vv=(v_1,\ldots,v_n)\in ((h_2,h_3)\CC[[h_2,h_3]])^n,\
     \uu=(u_1,\ldots,u_n)\in (1+(h_2,h_3)\CC[[h_2,h_3]])^n.$$
 Associated to this data, we have a collection of Fock
 $\ddot{U}'_R(\gl_1)$-representations $\{F'_R(u_i)\}_{i=1}^n$
 and a collection of Fock $\ddot{Y}'_R(\gl_1)$-representations
 $\{{^{a}F}'_R(v_i)\}_{i=1}^n$. Consider the tensor products
   $$F'_R(\uu):=F'_R(u_1)\otimes F'_R(u_2)\otimes \cdots \otimes F'_R(u_n) -
     \mathrm{representation\ of}\ \ddot{U}'_R(\gl_1),$$
   $${^{a}F}'_R(\vv):={^{a}F}'_R(v_1)\otimes {^{a}F}'_R(v_2)\otimes \cdots \otimes {^{a}F}'_R(v_n) -
     \mathrm{representation\ of}\ \ddot{Y}'_R(\gl_1),$$
 whenever these representations are well defined, i.e.,
 $\{u_i\}_{i=1}^n$ and $\{v_i\}_{i=1}^n$ are \emph{not in resonance}.
 Both of these tensor products have natural bases $\{|\bar{\lambda}\rangle\}$
 labeled by $n$-tuples of Young diagrams $\bar{\lambda}$.
 Our key result establishes an isomorphism of these tensor products compatible with $\Upsilon$.

\begin{thm}\label{intertwiner}
 For any $\vv$ as above, define $u_i:=e^{v_i}$.
 There exists a unique collection of constants $\bc_{\bar{\lambda}}\in R$
 such that $\bc_{\bar{\emptyset}}=1$ and the corresponding $R$-linear isomorphism of vector spaces
   $$I_{\vv}: F'_R(\uu)\iso  {^{a}F}'_R(\vv)\ \mathrm{given\ by}\
     |\bar{\lambda}\rangle \mapsto \bc_{\bar{\lambda}}\cdot |\bar{\lambda}\rangle$$
 satisfies the property
\begin{equation}\tag{\dag}\label{dag}
   I_{\vv}(X(w))=\Upsilon(X)(I_{\vv}(w))\
   \mathrm{for\ all}\ w\in F'_R(\uu)\ \mathrm{and}\ X\in \{H_i, e_i, f_i\}_{i\in \ZZ}.
\end{equation}
\noindent
 We say that $I_{\vv}$ is compatible with $\Upsilon$ if ($\dag$) holds.
\end{thm}

\begin{proof}
$\ $

 By straightforward computation, one can see that ($\dag$) holds for all
  $w=|\bar{\lambda}\rangle, X=H_i$
 and an arbitrary choice of $\bc_{\bar{\lambda}}$.
 On the other hand, the equalities
   $$I_{\vv}(e_i(|\bar{\lambda}\rangle))=\Upsilon(e_i)(I_{\vv}(|\bar{\lambda}\rangle))\
     \mathrm{and}\
     I_{\vv}(f_i(|\bar{\lambda}\rangle))=\Upsilon(f_i)(I_{\vv}(|\bar{\lambda}\rangle))\
     \mathrm{for\ all}\ \bar{\lambda}, i$$
 are both equivalent to
\begin{equation}\label{ratio 1}
   \frac{\bc_{\bar{\lambda}+\square^k_i}}{\bc_{\bar{\lambda}}}=\bd_{\bar{\lambda},\square^k_i},
\end{equation}
 where
\begin{multline}\label{compatibility ratio 1}
  \bd_{\bar{\lambda},\square^k_i}:=
  q^{k-1-n/2}\cdot\left(\frac{h_1}{1-q_1^{-1}}\cdot \frac{q_2-1}{h_2}\right)^{1/2}\times \\
  \prod_{(a,j)\ne (k,i)}
   \left\{\left(\frac{(\chi^{(k)}_i-q_2\chi^{(a)}_j)(\chi^{(k)}_i-q_3\chi^{(a)}_j)}
              {(\chi^{(k)}_i-\chi^{(a)}_j)(\chi^{(k)}_i-q_1^{-1}\chi^{(a)}_j)}\right)\cdot
   \left(\frac{(x^{(k)}_i-x^{(a)}_j)(x^{(k)}_i-x^{(a)}_j+h_1)}
              {(x^{(k)}_i-x^{(a)}_j-h_2)(x^{(k)}_i-x^{(a)}_j-h_3)}\right)\right\}^{\epsilon^{(a,j)}_{(k,i)}}.
\end{multline}
 In this formula, we use the following notation:
   $$x^{(a)}_j:=(\lambda^{a}_j-1)h_1+(j-1)h_2+v_a,\
     \chi^{(a)}_j:=\exp(x^{(a)}_j),\
     \epsilon^{(a,j)}_{(k,i)}:=
     \begin{cases}
       1/2 & \text{if}\ \ (a,j)\succ (k,i)\\
       -1/2 & \text{if}\ \ (a,j)\prec (k,i)
     \end{cases}.$$
 The uniqueness of $\bc_{\bar{\lambda}}\in R$ satisfying the relation
 (\ref{ratio 1}) with the initial condition $\bc_{\bar{\emptyset}}=1$ is obvious.
 The existence of such $\bc_{\bar{\lambda}}$ is equivalent to
   $$\bd_{\bar{\lambda}+\square^k_i,\square^l_j}\cdot \bd_{\bar{\lambda}, \square^k_i}=
     \bd_{\bar{\lambda}+\square^l_j,\square^k_i}\cdot \bd_{\bar{\lambda}, \square^l_j}$$
 for all possible $\bar{\lambda}, \square^k_i, \square^l_j$.
 The verification of this identity is straightforward.
\end{proof}

\subsection{Proof of Theorem~\ref{homomorphism}}\label{Proof}
$\ $

 Recall the faithful $\ddot{U}'_R(\gl_1)$-representation ${\bf{F}}'_R$
 and the faithful  $\ddot{Y}'_R(\gl_1)$-representation ${^{a}\bf{F}}'_R$
 from Corollary~\ref{faithfulness}.
 Note that $\exp:(h_2,h_3)\CC[[h_2,h_3]]\to 1+(h_2,h_3)\CC[[h_2,h_3]]$ is a
 bijective map and $\{v_i\}_{i=1}^n$ are not in resonance if and only if
 $\{e^{v_i}\}_{i=1}^n$ are not in resonance.

 According to Theorem~\ref{intertwiner},
 we have an $R$-linear isomorphism
   ${\bf{I}}:{\bf{F}}'_R\iso {^{a}\bf{F}}'_R$ compatible with $\Upsilon$
 in the following sense:
  $${\bf{I}}(X(w))=\Upsilon(X)({\bf{I}}(w))\
    \mathrm{for\ all}\ w\in {\bf{F}}'_R\ \mathrm{and}\ X\in \{H_i,e_i,f_i\}_{i\in \ZZ}.$$
 For any $X\in \{H_i,e_i,f_i\}_{i\in \ZZ}$, consider the assignment $X\mapsto \Upsilon(X)$
 with the right-hand side defined via~(\ref{Upsilon}).
 Then Theorem~\ref{homomorphism} is equivalent to saying that this
 assignment preserves all the defining relations of $\ddot{U}'_R(\gl_1)$.
 The latter follows immediately from the faithfulness of ${^{a}\bf{F}}'_R$
 combined with an existence of the compatible isomorphism $\bf{I}$.

\begin{rem}
 One can directly check that the aforementioned assignment given
 by~(\ref{Upsilon}) preserves all the defining relations of $\ddot{U}'_{h_1,h_2,h_3}(\gl_1)$,
 except for the Serre relations (compare with~\cite{GTL}).
 In particular, the compatibility with the relations~(T4H,T5H) follows from Corollary~\ref{Borel}.
 Actually, we used this approach to determine the formulas in~(\ref{Upsilon}).
 However, the arguments of~\cite{GTL} on the compatibility with the Serre relations
 are not applicable in our settings. Instead, one can prove this compatibility by utilizing
 the shuffle approach from the next section. We refer the interested reader to~\cite{TB},
 where we discuss this in the greater generality.
\end{rem}


\section{Shuffle algebras $S^m$ and $S^a$}

 We introduce the \emph{small multiplicative} and \emph{small additive shuffle algebras}.
 Their relation to $\ddot{U}_{q_1,q_2,q_3}(\gl_1)$ and $\ddot{Y}_{h_1,h_2,h_3}(\gl_1)$ is recalled.
 We also discuss their commutative subalgebras.

\subsection{Small multiplicative shuffle algebra $S^m$}
$\ $

 Consider a $\ZZ_+$-graded $\CC$-vector space $\sS^m=\bigoplus_{n\geq 0}\sS^m_n$,
 where $\sS^m_n$ consists of rational functions $\frac{f(x_1,\ldots,x_n)}{\Delta(x_1,\ldots,x_n)}$
 with $f\in \CC[x_1^{\pm 1},\ldots,x_n^{\pm 1}]^{\mathfrak{S}_n}$
 and $\Delta(x_1,\ldots,x_n):=\prod_{i\ne j}(x_i-x_j)$.

 Define the star product $\overset{m}\star:\sS^m_k\times \sS^m_l\to \sS^m_{k+l}$ by
  $$(F\overset{m}\star G)(x_1,\ldots,x_{k+l}):=
    \mathrm{Sym}_{\mathfrak{S}_{k+l}}\left(F(x_1,\ldots,x_k)G(x_{k+1},\ldots,x_{k+l})\prod_{i\leq k}^{j>k}\omega^m(x_j,x_i)\right)$$
 with
  $$\omega^m(x,y):=\frac{(x-q_1y)(x-q_2y)(x-q_3y)}{(x-y)^3}.$$
 This endows $\sS^m$ with a structure of an associative unital $\CC$-algebra
 with the unit ${\bf{1}}\in \sS^m_0$.

 We say that an element $\frac{f(x_1,\ldots,x_n)}{\Delta(x_1,\ldots,x_n)}\in \sS^m_n$
 satisfies the \emph{wheel conditions} if $f(x_1,\ldots,x_n)=0$ once
   $x_{i_1}/x_{i_2}=q_1\ \mathrm{and}\ x_{i_2}/x_{i_3}=q_2\ \mathrm{for\ some}\ 1\leq i_1,i_2,i_3\leq n$.
 Let $S^m\subset \sS^m$ be a $\ZZ_+$-graded subspace consisting of all such elements.
 The subspace $S^m$ is closed with respect to $\overset{m}\star$.

\begin{defn}
 The algebra $(S^m,\overset{m}\star)$ is called the \textit{small multiplicative shuffle algebra}.
\end{defn}

 Recall that $q_1,q_2,q_3$ are \emph{generic} if $q_1^aq_2^bq_3^c=1\Longleftrightarrow a=b=c$.
 We have the following result:

\begin{thm}\cite[Proposition 3.5]{N}\label{generation_m}
 The algebra $S^m$ is generated by $S^m_1$ for generic $q_1,q_2,q_3$.
\end{thm}

 The connection between the algebras $S^m$ and $\wt{\E}$ was established in~\cite{SV}:

\begin{prop}\cite{SV}\label{Hall_vs_shuffle}
 The map $u_{1,i}\mapsto x_1^i$ extends to an injective homomorphism $\wt{\E}^+\to S^m$,
 where $\wt{\E}^+$ is the subalgebra of $\wt{\E}$ generated by $\{u_{i,j}|i\in \NN, j\in \ZZ\}$.
\end{prop}

 Combining this result with Theorems~\ref{generation_m} and~\ref{Hall}, we get:

\begin{thm}
 The algebras $\wt{\E}^+, \ddot{U}^+, S^m$ are isomorphic.
\end{thm}

\subsection{Commutative subalgebra $\A^m\subset S^m$}
$\ $

 Following~\cite{FHHSY}, we introduce an important $\ZZ_+$-graded subspace
  $\A^m=\bigoplus_{n\geq 0} \A^m_n$ of $S^m$.  Its degree $n$ component is defined by
  $\A^m_n=\{F\in S^m_n| \partial^{(0,k)}F=\partial^{(\infty,k)}F\ \ \forall\ 1\leq k\leq n\}$, where
  $$\partial^{(0,k)}F:=\underset{\xi\to 0}\lim F(x_1,\ldots,\xi\cdot x_{n-k+1},\ldots,\xi\cdot x_n),\
    \partial^{(\infty,k)}F:=\underset{\xi\to \infty}\lim F(x_1,\ldots,\xi\cdot x_{n-k+1},\ldots,\xi\cdot x_n).$$

 This subspace satisfies the following properties:

\begin{thm}\cite[Section 2]{FHHSY}\label{comm_subalg_m}
 We have:

\noindent
 (a) Suppose $F\in S^m_n$ and $\partial^{(\infty,k)}F$ exist for all $1\leq k\leq n$, then $F\in \A^m_n$.

\noindent
 (b) The subspace $\A^m\subset S^m$ is $\overset{m}\star$-commutative.

\noindent
 (c) $\A^m$ is $\overset{m}\star$-closed and it is a
 polynomial algebra in $\{K^m_n\}_{n\in \NN}$ with $K^m_n\in S^m_n$ defined by
  $$K^m_1(x_1)=x_1^0,
    K^m_2(x_1,x_2)=\frac{(x_1-q_1x_2)(x_2-q_1x_1)}{(x_1-x_2)^2},
    K^m_n(x_1,\ldots,x_n)=\prod_{1\leq i<j\leq n} K^m_2(x_i,x_j).$$
\end{thm}

\begin{rem}
 These elements $\{K^m_n\}_{n\in \NN}$ played a crucial role in~\cite{FT}, where they were
 used to construct an action of the Heisenberg algebra on the vector space $M$ from Section 2.2.
\end{rem}

 Our next result provides an alternative choice of generators for the algebra $\A^m$
 explicitly expressed via $S^m_1$. We use the following notation:
 $[P,Q]_m=P\overset{m}\star Q-Q\overset{m}\star P$ for $P,Q\in S^m$.

\begin{prop}\label{alt_gener_m}
 The algebra $\A^m$ is a polynomial algebra in $\{L^m_n\}_{n\in \NN}$ defined by
  $$L^m_1(x_1)=x_1^0\ \ \mathrm{and}\ \
    L^m_n=\underbrace{[x^1,[x^0,[x^0,\ldots,[x^0,x^{-1}]_m\ldots]_m]_m]_m}_{\mbox{$n$ factors}}\in S^m_n
    \ \ \mathrm{for}\ n\geq 2.$$
\end{prop}

 We refer the reader to Appendix D for the proof of this result.

\subsection{Small additive shuffle algebra $S^a$}
$\ $

 Consider a $\ZZ_+$-graded $\CC$-vector space $\sS^a=\bigoplus_{n\geq 0}\sS^a_n$,
 where $\sS^a_n$ consists of rational functions $\frac{f(x_1,\ldots,x_n)}{\Delta(x_1,\ldots,x_n)}$
 with $f\in \CC[x_1,\ldots,x_n]^{\mathfrak{S}_n}$.
 Define the star product $\overset{a}\star:\sS^a_k\times \sS^a_l\to \sS^a_{k+l}$ by
   $$(F\overset{a}\star G)(x_1,\ldots,x_{k+l}):=
     \mathrm{Sym}_{\mathfrak{S}_{k+l}}\left(F(x_1,\ldots,x_k)G(x_{k+1},\ldots,x_{k+l})\prod_{i\leq k}^{j>k}\omega^a(x_j,x_i)\right)$$
 with
   $$\omega^a(x,y):=\frac{(x-y-h_1)(x-y-h_2)(x-y-h_3)}{(x-y)^3}.$$
 This endows $\sS^a$ with a structure of an associative unital $\CC$-algebra with the unit ${\bf{1}}\in \sS^a_0$.

 We say that an element $\frac{f(x_1,\ldots,x_n)}{\Delta(x_1,\ldots,x_n)}\in \sS^a_n$
 satisfies the \emph{wheel conditions} if $f(x_1,\ldots,x_n)=0$ once
   $x_{i_1}-x_{i_2}=h_1\ \mathrm{and}\ x_{i_2}-x_{i_3}=h_2\ \mathrm{for\ some}\ 1\leq i_1,i_2,i_3\leq n$.
 Let $S^a\subset \sS^a$ be a $\ZZ_+$-graded subspace consisting of all such elements.
 The subspace $S^a$ is closed with respect to $\overset{a}\star$.

\begin{defn}
 The algebra $(S^a,\overset{a}\star)$ is called the \textit{small additive shuffle algebra}.
\end{defn}

 The following result is proved analogously to Theorem~\ref{generation_m}:

\begin{thm}\label{generation_a}
 For generic $h_1,h_2,h_3$ (that is $ah_1+bh_2+ch_3=0\ \mathrm{for}\ a,b,c\in \ZZ\Longleftrightarrow a=b=c$),
 the map $e_j\mapsto x_1^j$ extends to an isomorphism $\ddot{Y}^+\iso S^a$.
 In particular, $S^a$ is generated by $S^a_1$.
\end{thm}

\subsection{Commutative subalgebra $\A^a\subset S^a$}
$\ $

 Let us introduce an \emph{additive version} of $\A^m$:
 a $\ZZ_+$-graded subspace $\A^a=\bigoplus_{n\geq 0} \A^a_n$ of  $S^a$.
 Its degree $n$ component $\A^a_n$ consists of those $F\in S^a_n$ such that the limits
   $$\partial^{(\infty,k)}F:=\underset{\xi\to \infty}\lim F(x_1,\ldots,x_{n-k},x_{n-k+1}+\xi,\ldots,x_n+\xi)$$
 exist for all $1\leq k\leq n$.
 The following is an \emph{additive} counterpart of Theorem~\ref{comm_subalg_m}:

\begin{thm}\label{comm_subalg_a}
 We have:

\noindent
 (a) The subspace $\A^a\subset S^a$ is $\overset{a}\star$-commutative.

\noindent
 (b) $\A^a$ is $\overset{a}\star$-closed and it is a
 polynomial algebra in $\{K^a_n\}_{n\in \NN}$ with $K^a_n\in S^a_n$ defined by
 $$K^a_1(x_1)=x^0_1,
   K^a_2(x_1,x_2)=\frac{(x_1-x_2-h_1)(x_2-x_1-h_1)}{(x_1-x_2)^2},
   K^a_n(x_1,\ldots,x_n)=\prod_{1\leq i<j\leq n}K^a_2(x_i,x_j).$$
\end{thm}

 Analogously to Proposition~\ref{alt_gener_m}, the commutative subalgebra $\A^a$ admits
 an alternative set of generators explicitly expressed via $S^a_1$.
 Define $[P,Q]_a:=P\overset{a}\star Q-Q\overset{a}\star P$ for $P,Q\in S^a$.

\begin{prop}\label{alt_gener_a}
 The algebra $\A^a$ is a polynomial algebra in $\{L^a_n\}_{n\in \NN}$ defined by
  $$L^a_1(x_1)=x_1^0\ \ \mathrm{and}\ \
    L^a_n=\underbrace{[x^0,[x^0,\ldots,[x^0,x^{n-1}]_a\ldots]_a]_a}_{\mbox{$n$ factors}}\in S^a_n\ \ \mathrm{for}\ n\geq 2.$$
\end{prop}

\begin{rem}
 The commutativity of $\{L^m_n\}$ and $\{L^a_n\}$ was shown in~\cite{SV},~\cite[Appendix E]{SV3}.
\end{rem}


\section{Horizontal realization and Whittaker vectors}


\subsection{Horizontal realization of $\ddot{U}_{q_1,q_2,q_3}(\gl_1)$}
$\ $

 Recall the distinguished collection of elements
 $\{u_\xx, \kappa_\yy\}\subset \ddot{U}_{q_1,q_2,q_3}(\gl_1)$ from Theorem~\ref{Hall}.
 Note that there is a natural $\mathrm{SL}_2(\ZZ)$-action on
  $\ddot{U}_{q_1,q_2,q_3}(\gl_1)/(\psi_0^{\pm 1}-1)\simeq \widetilde{\E}/(\kappa_\yy-1)_{\yy\in \ZZ^2}$.
 In particular, we have a natural automorphism of
 $\widetilde{\E}/(\kappa_\yy-1)_{\yy\in \ZZ^2}$ induced by $u_{k,l}\mapsto u_{-l,k}$.
 Although there is no such automorphism for $\widetilde{\E}/(\kappa_{0,1}-1)$,
 but we still have a nice presentation of this algebra in terms of
 the generators $\{u_{i,\pm 1},u_{j,0}, \kappa_{\pm 1,0}\}_{i\in \ZZ}^{j\in \ZZ^*}$ rather than
 $\{u_{\pm 1,i},u_{0,j}, \kappa_{\pm 1,0}\}_{i\in \ZZ}^{j\in \ZZ^*}$.

 To formulate the main result, we need to introduce a modification
 of the algebra $\ddot{U}_{q_1,q_2,q_3}(\gl_1)$.
 The algebra $\ddot{\bU}_{q_1,q_2,q_3}$ is the unital associative $\CC$-algebra
 generated by $\{\wt{e}_i,\wt{f}_i,\wt{\psi}_j, \gamma^{\pm 1/2}\}_{i\in \ZZ}^{j\in \ZZ^*}$
 with the following defining relations:
\begin{equation}\tag{TT0} \label{TT0}
  \wt{\psi}^\pm(z)\wt{\psi}^\pm(w)=\wt{\psi}^\pm(w)\wt{\psi}^\pm(z),\
  g(\gamma^{-1}z/w)\wt{\psi}^+(z)\wt{\psi}^-(w)=g(\gamma z/w)\wt{\psi}^-(w)\wt{\psi}^+(z),
\end{equation}
\begin{equation}\tag{TT1} \label{TT1}
  \wt{e}(z)\wt{e}(w)=g(z/w)\wt{e}(w)\wt{e}(z),
\end{equation}
\begin{equation}\tag{TT2}\label{TT2}
  \wt{f}(z)\wt{f}(w)=g(w/z)\wt{f}(w)\wt{f}(z),
\end{equation}
\begin{equation}\tag{TT3}\label{TT3}
  \beta_1\cdot [\wt{e}(z),\wt{f}(w)]=
  \delta(\gamma w/z)\wt{\psi}^{+}(\gamma^{1/2}w)-\delta(\gamma z/w)\wt{\psi}^{-}(\gamma^{1/2}z),
\end{equation}
\begin{equation}\tag{TT4}\label{TT4}
  \wt{\psi}^\pm(z)\wt{e}(w)=g(\gamma^{\pm 1/2}z/w)\wt{e}(w)\wt{\psi}^{\pm}(z),
\end{equation}
\begin{equation}\tag{TT5}\label{TT5}
  \wt{\psi}^{\pm}(z)\wt{f}(w)=g(\gamma^{\pm 1/2}w/z)\wt{f}(w)\wt{\psi}^{\pm}(z),
\end{equation}
\begin{equation}\tag{TT6}\label{TT6}
  \mathrm{Sym}_{\mathfrak{S}_3} [\wt{e}_{i_1},[\wt{e}_{i_2+1},\wt{e}_{i_3-1}]]=0,\ \
  \mathrm{Sym}_{\mathfrak{S}_3}  [\wt{f}_{i_1},[\wt{f}_{i_2+1},\wt{f}_{i_3-1}]]=0,
\end{equation}
 where these generating series are defined as follows:
   $$\wt{e}(z):=\sum_{i=-\infty}^{\infty}{\wt{e}_iz^{-i}},\
     \wt{f}(z):=\sum_{i=-\infty}^{\infty}{\wt{f}_iz^{-i}},\
     \wt{\psi}^{\pm}(z):=1+\sum_{j>0}{\wt{\psi}_{\pm j}z^{\mp j}},$$
 while $g(y):=\frac{(1-q_1y)(1-q_2y)(1-q_3y)}{(1-q_1^{-1}y)(1-q_2^{-1}y)(1-q_3^{-1}y)}$.
 Note that $g(y)=g(y^{-1})^{-1}$.

\medskip

  The following result is analogous to Theorem~\ref{Hall}:

\begin{thm}\label{Hall2}
 The assignment
   $$\kappa_{1,0}^{\pm 1/2}\mapsto \gamma^{\mp 1/2},\ \theta_{j,0}\mapsto \wt{\psi}_{-j},\
     u_{-i,1}\mapsto \gamma^{-|i|/2}\wt{e}_i,\ u_{-i,-1}\mapsto \gamma^{|i|/2}\wt{f}_i\ \
     \mathrm{for}\ i\in \ZZ, j\in \ZZ^*$$
 extends to an isomorphism of algebras
   $\Xi_{h}: \widetilde{\E}[\kappa_{1,0}^{\pm 1/2}]/(\kappa_{0,1}-1)\iso \ddot{\bU}_{q_1,q_2,q_3}$.
\end{thm}

 Following~\cite{DI}, we equip the algebra $\ddot{\bU}_{q_1,q_2,q_3}$ with a formal coproduct $\Delta_{h}$ determined by
  $$\Delta_h(\gamma^{\pm 1/2})=\gamma^{\pm 1/2}\otimes \gamma^{\pm 1/2},\
    \Delta_h(\wt{\psi}^\pm (z))=\wt{\psi}^\pm(\gamma^{\pm 1/2}_{(2)}z)\otimes \wt{\psi}^\pm(\gamma^{\mp 1/2}_{(1)}z),$$
  $$\Delta_h(\wt{e}(z))=\wt{e}(z)\otimes 1+\wt{\psi}^-(\gamma^{1/2}_{(1)}z)\otimes  \wt{e}(\gamma_{(1)}z),\
    \Delta_h(\wt{f}(z))=1\otimes \wt{f}(z)+\wt{f}(\gamma_{(2)}z)\otimes \wt{\psi}^{+}(\gamma^{1/2}_{(2)}z),$$
 where $\gamma^{\pm 1/2}_{(1)}:=\gamma^{\pm 1/2}\otimes 1$
 and $\gamma^{\pm 1/2}_{(2)}:=1\otimes \gamma^{\pm 1/2}$.

\begin{rem}
 According to Theorems~\ref{Hall} and~\ref{Hall2}, the algebras
 $\ddot{U}_{q_1,q_2,q_3}(\gl_1)[\psi_0^{\pm 1/2}]$ and $\ddot{\bU}_{q_1,q_2,q_3}$ are isomorphic.
 This allows to view $\Delta_h$ as a \emph{horizontal coproduct} on
 $\ddot{U}_{q_1,q_2,q_3}(\gl_1)[\psi_0^{\pm 1/2}]$, providing a monoidal structure on
 the category $\Oo$ from Section 4.6 (the category $\Oo$ for the bigger algebra with
 the central elements $\psi_0^{\pm 1/2}$ added is defined as before).
 For two $\ddot{U}_{q_1,q_2,q_3}(\gl_1)[\psi_0^{\pm 1/2}]$-modules $L_1,L_2$,
 we denote the corresponding tensor product by $L_1\underset{h}\otimes L_2$.
\end{rem}

\subsection{Horizontal realization of $F(u), V(u)$}$\ $
$\ $

 Identifying $\ddot{\bU}_{q_1,q_2,q_3}$ with $\ddot{U}_{q_1,q_2,q_3}(\gl_1)[\psi_0^{\pm 1/2}]$,
 let us describe how $\wt{e}(z), \wt{f}(z), \wt{\psi}^\pm(z)$ act on the Fock module $F(u)$.
 Consider the Heisenberg algebra $\mathfrak{h}$ generated by $\{a_n\}_{n\in \ZZ^*}$
 with the defining relation
   $$[a_m,a_n]=m(1-q_1^{|m|})/(1-q_2^{-|m|})\delta_{m,-n}.$$
 Let $\mathfrak{h}^+$ be the subalgebra generated by $\{a_n\}_{n\in \NN}$ and
   $\F:=\mathrm{Ind}_{\mathfrak{h}^+}^{\mathfrak{h}} \CC{\bf{1}}$
 be the Fock $\mathfrak{h}$-representation.
 On the other hand, note that $\{u_{j,0}\}_{j\in \ZZ^*}\subset \wt{\E}$
 also form a Heisenberg algebra and the highest weight vector $|\emptyset\rangle\in F(u)$
 is annihilated by $\{u_{j,0}\}_{j<0}$. The following is straightforward:

\begin{prop}\label{Fock_horizontal}
 There exists a unique isomorphism $F(u)\iso \F$
 such that $|\emptyset\rangle\mapsto {\bf{1}}$ and action of
  $\wt{e}(z), \wt{f}(z), \wt{\psi}^\pm(z), \gamma^{\pm 1/2}$ gets intertwined with
  $\rho_c(\wt{e}(z)), \rho_c(\wt{f}(z)),  \rho_c(\wt{\psi}^\pm(z)), \rho_c(\gamma^{\pm 1/2})$
 given by
   $$\rho_c(\gamma^{\pm 1/2})=q_3^{\pm 1/4},\
     \rho_c(\wt{\psi}^{\pm}(z))=\exp\left(\mp \sum_{n>0}\frac{1-q_2^{\mp n}}{n}(1-q_3^n)q_3^{-n/4}a_{\pm n}z^{\mp n}\right),$$
   $$\rho_c(\wt{e}(z))=
     c\exp\left(\sum_{n>0}\frac{1-q_2^n}{n}a_{-n}z^n\right)\exp\left(-\sum_{n>0} \frac{1-q_2^{-n}}{n}a_nz^{-n}\right),$$
   $$\rho_c(\wt{f}(z))=
     \frac{-q_1q_2c^{-1}}{(1-q_1)^2(1-q_2)^2}\exp\left(-\sum_{n>0}\frac{1-q_2^n}{n}q_3^{n/2}a_{-n}z^n\right)
     \exp\left(\sum_{n>0} \frac{1-q_2^{-n}}{n}q_3^{n/2}a_nz^{-n}\right),$$
 where $c=-u/(1-q_1)(1-q_2)$.
\end{prop}

 The aforementioned $\ddot{\bU}_{q_1,q_2,q_3}$-representations $\{\rho_c\}$ (acting on $\F$)
 were constructed in~\cite[Proposition A.6]{FHHSY}. Likewise, computing the action of
 $\wt{e}(z), \wt{f}(z), \wt{\psi}^\pm(z)$ on the vector representation $V(u)$,
 one recovers the formulas for the $\ddot{\bU}_{q_1,q_2,q_3}$-representations $\pi_c$
 (acting on $\CC[x,x^{-1}]$) with $c=-u/(1-q_1)(1-q_2)$ considered in~\cite[Proposition A.5]{FHHSY}.

\subsection{Correlation functions}
$\ $

 For a $\ddot{\bU}_{q_1,q_2,q_3}$-module $L$ and a pair $v\in L, w\in L^*$,
 we define the correlation function
   $$m_{w,v}(z_1,\ldots,z_n):=\langle w|\wt{e}(z_1)\cdots\wt{e}(z_n)|v \rangle\cdot \prod_{i<j}\omega^m(z_i,z_j).$$
 The relation (TT1) implies that $m_{w,v}(z_1,\ldots,z_n)$ is $\mathfrak{S}_n$--symmetric.

\begin{prop}\label{matrix_realization}
 For $L=\rho_c$ and $v={\bf{1}}\in L, w={\bf{1}}^*\in L^*$--the dual of ${\bf{1}}$, we have
  $$m_{\bf{1}^*,\bf{1}}(z_1,\ldots,z_n)=(-q_3)^{-n(n-1)/2}c^n\prod_{i<j}\frac{(z_i-q_3z_j)(z_j-q_3z_i)}{(z_i-z_j)^2}.$$
\end{prop}

\begin{proof} $\ $

 For $n>0$, we have
   $\exp(u\cdot a_n)\exp(v\cdot a_{-n})=\exp(v\cdot a_{-n})\exp(u\cdot a_n)\exp(uv\cdot  n(1-q_1^n)/(1-q_2^{-n}))$.
 Therefore
   $$\rho_c(\wt{e}(z_i))\rho_c(\wt{e}(z_j))=\
     :\rho_c(\wt{e}(z_i))\rho_c(\wt{e}(z_j)):\cdot \prod_{n>0}\exp\left(-\frac{(1-q_1^n)(1-q_2^n)}{n}(z_j/z_i)^n\right).$$
 It remains to use the equality
   $\prod_{n>0}\exp\left(-\frac{(1-q_1^n)(1-q_2^n)}{n}(z_j/z_i)^n\right)=
    \frac{(z_i-z_j)(z_i-q_1q_2z_j)}{(z_i-q_1z_j)(z_i-q_2z_j)}$.
\end{proof}

 In the more general case of $\rho_{c_1}\underset{h}\otimes\cdots\underset{h}\otimes\rho_{c_k}$, we have the following result:

\begin{prop}
 Consider
   ${\bar{\bf{1}}:=\bf{1}\otimes \cdots \otimes \bf{1}}\in \rho_{c_1}\underset{h}\otimes\cdots\underset{h}\otimes\rho_{c_k}$.
 Then
   $m_{\bar{\bf{1}}^*,\bar{\bf{1}}}(z_1,\ldots,z_n)\in \A^m.$
\end{prop}

\begin{proof}
$\ $

 Combining the formulas from Proposition~\ref{Fock_horizontal} with the formulas for $\Delta_h$, we find
   $$m_{\bar{\bf{1}}^*,\bar{\bf{1}}}(z_1,\ldots,z_n)=
     \sum_{f}c_{f(1)}\cdots c_{f(n)}\prod_{i<j}W_f(z_i,z_j)\prod_{i<j}\omega^m(z_i,z_j),$$
 where the sum is over all maps $f:\{1,\ldots,n\}\to \{1,\ldots,k\}$ and
 $W_f(z_i,z_j)$ is $1$ (if $f(i)>f(j)$), $\frac{(z_i-z_j)(z_i-q_1q_2z_j)}{(z_i-q_1z_j)(z_i-q_2z_j)}$
 (if $f(i)=f(j)$), and $g(z_i/z_j)$ (if $f(i)<f(j)$). The claim follows.
\end{proof}

\begin{rem}
 (a) This approach provides a \emph{Bethe algebra realization} of $\A^m$ (see~\cite{FT2}).

\noindent
 (b) According to~\cite[Proposition A.10]{FHHSY}, the correlation functions of
 $\pi_{c_1}\underset{h}\otimes \cdots \underset{h}\otimes \pi_{c_k}$
 are identified with the classical Macdonald difference operators.
\end{rem}

\subsection{Whittaker vector in K-theory}
$\ $

 Consider the \emph{Whittaker vector}
   $v^K_r:=\sum_{n\geq 0}{\left[\mathcal{O}_{M(r,n)}\right]}\in \widehat{M}^r,$
 where $\widehat{M}^r:=\prod_{n=1}^\infty M^r_n$.
 To state our main result, we introduce a family of the elements
 $\{K_n^{(m;j)}\}^{j\in \ZZ}_{n\in \NN}$ of $S^m$ defined by
   $$K^{(m;j)}_n(x_1,\ldots,x_n):=K^m_n(x_1,\ldots,x_n)x_1^j\cdots x_n^j=
     \prod_{1\leq a<b\leq n}\frac{(x_a-q_1x_b)(x_b-q_1x_a)}{(x_a-x_b)^2}\prod_{1\leq s\leq n} x_s^j.$$
 Let $\{K^{(m;j)}_{-n}\}^{j\in \ZZ}_{n\in \NN}$ be analogous elements
 in the opposite algebra $(S^m)^{\mathrm{opp}}$ (note $(S^m)^{\mathrm{opp}}\simeq\ddot{U}^-$).
 The name \emph{Whittaker} is motivated by the following result (see Appendix E for a proof):

\begin{theorem}\label{Whittaker_m}
 The vector $v^K_r$ is an eigenvector with respect to $\{K^{(m;j)}_{-n}|0\leq j\leq r, n>0\}$.
 More precisely: $K^{(m;j)}_{-n}(v^K_r)=C_{j,-n}\cdot v^K_r$, where
  $$C_{0,-n}=\frac{(-1)^nt^{(r-2)n}(-t_1)^{n(n-1)/2}}{(1-t_1)^n(1-t_2)(1-t_2^2)\cdots(1-t_2^n)},\
    C_{1,-n}=\ldots=C_{r-1,-n}=0,$$
  $$C_{r,-n}=\frac{(-t)^{(r-2)n}(-t_1t_2)^{n(n-1)/2}}{(1-t_1)^n(1-t_2)(1-t_2^2)\cdots(1-t_2^n)\cdot (\chi_1\cdots\chi_r)^n}.$$
\end{theorem}

\begin{rem}
 The subalgebra of $(S^m)^{\mathrm{opp}}$ generated by $\{K^{(m;j)}_{-n}\}^{0\leq j\leq r}_{n>0}$
 corresponds to the subalgebra of $\ddot{U}^-$ generated by
  $\{f_j, [f_{j+1},f_{j-1}],[f_{j+1},[f_j,f_{j-1}]],\cdots\}_{j=0}^r$,
 due to Proposition~\ref{alt_gener_m}.
\end{rem}

\subsection{Whittaker vector in cohomology}
$\ $

 Consider the \emph{Whittaker vector} $v^H_r:=\sum_{n\geq 0}[M(r,n)]\in \widehat{V}^r,$
 where $\widehat{V}^r:=\prod_{n=1}^\infty V^r_n$.
 To state our main result, we introduce a family of the elements
 $\{K_n^{(a;j)}\}_{n\in \NN}^{j\in \ZZ_+}$ of $S^a$ defined by
   $$K^{(a;j)}_n(x_1,\ldots,x_n):=K^a_n(x_1,\ldots,x_n)x_1^j\cdots x_n^j=
     \prod_{1\leq a<b\leq n}\frac{(x_a-x_b-h_1)(x_b-x_a-h_1)}{(x_a-x_b)^2}\prod_{1\leq s\leq n} x_s^j.$$
 Let $\{K^{(a;j)}_{-n}\}^{j\in \ZZ_+}_{n\in \NN}$ be analogous elements in the
 opposite algebra $(S^a)^{\mathrm{opp}}$. The following result has been already
 proved in~\cite{SV3} (we refer the reader to Appendix E for an alternative  proof).

\begin{theorem}\label{Whittaker_a}
 The vector $v^H_r$ is an eigenvector with respect to $\{K^{(a;j)}_{-n}|0\leq j\leq r, n>0\}$.
 More precisely: $K^{(a;j)}_{-n}(v^H_r)=D_{j,-n}\cdot v^H_r$,
 where $D_{r,-n}$ is a degree $n$ polynomial in $x_a$ and
  $$D_{0,-n}=\ldots=D_{r-2,-n}=0,\
    D_{r-1,-n}=\frac{(-1)^{n(n-1)/2}}{n!s_1^ns_2^n},\
    D_{r,-1}=\frac{-1}{s_1s_2}\sum_{a=1}^r x_a.$$
\end{theorem}

\begin{rem}
 According to Proposition~\ref{alt_gener_a}, the subalgebra of $(S^a)^{\mathrm{opp}}$ generated by
 $\{K^{(a;j)}_{-n}\}^{0\leq j\leq r}_{n>0}$ corresponds to the subalgebra of $\ddot{Y}^-$ generated by
 $\{f_j, [f_j,f_{j+1}],[f_j,[f_j,f_{j+2}]],\cdots\}_{j=0}^r$.
\end{rem}


\appendix


\section{Proof of Proposition~\ref{triangular_m}}

  We follow a standard argument.
 Consider a unital associative $\CC$-algebra $\ddot{V}_{q_1,q_2,q_3}(\gl_1)$
 generated by $\{e_i, f_i, \psi_i, \psi_0^{-1}|i\in \ZZ\}$ subject to the relations (T0,T3,T4,T5).
 Let $\ddot{V}^-, \ddot{V}^0, \ddot{V}^+$ be the subalgebras of $\ddot{V}_{q_1,q_2,q_3}(\gl_1)$
 generated by $\{f_i\},\ \{\psi_i,\psi_0^{-1}\},\ \mathrm{and}\ \{e_i\}$, respectively.
  Let $I^+$ and $I^-$ be the two-sided ideals of $\ddot{V}_{q_1,q_2,q_3}(\gl_1)$ generated by the
 quadratic and cubic relations in $e_i$ and $f_i$, respectively, arising from the relations (T1,T2,T6).
 Explicitly, $I^+$ is generated by
  $$A_{i,j}=e_{i+3}e_j-\varsigma_1e_{i+2}e_{j+1}+\varsigma_2e_{i+1}e_{j+2}-e_ie_{j+3}-
            e_je_{i+3}+\varsigma_2e_{j+1}e_{i+2}-\varsigma_1e_{j+2}e_{i+1}+e_{j+3}e_i,$$
  $$B_{i_1,i_2,i_3}=\Sym_{\mathfrak{S}_3}[e_{i_1}, [e_{i_2+1},e_{i_3-1}]],$$
 where $\varsigma_1:=q_1+q_2+q_3,\ \varsigma_2:=q_1^{-1}+q_2^{-1}+q_3^{-1}$ and
 $i,j,i_1,i_2,i_3\in \ZZ$.
  Let $J^\pm$ stay for the corresponding two-sided ideals of $\ddot{V}^\pm$.
 Our next result implies Proposition~\ref{triangular_m}.

\begin{lem}\label{triangular_m'}
  (a) Multiplication induces an isomorphism of vector spaces
    $$m:\ddot{V}^-\otimes \ddot{V}^0\otimes \ddot{V}^+\iso \ddot{V}_{q_1,q_2,q_3}(\gl_1).$$

\noindent
 (b) $\ddot{V}^-$ and $\ddot{V}^+$ are free associative algebras in $\{f_i\}$ and $\{e_i\}$, respectively,
  while $\ddot{V}^0$ is the algebra generated by $\{\psi_i, \psi_0^{-1}\}$ with the defining relations (T0).

\noindent
 (c) We have $I^+=m(\ddot{V}^-\otimes \ddot{V}^0\otimes J^+)$ and $I^-=m(J^-\otimes \ddot{V}^0\otimes \ddot{V}^+)$.
\end{lem}

\begin{proof}[Proof of Lemma~\ref{triangular_m'}]$\ $

  Parts (a) and (b) are standard.

  The first equality in part (c) is equivalent to $\ddot{V}^-\ddot{V}^0J^+$
 being a two-sided ideal of $\ddot{V}_{q_1,q_2,q_3}(\gl_1)$.
 To prove this, it suffices to show
  $[A_{i,j}, t_r], [B_{i_1,i_2,i_3}, t_r], [A_{i,j}, f_r], [B_{i_1,i_2,i_3}, f_r]\in  \ddot{V}^0J^+.$
 The first two commutators are just the linear combinations of
 $A_{i',j'}$ and $B_{i'_1, i'_2, i'_3}$, respectively, due to (T4t).
 Also $[A_{i,j}, f_r]=0$, due to the relations (T3) and (T4).

  To prove $[B_{i_1,i_2,i_3},f_r]\in \ddot{V}^0J^+$, we work with the generating series.
 The relation (T3) implies
  $$\beta_1\cdot [e(z_1)e(z_2)e(z_3),f(w)]=$$
  $$\delta\left(\frac{z_1}{w}\right)\psi(z_1)e(z_2)e(z_3)+
    \delta\left(\frac{z_2}{w}\right)\psi(z_2)e(z_1)e(z_3)g\left(\frac{z_1}{z_2}\right)+
    \delta\left(\frac{z_3}{w}\right)\psi(z_3)e(z_1)e(z_2)g\left(\frac{z_1}{z_3}\right)g\left(\frac{z_2}{z_3}\right),$$
 where $\psi(z)=\psi^+(z)-\psi^-(z)$.
 Hence, we have
  $$\left[\Sym_{\mathfrak{S}_3}
    \left\{\left(\frac{z_2}{z_1}+\frac{z_2}{z_3}-\frac{z_1}{z_2}-\frac{z_3}{z_2}\right)e(z_1)e(z_2)e(z_3)\right\},f(w)\right]=$$
  $$\beta_1^{-1}\left(\delta(z_1/w)\psi(z_1)C_1(z_2,z_3)+\delta(z_2/w)\psi(z_2)C_2(z_3,z_1)+\delta(z_3/w)\psi(z_3)C_3(z_1,z_2)\right),$$
 where $C_1(z_2,z_3)=e(z_2)e(z_3)C_{123}+e(z_3)e(z_2)C_{132}$ and
  $$C_{123}=
    \left(\frac{z_2}{z_1}+\frac{z_2}{z_3}-\frac{z_1}{z_2}-\frac{z_3}{z_2}\right)+
    g(z_2/z_1)\left(\frac{z_1}{z_2}+\frac{z_1}{z_3}-\frac{z_2}{z_1}-\frac{z_3}{z_1}\right)+$$
  $$g(z_2/z_1)g(z_3/z_1)\left(\frac{z_3}{z_1}+\frac{z_3}{z_2}-\frac{z_1}{z_3}-\frac{z_2}{z_3}\right),$$
  $$C_{132}=
   \left(\frac{z_3}{z_1}+\frac{z_3}{z_2}-\frac{z_1}{z_3}-\frac{z_2}{z_3}\right)+
    g(z_3/z_1)\left(\frac{z_1}{z_2}+\frac{z_1}{z_3}-\frac{z_2}{z_1}-\frac{z_3}{z_1}\right)+$$
  $$g(z_2/z_1)g(z_3/z_1)\left(\frac{z_2}{z_1}+\frac{z_2}{z_3}-\frac{z_1}{z_2}-\frac{z_3}{z_2}\right).$$

  The equality $C_{132}=-g(z_2/z_3)C_{123}$ implies that
 $C_1(z_2,z_3)$ is proportional to the generating function of $A_{i,j}$.
 The same holds for $C_2(z_3,z_1), C_3(z_1,z_2)$.
 This yields $[B_{i_1,i_2,i_3},f_r]\in \ddot{V}^0J^+$ for any $i_1,i_2,i_3,r\in \ZZ$.
 The second equality in part (c) is proved analogously.
\end{proof}


\section{Proofs of Theorems~\ref{Gieseker_m},~\ref{Gieseker_a}}

\subsection{Sketch of the proof of Theorem~\ref{Gieseker_m}}
$\ $

  We generalize the key technical result of~\cite{FT},
 required to prove Theorem~\ref{Gieseker_m} (all other arguments stay the same).
 Verification of the relations (T0,T1,T2,T6t) is straightforward
 and is based on Lemma~\ref{matrix_coeff_Gis_m}(a).
 Likewise, it is easy to check that the operators $[e_i,f_j]$ are diagonal
 in the fixed point basis and depend on $i+j$ only:
  $[e_i,f_j]([\bar{\lambda}])=\gamma_{i+j\mid_{\bar{\lambda}}}\cdot [\bar{\lambda}]$.

\begin{lem}\label{gamma_m}
  We have
   $$\gamma_{0\mid_{\bar{\lambda}}}=\frac{t^{-r}-t^r}{(1-t_1)(1-t_2)(1-t_3)},$$
   $$\gamma_{1\mid_{\bar{\lambda}}}=
     t^{-r}\left(\frac{1}{(1-t_1)(1-t_2)}\sum_{a=1}^r\chi_a^{-1}-\sum_{a=1}^r\sum_{\square\in\lambda^a}\chi(\square)\right).$$
\end{lem}

\begin{proof}$\ $

  Fix positive integers $L_a>(\lambda^{a*})_1$ for $1\leq a\leq r$.
  Applying Lemma~\ref{matrix_coeff_Gis_m}(a), we find
\begin{multline}\tag{$\heartsuit$}\label{LONG_m}
 \gamma_{s\mid_{\bar{\lambda}}}=\\
 \sum_{l=1}^r\sum_{j=1}^{L_l}\frac{T}{(1-t_1)^2}(\chi_j^{(l)})^{s-r}\cdot
 \frac{\chi^{(l)}_j(1-\chi^{(l)}_jt_1t_2^{1-L_l}\chi_l)}{\chi^{(l)}_j-t_2^{L_l}\chi_l^{-1}}\cdot
 \prod_{k\ne j}^{k\leq L_l}
 \frac{(\chi^{(l)}_j-t_1t_2\chi^{(l)}_k)(\chi^{(l)}_k-t_2\chi^{(l)}_j)}{(\chi^{(l)}_j-t_1\chi^{(l)}_k)(\chi^{(l)}_k-\chi^{(l)}_j)}\times \\
 \prod_{a\ne l}\left(\frac{\chi^{(l)}_j(1-\chi^{(l)}_jt_1t_2^{1-L_a}\chi_a)}{\chi^{(l)}_j-t_2^{L_a}\chi_a^{-1}} \cdot
 \prod_{k=1}^{L_a}
 \frac{(\chi^{(l)}_j-t_1t_2\chi^{(a)}_k)(\chi^{(a)}_k-t_2\chi^{(l)}_j)}{(\chi^{(l)}_j-t_1\chi^{(a)}_k)(\chi^{(a)}_k-\chi^{(l)}_j)}\right)-\\
 \sum_{l=1}^r\sum_{j=1}^{L_l}\frac{T}{(1-t_1)^2}(t_1\chi_j^{(l)})^{s-r}\cdot
 \frac{\chi^{(l)}_j(1-\chi^{(l)}_jt_1^2t_2^{1-L_l}\chi_l)}{\chi^{(l)}_j-t_1^{-1}t_2^{L_l}\chi_l^{-1}}\cdot
 \prod_{k\ne j}^{k\leq L_l}
 \frac{(\chi^{(l)}_k-t_1t_2\chi^{(l)}_j)(\chi^{(l)}_j-t_2\chi^{(l)}_k)}{(\chi^{(l)}_k-t_1\chi^{(l)}_j)(\chi^{(l)}_j-\chi^{(l)}_k)}\times \\
 \prod_{a\ne l}\left(\frac{\chi^{(l)}_j(1-\chi^{(l)}_jt_1^2t_2^{1-L_a}\chi_a)}{\chi^{(l)}_j-t_1^{-1}t_2^{L_a}\chi_a^{-1}} \cdot
 \prod_{k=1}^{L_a}
 \frac{(\chi^{(a)}_k-t_1t_2\chi^{(l)}_j)(\chi^{(l)}_j-t_2\chi^{(a)}_k)}{(\chi^{(a)}_k-t_1\chi^{(l)}_j)(\chi^{(l)}_j-\chi^{(a)}_k)}\right),
\end{multline}
 where $T=(-t)^{r-2}\chi_1^{-1}\cdots \chi_r^{-1}$ and
 $\chi^{(a)}_k=t_1^{\lambda^a_k-1}t_2^{k-1}\chi_a^{-1}$ as before.

 (i) For $s=0$, the right-hand side of~(\ref{LONG_m}) is a degree $0$ rational function
 in the variables $\chi^{(a)}_k$. It is easy to check that it has no poles, in fact.
 Therefore, it is an element of $\mathbb{F}_r$ independent of $\bar{\lambda}$.
 It suffices to compute its value at $\bar{\emptyset}$, the $r$-tuple of empty diagrams.
 For $\bar{\lambda}=\bar{\emptyset}$, we can choose $L_1=\ldots=L_r=1$,
 while $\chi^{(a)}_k=t_1^{-1}t_2^{k-1}\chi_a^{-1}$. Applying~(\ref{LONG_m}), we get
  $$\gamma_{0\mid_{\bar{\lambda}}}=\gamma_{0\mid_{\bar{\emptyset}}}=
    -\frac{T}{(1-t_1)^2}\sum_{l=1}^rt_1^{-r}\frac{1-t_1}{t_1^{-1}\chi_l^{-1}-t_2t_1^{-1}\chi_l^{-1}}
    \prod_{a\ne l}\frac{(t_1^{-1}\chi_a^{-1}-t_2\chi_l^{-1})(1-t_1\chi_a\chi_l^{-1})}{(t_1^{-1}\chi_l^{-1}-t_1^{-1}\chi_a^{-1})(t_1^{-1}\chi_a^{-1}-\chi_l^{-1})}=$$
  $$\frac{t^{r-2}}{(1-t_1)(1-t_2)}\sum_{l=1}^r\prod_{a\ne l}\frac{\chi_l-t_1t_2\chi_a}{\chi_l-\chi_a}=
    \frac{t^{r-2}}{(1-t_1)(1-t_2)}\cdot \frac{1-t_1^rt_2^r}{1-t_1t_2}=\frac{t^{-r}-t^r}{(1-t_1)(1-t_2)(1-t_3)},$$
 where we used the identity
  $\sum_{l=1}^r\prod_{a\ne l}\frac{\chi_l-u\chi_a}{\chi_l-\chi_a}=\frac{1-u^r}{1-u}$.
 The first equality follows.

 (ii) For $s=1$, the right-hand side of~(\ref{LONG_m}) is a degree $1$ rational function
 in the variables $\chi^{(a)}_k$. It is easy to check that it has no poles.
 Therefore, it is a linear function with the leading linear part
  $T\cdot (-1)^r\chi_1\ldots\chi_r\frac{t_1^{r}t_2^{r-1}}{1-t_1}\cdot \sum_{l=1}^r\sum_{j=1}^{L_l}\chi^{(l)}_j$.
 Hence, we have
  $$\gamma_{1\mid_{\bar{\lambda}}}=\frac{t^{-r}t_1}{1-t_1}\cdot\sum_{l=1}^r\sum_{j=1}^\infty\wt{\chi}^{(l)}_j+C$$
 for a constant $C\in \mathbb{F}_r$ independent of $\bar{\lambda}$, where
  $\wt{\chi}^{(l)}_j:=\chi^{(l)}_j-{t_1^{-1}t_2^{j-1}\chi_l^{-1}}$.
 Note that
   $$\sum_{a=1}^r\sum_{\square\in \lambda^a}\chi(\square)=
    \sum_{a=1}^r\sum_{j=1}^{(\lambda^{a*})_1}t_2^{j-1}(1+t_1+\ldots+t_1^{\lambda^a_j-1})\chi_a^{-1}=
    \sum_{l=1}^r\sum_j\frac{-t_1}{1-t_1}\wt{\chi}^{(l)}_j.$$
 On the other hand, we have $C=\gamma_{1\mid_{\bar{\emptyset}}}$.
 Applying~(\ref{LONG_m}), we find
  $$C=\gamma_{1\mid_{\bar{\emptyset}}}=-\frac{T}{(1-t_1)^2}\sum_{l=1}^rt_1^{-r}\frac{\chi_l^{-1}(1-t_1)}{t_1^{-1}\chi_l^{-1}-t_2t_1^{-1}\chi_l^{-1}}
    \prod_{a\ne l}\frac{(t_1^{-1}\chi_a^{-1}-t_2\chi_l^{-1})(1-t_1\chi_a\chi_l^{-1})}{(t_1^{-1}\chi_l^{-1}-t_1^{-1}\chi_a^{-1})(t_1^{-1}\chi_a^{-1}-\chi_l^{-1})}=$$
  $$\frac{t^{r-2}}{(1-t_1)(1-t_2)}\cdot \sum_{l=1}^r\frac{1}{\chi_l}\prod_{a\ne l}\frac{\chi_l-t_1t_2\chi_a}{\chi_l-\chi_a}=
    \frac{t^{r-2}t_1^{r-1}t_2^{r-1}}{(1-t_1)(1-t_2)}\sum_{l=1}^r\chi_l^{-1}=\frac{t^{-r}}{(1-t_1)(1-t_2)}\sum_{l=1}^r\chi_l^{-1},$$
 where we used
   $\sum_{l=1}^r\frac{1}{\chi_l}\prod_{a\ne l}\frac{\chi_l-u\chi_a}{\chi_l-\chi_a}=u^{r-1}\sum_{l=1}^r\chi_l^{-1}$.
 The second equality follows.
\end{proof}

 Next, we introduce the operator series
  $\phi^{\pm}(z)=\sum_{j=0}^{\infty}{\phi^{\pm}_jz^{\mp j}}\in \End(M^r)[[z^{\mp 1}]]$,
 diagonal in the fixed point basis with the eigenvalues given by
  $$\phi^+_{0\mid_{\bar{\lambda}}}=t^{-r},\ \phi^-_{0\mid_{\bar{\lambda}}}=t^r,\
    \phi^{\pm}_{j\mid_{\bar{\lambda}}}=\pm (1-t_1)(1-t_2)(1-t_3)\gamma_{\pm j\mid_{\bar{\lambda}}}\ \mathrm{for}\ j\in \NN.$$
 The following is a consequence of Lemma~\ref{gamma_m}:

\begin{cor}\label{propery of phi}
 We have
  $$[e(z),f(w)]=\frac{\delta(z/w)}{(1-t_1)(1-t_2)(1-t_3)}(\phi^{+}(w)-\phi^{-}(z)).$$
\end{cor}

 Our next result follows from the explicit computations in the fixed point basis.

\begin{lem}
 The series $\phi^\pm(z)$ satisfy the following relations:
\begin{equation}\label{4d}
 \phi^{\pm}(z)e(w)(z-t_1w)(z-t_2w)(z-t_3w)=-e(w)\phi^{\pm}(z)(w-t_1z)(w-t_2z)(w-t_3z),
\end{equation}
\begin{equation}\label{5d}
 \phi^{\pm}(z)f(w)(w-t_1z)(w-t_2z)(w-t_3z)=-f(w)\phi^{\pm}(z)(z-t_1w)(z-t_2w)(z-t_3w).
\end{equation}
\end{lem}

 Relation~(\ref{4d}) implies the following identity:
  $$\phi^+(z)_{\mid_{\bar{\lambda}+\square^l_j}}=\phi^+(z)_{\mid_{\bar{\lambda}}}\cdot
    \frac{(1-t_1^{-1}\chi(\square^l_j)/z)(1-t_2^{-1}\chi(\square^l_j)/z)(1-t_3^{-1}\chi(\square^l_j)/z)}
         {(1-t_1\chi(\square^l_j)/z)(1-t_2\chi(\square^l_j)/z)(1-t_3\chi(\square^l_j)/z)}.$$
 Therefore, we get
  $$\phi^+(z)_{\mid_{\bar{\lambda}}}=\phi^+(z)_{\mid_{\bar{\emptyset}}}\cdot \bc_r(z)^+_{\mid_{\bar{\lambda}}}.$$
 Applying~(\ref{LONG_m}) once again, we find
  $$\phi^+(z)_{\mid_{\bar{\emptyset}}}=(\phi^-_0+\sum_{i\geq 0}(1-t_1)(1-t_2)(1-t_3)\gamma_iz^{-i})_{\mid_{\bar{\emptyset}}}=$$
  $$t^r-t^r(1-t_1t_2)\sum_{l=1}^r\frac{1}{1-\chi_l^{-1}z^{-1}}\prod_{a\ne l}\frac{\chi_l-t_1t_2\chi_a}{\chi_l-\chi_a}=
    t^r\prod_{l=1}^r\left(\frac{1-t_1t_2\chi_lz}{1-\chi_lz}\right)^+,$$
 where we used the identity
   $$1-(1-u)\sum_{l=1}^r\frac{1}{1-1/(\chi_lz)}\prod_{a\ne l}\frac{\chi_l-u\chi_a}{\chi_l-\chi_a}=
     \prod_{l=1}^r\frac{u-\frac{1}{\chi_lz}}{1-\frac{1}{\chi_lz}}.$$
 This proves $\phi^+(z)=\psi^+(z)$. Analogous arguments also imply $\phi^-(z)=\psi^-(z)$.

 The relation (T3) follows from Corollary~\ref{propery of phi},
 while the relations (T4,T5) follow from~(\ref{4d},\ref{5d}).
 This completes our proof of Theorem~\ref{Gieseker_m}.

\subsection{Sketch of the proof of Theorem~\ref{Gieseker_a}}
$\ $

  The proof of Theorem~\ref{Gieseker_a} is completely analogous to its K-theoretical counterpart.
 Verification of the relations (Y0,Y1,Y2,Y6) is straightforward.
 Likewise, it is easy to check that $[e_i,f_j]$ are diagonal in the fixed point basis and depend on $i+j$ only:
  $[e_i,f_j]([\bar{\lambda}])=\gamma_{i+j\mid_{\bar{\lambda}}}\cdot [\bar{\lambda}]$.
 To verify the remaining relations, we will need the following generalization of Lemma~\ref{psi_0,1,2}:

\begin{lem}
  We have
   $\gamma_{0\mid_{\bar{\lambda}}}=\frac{-r}{s_1s_2},\
     \gamma_{1\mid_{\bar{\lambda}}}=\frac{1}{s_1s_2}\left(\sum_{a=1}^r x_a-\binom{r}{2}(s_1+s_2)\right),$
   $$\gamma_{2\mid_{\bar{\lambda}}}=2|\bar{\lambda}|-\frac{1}{s_1s_2}\left(\sum_{a=1}^r x_a^2-(r-1)(s_1+s_2)\sum_{a=1}^r x_a+\binom{r}{3}(s_1+s_2)^2\right).$$
\end{lem}

\begin{proof}$\ $

 Fix positive integers $L_a>(\lambda^{a*})_1$ for $1\leq a\leq r$.
 Applying Lemma~\ref{matrix_coeff_Gis_a}(a), we find
\begin{multline}\tag{$\spadesuit$}\label{LONG_a}
 \gamma_{s\mid_{\bar{\lambda}}}=\\
 \sum_{l=1}^r\sum_{j=1}^{L_l}\frac{1}{s_1^2}(x_j^{(l)})^s\cdot
 \frac{x^{(l)}_j+s_1+(1-L_l)s_2+x_l}{-x^{(l)}_j+L_ls_2-x_l}\cdot
 \prod_{k\ne j}^{k\leq L_l}
 \frac{(x^{(l)}_j-x^{(l)}_k-s_1-s_2)(x^{(l)}_k-x^{(l)}_j-s_2)}{(x^{(l)}_j-x^{(l)}_k-s_1)(x^{(l)}_k-x^{(l)}_j)}\times \\
 \prod_{a\ne l}\left(\frac{x^{(l)}_j+s_1+(1-L_a)s_2+x_a}{x^{(l)}_j-L_as_2+x_a} \cdot
 \prod_{k=1}^{L_a}
 \frac{(x^{(l)}_j-x^{(a)}_k-s_1-s_2)(x^{(a)}_k-x^{(l)}_j-s_2)}{(x^{(l)}_j-x^{(a)}_k-s_1)(x^{(a)}_k-x^{(l)}_j)}\right)-\\
 \sum_{l=1}^r\sum_{j=1}^{L_l}\frac{1}{s_1^2}(x_j^{(l)}+s_1)^s\cdot
 \frac{x^{(l)}_j+2s_1+(1-L_l)s_2+x_l}{-x^{(l)}_j+L_ls_2-s_1-x_l}\cdot
 \prod_{k\ne j}^{k\leq L_l}
 \frac{(x^{(l)}_k-x^{(l)}_j-s_1-s_2)(x^{(l)}_j-x^{(l)}_k-s_2)}{(x^{(l)}_k-x^{(l)}_j-s_1)(x^{(l)}_j-x^{(l)}_k)}\times \\
 \prod_{a\ne l}\left(\frac{x^{(l)}_j+2s_1+(1-L_a)s_2+x_a}{x^{(l)}_j-L_as_2+s_1+x_a} \cdot
 \prod_{k=1}^{L_a}
 \frac{(x^{(a)}_k-x^{(l)}_j-s_1-s_2)(x^{(l)}_j-x^{(a)}_k-s_2)}{(x^{(a)}_k-x^{(l)}_j-s_1)(x^{(l)}_j-x^{(a)}_k)}\right),
\end{multline}
 where $x^{(a)}_k=(\lambda^a_k-1)s_1+(k-1)s_2-x_a$ as before.
 The right-hand side of~(\ref{LONG_a}) is a degree $s$ rational function in the variables $x^{(a)}_k$.
 It is easy to see that it has no poles for $s\geq 0$.

 (i) For $s=0$, we therefore see that $\gamma_{0\mid_{\bar{\lambda}}}$ must be
 an element of $\CC(s_1,s_2,x_1,\ldots,x_l)$ independent of $\bar{\lambda}$.
 Evaluating~(\ref{LONG_a}) at $\bar{\emptyset}$, we find
  $\gamma_{0\mid_{\bar{\lambda}}}=\gamma_{0\mid_{\bar{\emptyset}}}=-r/s_1s_2$.

 (ii) For $s=1$, we therefore see that $\gamma_{1\mid_{\bar{\lambda}}}$ is a linear function.
 But its leading linear part is zero.
 So $\gamma_{1\mid_{\bar{\lambda}}}=\gamma_{1\mid_{\bar{\emptyset}}}$.
 Evaluating~(\ref{LONG_a}) at $\bar{\emptyset}$, we find
   $\gamma_{1\mid_{\bar{\lambda}}}=\gamma_{1\mid_{\bar{\emptyset}}}=\frac{1}{s_1s_2}\left(\sum_{a=1}^r x_a-\binom{r}{2}(s_1+s_2)\right)$.

 (iii) For $s=2$, we therefore see that $\gamma_{2\mid_{\bar{\lambda}}}$ is a quadratic function.
 But its leading quadratic part is zero.
 So $\gamma_{2\mid_{\bar{\lambda}}}$ is a linear function.
 Analogously to Lemma~\ref{psi_0,1,2},
 we find that the leading linear part is actually
  $\frac{2}{s_1}\sum_{a=1}^r\sum_{k=1}^\infty \wt{x}^{(a)}_k=2|\bar{\lambda}|$,
 where $\wt{x}^{(a)}_k:=x^{(a)}_k-(-s_1+(k-1)s_2-x_a)$. Hence,
  $\gamma_{2\mid_{\bar{\lambda}}}=2|\bar{\lambda}|+\gamma_{2\mid_{\bar{\emptyset}}}$.
 Applying~(\ref{LONG_a}) once again, we recover the last formula.
\end{proof}

  Define $\phi_j:=[e_j,f_0]\in \End(V^r)$.
 Explicit computations in the fixed point basis show that $\{\phi_j,e_j,f_j\}_{j\in \ZZ_+}$
 satisfy the relations (Y3,Y4,Y4$'$,Y5,Y5$'$) with $\psi_j$ replaced by $\phi_j$.
 This implies
  $$\phi(z)_{\mid_{\bar{\lambda}+\square^l_j}}=\phi(z)_{\mid_{\bar{\lambda}}}\cdot
    \frac{(z-\chi(\square^l_j)+s_1)(z-\chi(\square^l_j)+s_2)(z-\chi(\square^l_j)+s_3)}
         {(z-\chi(\square^l_j)-s_1)(z-\chi(\square^l_j)-s_2)(z-\chi(\square^l_j)-s_3)},$$
 where $\phi(z):=1+\sigma_3\sum_{j\geq 0} \phi_jz^{-j-1}$.
 Therefore,
  $\phi(z)_{\mid_{\bar{\lambda}}}=\phi(z)_{\mid_{\bar{\emptyset}}}\cdot \bC_r(z)_{\mid_{\bar{\lambda}}}.$
 Applying~(\ref{LONG_a}), we find
  $$\phi(z)_{\mid_{\bar{\emptyset}}}=
    1-\frac{\sigma_3}{s_1s_2}\sum_{i\geq 0}\sum_{l=1}^r(-x_l)^iz^{-i-1}\prod_{a\ne l}\frac{x_l-x_a-s_1-s_2}{x_l-x_a}=
    1-s_3\sum_{l=1}^r\frac{1}{z+x_l}\prod_{a\ne l}\frac{x_l-x_a+s_3}{x_l-x_a}.$$
 Combining this with
   $1+u\sum_{l=1}^r\frac{1}{z+x_l}\prod_{a\ne l}\frac{x_l-x_a-u}{x_l-x_a}=\prod_{l=1}^r\frac{z+x_l+u}{z+x_l}$,
 we finally get $\phi(z)=\psi(z)$.


\section{Proofs of Theorems~\ref{limit_1.1},~\ref{limit_2.1}}

\subsection{Proof of Theorem~\ref{limit_1.1}}
$\ $

  As pointed out in Section 5.4, all the defining relations of the algebra $\ddot{U}_{h_0}(\gl_1)$
 are of Lie-type. Therefore, it is a universal enveloping algebra of the Lie algebra $\ddot{u}_{h_0}$
 generated by $\{e_i, f_i, H_i\}_{i\in \ZZ}$ with the same defining relations.
 Moreover, $\ddot{u}_{h_0}$ is a $\CC\cdot H_0$--central extension of the Lie-algebra $\ul{\ddot{u}}_{h_0}$
 generated by $\{e_i,f_i,H_m\}_{i\in \ZZ}^{m\in \ZZ^*}$ with the following defining relations:
\begin{equation}\tag{u0}\label{u0}
 [H_k, H_m]=0,
\end{equation}
\begin{equation}\tag{u1}\label{u1}
 [e_{i+3},e_j]-(1+q+q^{-1})[e_{i+2},e_{j+1}]+(1+q+q^{-1})[e_{i+1}, e_{j+2}]-[e_i, e_{j+3}]=0,
\end{equation}
\begin{equation}\tag{u2}\label{u2}
 [f_{i+3},f_j]-(1+q+q^{-1})[f_{i+2},f_{j+1}]+(1+q+q^{-1})[f_{i+1}, f_{j+2}]-[f_i, f_{j+3}]=0,
\end{equation}
\begin{equation}\tag{u3}\label{u3}
 [e_i,f_j]=H_{i+j} \ \mathrm{for}\ {j\ne -i},\ \ [e_i,f_{-i}]=0,
\end{equation}
\begin{equation}\tag{u4}\label{u4}
 [H_m,e_i]=-(1-q^m)(1-q^{-m})e_{i+m},
\end{equation}
\begin{equation}\tag{u5}\label{u5}
 [H_m, f_i]=(1-q^m)(1-q^{-m})f_{i+m},
\end{equation}
\begin{equation}\tag{u6}\label{u6}
 [e_0,[e_1,e_{-1}]]=0,\ \ [f_0,[f_1,f_{-1}]]=0,
\end{equation}
 where $q=e^{h_0}\in \CC^*$.
 Note that $h_0\notin \mathbb{Q}\cdot \pi \sqrt{-1}\Longleftrightarrow q\ne \sqrt{1}$ ($q$ is not a root of $1$).

 Hence, it suffices to check that the corresponding homomorphism
  $\theta:\ul{\ddot{u}}_{h_0}\to \dd^0_{q}$
 defined by
  $$\theta: e_i\mapsto Z^iD,\ f_i\mapsto -D^{-1}Z^i,\ H_m\mapsto (q^{-m}-1)Z^m\ \
    \mathrm{for}\ i\in \ZZ, m\in \ZZ^*$$
 is an isomorphism of the $\CC$-Lie algebras for $q\ne \sqrt{1}$.

 The Lie algebra $\ul{\ddot{u}}_{h_0}$ is $\ZZ^2$-graded via
  $\deg(e_i):=(i,1), \deg(f_i):=(i,-1), \deg(H_m):=(m,0)$.
 The Lie algebra $\dd^0_{q}$ is also $\ZZ^2$-graded via $\deg(Z^iD^j)=(i,j)$.
 Note that $\theta$ is $\ZZ^2$-graded and surjective.
 Since $\dim (\dd^0_q)_{i,j}=1$ for $(i,j)\ne (0,0)$,
 it suffices to show that $\dim(\ul{\ddot{u}}_{h_0})_{i,j}\leq 1$.
 This is clear for $j=0$, while the proof for $j>0$ will occupy the rest of this section.

 Let $\ul{\ddot{u}}_{h_0}^{\geq}$ be the Lie algebra generated by
 $\{e_i,H_m\}_{i\in \ZZ}^{m\in \ZZ^*}$ with the defining relations (u0,u1,u4,u6).
 It suffices to show that
  $\dim(\ul{\ddot{u}}_{h_0}^{\geq})_{i,j}\leq 1\ \ \mathrm{for}\ i\in \ZZ, j\in \NN$.
 We prove this by an induction on $j$.

\noindent
 $\bullet$ \emph{Case $j=1$}.

 It is clear that $(\ul{\ddot{u}}_{h_0}^{\geq})_{N,1}$ is spanned by $e_N$.

\noindent
 $\bullet$ \emph{Case $j=2$}.

 It is clear that $(\ul{\ddot{u}}_{h_0}^{\geq})_{N,2}$ is spanned by $\{[e_{i_1},e_{i_2}]|i_1+i_2=N\}$.
 However,~(\ref{u1}) implies
   $$[e_{i+2+k}, e_{i+1-k}]=\frac{q^{k+1}-q^{-k}}{q-1}[e_{i+2},e_{i+1}],\
     [e_{i+2+k}, e_{i-k}]=\frac{q^{k+2}-q^{-k}}{q^2-1}[e_{i+2},e_i].$$
 These formulas can be unified in the following way:
\begin{equation}\label{j=2}
   [e_{i_1},e_{i_2}]=\frac{q^{i_2}-q^{i_1}}{q^{i_1+i_2}-1}[e_0,e_{i_1+i_2}]\ \mathrm{if}\ i_1+i_2\ne 0,\
   [e_i,e_{-i}]=\frac{q^{i+1}-q^{1-i}}{q^2-1}[e_1,e_{-1}].
\end{equation}
 Therefore, $(\ul{\ddot{u}}_{h_0}^{\geq})_{N,2}$ is either spanned by
 $[e_0,e_N]$ (if $N\ne 0$) or $[e_1,e_{-1}]$ (if $N=0$).

\noindent
 $\bullet$ \emph{Case $j=3$}.

 Let us introduce the \emph{length $n$ commutator}:
  $[a_1; a_2; \ldots;a_n]_n:=[a_1,[a_2,[\ldots[a_{n-1},a_n]\ldots]]]$
 (we will omit the subscript $n$ when the length is clear).
 The space $(\ul{\ddot{u}}_{h_0}^{\geq})_{N,3}$ is spanned by $\{[e_{i_1};e_{i_2};e_{i_3}]|i_1+i_2+i_3=N\}$.
 Using the automorphism $\pi$ of the Lie algebra $\ul{\ddot{u}}_{h_0}^{\geq}$ determined by
 $e_i\mapsto e_{i+1}, H_m\mapsto H_m$, we can assume $i_1,i_2,i_3\in \NN$.
 Due to the above $j=2$ case, it suffices to show that $[e_k;e_0;e_l]$
 is a multiple of $[e_0;e_0;e_{k+l}]$ for any $k,l\in \NN$.

 For $m\in \ZZ^*$, define $\hh_m:=-\frac{H_m}{(1-q^m)(1-q^{-m})}$, so that $\ad(\hh_m)(e_i)=e_{i+m}$.
 Set $A_1:=\ad(\hh_1)$. Then
  $$A_1([e_k;e_0;e_l])=[e_{k+1};e_0;e_l]+[e_k;e_1;e_l]+[e_k;e_0;e_{l+1}].$$
 Assuming $[e_k;e_0;e_l]$ is a multiple of $[e_0;e_0;e_{k+l}]$, we get
 $[e_{k+1};e_0;e_l]$ is a linear combination of $[e_0;e_0;e_{k+l+1}]$
 and $[e_1;e_0;e_{k+l}]$, due to~(\ref{j=2}).
 It remains to consider the $k=1$ case.

 We prove by an induction on $N>1$ that
  $[e_1;e_0;e_{N-1}]=\frac{(q^{N-1}-q^2)(q^{N-1}-1)}{(q^N-1)^2}[e_0;e_0;e_N]$.
 This is equivalent to $[e_1;e_0;e_{N-1}]$ being a multiple of $[e_0;e_0;e_N]$,
 since we can recover the constant $\lambda_{N,3}=\frac{(q^{N-1}-q^2)(q^{N-1}-1)}{(q^N-1)^2}$
 by comparing the images $\theta([e_1;e_0;e_{N-1}])$ and $\theta([e_0;e_0;e_N])$.

\noindent
 $\circ$ \emph{Case $N=2$}.

  Analogously to Proposition~\ref{Serre_0}, the relation~(\ref{u6}) combined with~(\ref{u4}) imply
\begin{equation}\tag{u6$'$}\label{u6'}
 \Sym_{\mathfrak{S}_3}[e_{i_1};e_{i_2+1};e_{i_3-1}]=0\ \ \mathrm{for\ any}\ i_1,i_2,i_3\in \ZZ.
\end{equation}
 Plugging in $i_1=1, i_2=1, i_3=0$, we get
  $[e_1;e_2;e_{-1}]+[e_1;e_1;e_0]+[e_0;e_2;e_0]=0.$
 Combining this equality with~(\ref{j=2}), we find
  $[e_0;e_0;e_2]=-\frac{(q+1)^2}{q}[e_1;e_0;e_1]\Longrightarrow [e_1;e_0;e_1]=\lambda_{2,3}[e_0;e_0;e_2]$.

\noindent
 $\circ$ \emph{Case $N=3$}.

 Plugging in $i_1=1, i_2=2, i_3=0$ into~(\ref{u6'}), we get
  $$[e_1;e_3;e_{-1}]+[e_2;e_2;e_{-1}]+[e_2;e_1;e_0]+[e_0;e_3;e_0]+[e_0;e_2;e_1]=0.$$
 Applying~(\ref{j=2}), we get
  $-(q+2+q^{-1})[e_2;e_0;e_1]-(q+q^{-1})[e_1;e_0;e_2]-(1+\frac{q^2-q}{q^3-1})[e_0;e_0;e_3]=0.$
 Meanwhile, applying $A_1$ to the above equality $(q+1)^2[e_1;e_0;e_1]+q[e_0;e_0;e_2]=0$, we find
  $(q+1)^2[e_2;e_0;e_1]+(q^2+3q+1)[e_1;e_0;e_2]+(q-\frac{q^2-q^3}{q^3-1})[e_0;e_0;e_3]=0.$
 Combining these two linear combinations of $[e_2;e_0;e_1], [e_1;e_0;e_2], [e_0;e_0;e_3]$, we get
  $[e_1;e_0;e_2]=0=\lambda_{3,3}[e_0;e_0;e_3]$.

\noindent
 $\circ$ \emph{Case $N=k+2, k>1$}.

  By the induction assumption $[e_1;e_0;e_k]-\lambda_{k+1,3}[e_0;e_0;e_{k+1}]=0$.
  Applying $A_1$, we get
   $$([e_2;e_0;e_k]+[e_1;e_1;e_k]+[e_1;e_0;e_{k+1}])-\lambda_{k+1,3}([e_1;e_0;e_{k+1}]+[e_0;e_1;e_{k+1}]+[e_0;e_0;e_{k+2}])=0.$$
  Consider the linear operator $A_2:=\frac{1}{2}(\ad(\hh_1)^2-\ad(\hh_2))$
  acting on $\ul{\ddot{u}}_{h_0}^{\geq}$. Then
   $$A_2([e_{i_1};e_{i_2};e_{i_3}])=[e_{i_1+1};e_{i_2+1};e_{i_3}]+[e_{i_1+1};e_{i_2};e_{i_3+1}]+[e_{i_1};e_{i_2+1};e_{i_3+1}].$$
  By the induction assumption $[e_1;e_0;e_{k-1}]-\lambda_{k,3}[e_0;e_0;e_{k}]=0$.
  Applying $A_2$, we get
   $$([e_2;e_1;e_{k-1}]+[e_2;e_0;e_k]+[e_1;e_1;e_k])-\lambda_{k,3}([e_1;e_1;e_k]+[e_1;e_0;e_{k+1}]+[e_0;e_1;e_{k+1}])=0.$$
  Applying~(\ref{j=2}), we find two linear combinations of $[e_2;e_0;e_k], [e_1;e_0;e_{k+1}], [e_0;e_0;e_{k+2}]$
  which are zero. It is a routine verification to check that they are
  not proportional for $q\ne \sqrt{1}$.  Therefore, we can eliminate $[e_2;e_0;e_k]$,
  which proves that $[e_1;e_0;e_{k+1}]$ is a multiple of $[e_0;e_0;e_{k+2}]$.

\noindent
 $\bullet$ \emph{Case $j=n>3$}.

 Analogously to the previous case, it suffices to show that
 $[e_1;e_0;\ldots;e_0;e_{N-1}]_n$ is a multiple of $[e_0;\ldots;e_0;e_N]_n$.
 This is equivalent to
  $$[e_1;\ldots;e_0;e_{N-1}]_n=\lambda_{N,n}\cdot [e_0;\ldots;e_0;e_N]_n\ \
    \mathrm{with}\ \lambda_{N,n}=\frac{(q^{N-1}-1)^{n-2}(q^{N-1}-q^{n-1})}{(q^N-1)^{n-1}},$$
 where $\lambda_{N,n}$ is computed by comparing the images of these length $n$ commutators under $\theta$.

  We will need the following generalization of~(\ref{u6}), which follows from Proposition~\ref{alt_gener_m}:
\begin{equation}\tag{u7n}\label{u7n}
 [e_0;e_1;e_0;\ldots;e_0;e_{-1}]_n=0.
\end{equation}
 Analogously to Proposition~\ref{Serre_0}, one can see that~(\ref{u7n}) combined with~(\ref{u4}) imply
\begin{equation}\tag{u7$'$n}\label{u7'n}
 \Sym_{\mathfrak{S}_n} [e_{i_1};e_{i_2+1};e_{i_3};\ldots;e_{i_{n-1}};e_{i_n-1}]_n=0
 \ \ \mathrm{for\ any}\ i_1,\ldots,i_n\in \ZZ.
\end{equation}

 Now we proceed to the proof of the aforementioned result by an induction on $N>1$.

\noindent
 $\circ$ \emph{Case $N=2$}.

 Applying $A_2$ to the equality~(\ref{u7n}), we get
  $$[e_1,A_1([e_1;e_0;\ldots;e_0;e_{-1}]_{n-1})]+[e_0;A_2([e_1;e_0;\ldots;e_0;e_{-1}]_{n-1})]=0.$$
 By the induction assumption on length $n-1$ commutators,
 this has a form $a_n\cdot[e_1;\ldots;e_0;e_1]_n+b_n\cdot[e_0;\ldots;e_0;e_2]_n=0$.
 Computing the images under $\theta$, we find $a_n=\frac{(-1)^{n}(1-q^{n-1})^2}{q^{n-2}(1-q)^2}\ne 0$.

\noindent
 $\circ$ \emph{Case $N=3$}.

 Applying $A_1$ to the equality $[e_1;\ldots;e_0;e_1]_n=\lambda_{2,n}[e_0;\ldots;e_0;e_2]_n$, we get
  $$[e_2;e_0;\ldots;e_1]_n+[e_1,A_1([e_0;\ldots;e_1]_{n-1})]=\lambda_{2,n}([e_1;\ldots;e_0;e_2]_n+[e_0,A_1([e_0;\ldots;e_2]_{n-1})]).$$
 By the induction assumption on length $n-1$ commutators,
 this equation can be simplified to
  $$[e_2;e_0;\ldots;e_0;e_1]_n+c_n\cdot [e_1;e_0;\ldots;e_0;e_2]_n+d_n\cdot [e_0;e_0;\ldots;e_0;e_3]_n=0.$$
 Computing the images under $\theta$, we find $c_n=\frac{(1-q)^{n-2}(1+2q-2q^{n-1}-q^n)}{(1-q^2)^{n-1}}$.

 Define the linear operator $A_3:=\ad(\hh_1)\ad(\hh_2)-\ad(\hh_3)$.
 Applying $A_3$ to~(\ref{u7n}), we get
  $$[e_2, \ad(\hh_1)[e_1;e_0;\ldots;e_0;e_{-1}]_{n-1}]+[e_1,\ad(\hh_2)[e_1;e_0;\ldots;e_0;e_{-1}]_{n-1}]+$$
  $$[e_0, A_3([e_1;e_0;\ldots;e_0;e_{-1}]_{n-1})]=0.$$
 By the induction assumption on length $n-1$ commutators, this equation can be simplified to
  $$a'_n\cdot [e_2;e_0;\ldots;e_0;e_1]_n+c'_n\cdot [e_1;e_0;\ldots;e_0;e_2]_n+d'_n\cdot [e_0;e_0;\ldots;e_0;e_3]_n=0.$$
 Computing the images under $\theta$, we find the following formulas
  $$a'_n=\frac{(-1)^n(1-q^{n-1})^2}{q^{n-2}(1-q)^2},\
    c'_n=\frac{(-1)^n(1-q)^{n-3}(1-q^{n-1})(1-q^{2n-2})}{q^{n-2}(1-q^2)^{n-1}}.$$
 It remains to notice that $c'_n\ne a'_nc_n$ for $q\ne \sqrt{1}$.
 Therefore, eliminating $[e_2;e_0;\ldots;e_0;e_1]_n$, we see that
 $[e_1;e_0;\ldots;e_0;e_2]_n$ is a multiple of $[e_0;e_0;\ldots;e_0;e_3]_n$.

\noindent
 $\circ$ \emph{Case $N=k+2, k>1$}.

 By the induction: $[e_1;e_0;\ldots;e_0;e_k]_n=\lambda_{k+1,n}[e_0;\ldots;e_0;e_{k+1}]_n$.
 Applying $A_1$, we get
  $$[e_2;e_0;\ldots;e_0;e_k]_n+[e_1,A_1([e_0;\ldots;e_0;e_k]_{n-1})]=$$
  $$\lambda_{k+1,n}([e_1;e_0;\ldots;e_0;e_{k+1}]_n+[e_0,A_1([e_0;\ldots;e_0;e_{k+1}]_{n-1})]).$$
 By the induction assumption on length $n-1$ commutators, this equality can be simplified to
  $$[e_2;e_0;\ldots;e_0;e_k]_n+ v_n\cdot [e_1;e_0;\ldots;e_0;e_{k+1}]_n+ w_n\cdot [e_0;\ldots;e_0;e_{k+2}]_n=0.$$
 Computing the images under $\theta$, we find
  $v_n=\frac{(1-q^k)^{n-2}(q^{n+k}-2q^{k+1}-2q^{n-1}+q^n+q^k+1)}{(1-q^{k+1})^{n-1}(1-q)}$.

 On the other hand, by the induction assumption
  $[e_1;\ldots;e_0;e_{k-1}]_n=\lambda_{k,n}[e_0;\ldots;e_0;e_k]_n$.
 Applying $A_2$ to this equality, we find
  $$[e_2, A_1([e_0;\ldots;e_0;e_{k-1}]_{n-1})]+[e_1,A_2([e_0;\ldots;e_0;e_{k-1}]_{n-1})]=$$
  $$\lambda_{k,n}([e_1,A_1([e_0;\ldots;e_0;e_k]_{n-1})]+[e_0,A_2([e_0;\ldots;e_0;e_k]_{n-1})]).$$
 By the induction assumption on length $n-1$ commutators, this equality can be simplified to
   $$u'_n\cdot [e_2;e_0;\ldots;e_0;e_k]_n+ v'_n\cdot [e_1;e_0;\ldots;e_0;e_{k+1}]_n+ w'_n\cdot [e_0;\ldots;e_0;e_{k+2}]_n=0.$$
 Computing the images under $\theta$, we find
  $$u'_n=\frac{(1-q^{n-1})(1-q^{k-1})^{n-2}}{(1-q)(1-q^k)^{n-2}},\
    v'_n=\frac{(1-q^{k-1})^{n-2}(1-q^{n-1})}{(1-q^{k+1})^{n-2}(1-q)}\left(\frac{q-q^{n-1}}{1-q^2}+\frac{q^{k-1}-q^{n-1}}{1-q^k}\right).$$
 Since $v'_n\ne u'_nv_n$ for $q\ne \sqrt{1}$, we can eliminate $[e_2;e_0;\ldots;e_0;e_k]_n$,
 so that $[e_1;e_0;\ldots;e_0;e_{k+1}]_n$ is a multiple of $[e_0;\ldots;e_0;e_{k+2}]_n$.
 This completes our proof of $\dim(\ul{\ddot{u}}_{h_0}^{\geq})_{i,j}\leq 1$ for $j>0$.

\subsection{Proof of Theorem~\ref{limit_2.1}}$\ $
$\ $

 As pointed out in Section 5.5, all the defining relations of the algebra
 $\ddot{Y}_{h_0}(\gl_1)$ are of Lie-type. Therefore, it is a universal
 enveloping algebra of the Lie algebra $\ddot{y}_{h_0}$ generated by
 $\{e_j, f_j, \psi_j\}_{j\in \ZZ_+}$ with the same defining relations.
 Moreover, $\ddot{y}_{h_0}$ is a $\CC\cdot \psi_0$-central extension of the Lie-algebra $\ul{\ddot{y}}_{h_0}$
 generated by $\{e_j,f_j,\psi_{j+1}\}_{j\in \ZZ_+}$ with the following defining relations:
\begin{equation}\tag{y0}\label{y0}
 [\psi_k, \psi_l]=0,
\end{equation}
\begin{equation}\tag{y1}\label{y1}
 [e_{i+3},e_j]-3[e_{i+2},e_{j+1}]+3[e_{i+1}, e_{j+2}]-[e_i, e_{j+3}]-h_0^2([e_{i+1},e_j]-[e_i,e_{j+1}])=0,
\end{equation}
\begin{equation}\tag{y2}\label{y2}
 [f_{i+3},f_j]-3[f_{i+2},f_{j+1}]+3[f_{i+1}, f_{j+2}]-[f_i, f_{j+3}]-h_0^2([f_{i+1},f_j]-[f_i,f_{j+1}])=0,
\end{equation}
\begin{equation}\tag{y3}\label{y3}
 [e_0,f_0]=0,\ [e_i,f_j]=\psi_{i+j}\ \mathrm{for}\ i+j>0,
\end{equation}
\begin{equation}\tag{y4}\label{y4}
 [\psi_{i+3},e_j]-3[\psi_{i+2},e_{j+1}]+3[\psi_{i+1}, e_{j+2}]-[\psi_i, e_{j+3}]-h_0^2([\psi_{i+1},e_j]-[\psi_i,e_{j+1}])=0,
\end{equation}
\begin{equation}\tag{y4$'$}\label{y4'}
 [\psi_1,e_j]=0,\ [\psi_2,e_j]=2h_0^2e_j,
\end{equation}
\begin{equation}\tag{y5}\label{y5}
 [\psi_{i+3},f_j]-3[\psi_{i+2},f_{j+1}]+3[\psi_{i+1},f_{j+2}]-[\psi_i,f_{j+3}]-h_0^2([\psi_{i+1},f_j]-[\psi_i,f_{j+1}])=0,
\end{equation}
\begin{equation}\tag{y5$'$}\label{y5'}
 [\psi_1,f_j]=0,\ [\psi_2,f_j]=-2h_0^2f_j,
\end{equation}
\begin{equation}\tag{y6}\label{y6}
 \Sym_{\mathfrak{S}_3}[e_{i_1},[e_{i_2},e_{i_3+1}]]=0,\ \ \Sym_{\mathfrak{S}_3}[f_{i_1},[f_{i_2},f_{i_3+1}]]=0.
\end{equation}

 Hence, it suffices to check that the corresponding homomorphism
  $\vartheta:\ul{\ddot{y}}_{h_0}\to \D_{h_0}$
 defined by
  $$\vartheta: e_j\mapsto x^j\partial,\ f_j\mapsto -\partial^{-1}x^j,\
    \psi_{j+1}\mapsto ((x-h_0)^{j+1}-x^{j+1})\partial^0\ \ \mathrm{for}\ j\in \ZZ_+$$
 is an isomorphism of the $\CC$-Lie algebras for $h_0\ne 0$.
 The surjectivity of $\vartheta$ is clear.

 The Lie algebra $\ul{\ddot{y}}_{h_0}$ is $\ZZ$-graded via
  $\deg_2(e_j):=1, \deg_2(f_j):=-1, \deg_2(\psi_{j+1}):=0$
 and $\ZZ_+$-filtered as a quotient of the free algebra
 $\CC\langle e_j,f_j,\psi_{j+1} \rangle$ graded via
  $\deg_1(e_j):=j, \deg_1(f_j):=j, \deg_1(\psi_{j+1}):=j$.
 The Lie algebra $\D_{h_0}$ is also $\ZZ$-graded via $\deg_2(x^i\partial^j)=j$
 and $\ZZ_+$-filtered as a quotient of $\CC\langle x,\partial^{\pm 1}\rangle$ with
  $\deg_1(x)=1, \deg_1(\partial^{\pm 1})=0$.
 Note that $\vartheta$ is $\ZZ$-graded and preserves the $\ZZ_+$-filtration, while
  $\dim (\D_{h_0})_{\leq i,j}=\dim (\D_{h_0})_{\leq i-1,j}+1$.
 Hence, it suffices to prove
  $\dim(\ul{\ddot{y}}_{h_0})_{\leq i,j}-\dim(\ul{\ddot{y}}_{h_0})_{\leq i-1,j}\leq 1$.
 This is clear for $j=0$, while we consider $j>0$ below.

 Let $\ul{\ddot{y}}_{h_0}^{\geq}$ be the Lie algebra generated by $\{e_j,\psi_{j+1}\}_{j\geq 0}$
 subject to the relations (y0,y1,y4,y4$'$,y6). It suffices to prove
 $\dim(\ul{\ddot{y}}_{h_0}^{\geq})_{\leq i,j}-\dim(\ul{\ddot{y}}_{h_0}^{\geq})_{\leq i-1,j}\leq 1.$
 Let $W(n;N)$ be the subspace of $\ul{\ddot{y}}_{h_0}^{\geq}$ spanned by
  $\{[e_0;\ldots;e_0;e_M]_n| 0\leq M\leq N\}$.
 Let $V(n;N)$ be the subspace of $\ul{\ddot{y}}_{h_0}^{\geq}$ spanned by
  $\{[e_{i_1};\ldots;e_{i_n}]_n|i_1+\ldots+i_n\leq N\}$.
 Given $x,y\in V(n;N)$, we write $x\sim y$ if $x-y\in V(n;N-1)$.
 Given $x\in W(n;N)$, we write $x\equiv \nu\cdot [e_0;\ldots;e_N]_n$ if
 $x-\nu\cdot [e_0;\ldots;e_N]_n\in W(n;N-1)$.

 The required estimate on dimensions of $(\ul{\ddot{y}}_{h_0}^\geq)_{\leq i,j}$ follows from the following result:

\begin{prop}
 We have $V(n;N)=W(n;N)$ for any $n,N\in \NN$.
\end{prop}

\begin{proof}
$\ $

 It is clear that $W(n;N)\subset V(n;N)$. We prove $V(n;N)\subset W(n;N)$ by an induction on $n$.

\noindent
 $\bullet$ \emph{Case $n=1,2$}.

 The case $n=1$ is obvious. Let us now consider the case $n=2$.
 The relation~(\ref{y1}) implies
  $$[e_{i+2+k},e_{i+1-k}]\sim (2k+1)[e_{i+2},e_{i+1}]\ \mathrm{and}\
    [e_{i+2+k},e_{i-k}]\sim (k+1)[e_{i+2},e_i].$$
 These formulas can be unified in the following way:
\begin{equation}\label{j=2'}
  [e_i,e_j]\sim \frac{j-i}{i+j}[e_0,e_{i+j}]\ \ \mathrm{for}\ i,j\in \NN.
\end{equation}
 Assuming by induction $V(2;i+j-1)\subset W(2;i+j-1)$, we find $V(2;i+j)\subset W(2;i+j)$.

\noindent
 $\bullet$ \emph{Case $n=3$}.

 Define $\hh_1:=\frac{\psi_3}{6h_0^2}\ \mathrm{and}\ \hh_2:=\frac{\psi_4-h_0^2\psi_2}{12h_0^2}$,
 so that $[\hh_1,e_j]=e_{j+1}\ \mathrm{and}\ [\hh_2,e_j]=e_{j+2}$, due to~(\ref{y4},\ref{y4'}).
 Consider the linear operators
  $A_1:=\ad(\hh_1), A_2:=\frac{1}{2}(\ad(\hh_1)^2-\ad(\hh_2)) \in \End(\ul{\ddot{y}}^{\geq})$.

 By the induction assumption for $n=2$, it suffices to prove $[e_k;e_0;e_l]\in W(3;k+l)$.
 Assuming by induction on $k$ that $[e_k;e_0;e_l]=\sum_{M=0}^{k+l} \nu_M[e_0;e_0;e_M]$
 and applying $A_1$ to this equality, we find (as in Appendix C.1) that
 $[e_{k+1};e_0;e_l]$ can be expressed as a linear combination of
 $\{[e_0;e_0;e_M]\}_{M\leq k+l+1}\cup \{[e_1;e_0;e_M]\}_{M\leq k+l}$.
 Therefore, it remains to prove $[e_1;e_0;e_{N-1}]\in W(3;N)$.
 This is equivalent to $[e_1;e_0;e_{N-1}]\equiv \frac{N-4}{N}[e_0;e_0;e_N]$, where
 the constant $\beta_{N,3}=\frac{N-4}{N}$ can be recovered by comparing
 $\vartheta([e_1;e_0;e_{N-1}])$ and $\vartheta([e_0;e_0;e_N])$.

\noindent
 $\circ$ \emph{Case $N=1,2$}.

 We have $[e_1;e_0;e_0]=0=[e_0;e_0;e_1]$.
 Applying $A_1$, we also find $[e_1;e_0;e_1]=-[e_0;e_0;e_2]$.

\noindent
 $\circ$ \emph{Case $N=k+1, k>1$}.

 By the induction assumption: $[e_1;e_0;e_{k-1}]\equiv \beta_{k,3}[e_0;e_0;e_k]$.
 Applying $A_1$, we find
  $$[e_2;e_0;e_{k-1}]+[e_1;e_1;e_{k-1}]+[e_1;e_0;e_k]\equiv \beta_{k,3}([e_1;e_0;e_k]+[e_0;e_1;e_k]+[e_0;e_0;e_{k+1}]).$$
 Applying~(\ref{j=2'}), we get $[e_2;e_0;e_{k-1}]+\frac{k+2}{k}[e_1;e_0;e_k]\in W(3;k+1)$.
 On the other hand, applying $A_2$ to $[e_1;e_0;e_{k-2}]\equiv \beta_{k-1,3}[e_0;e_0;e_{k-1}]$, we find
  $$[e_2;e_1;e_{k-2}]+[e_2;e_0;e_{k-1}]+[e_1;e_1;e_{k-1}]\equiv \beta_{k-1,3}([e_1;e_1;e_{k-1}]+[e_1;e_0;e_k]+[e_0;e_1;e_k]).$$
 Applying~(\ref{j=2'}), we get $\frac{2(k-2)}{k-1}[e_2;e_0;e_{k-1}]+\frac{8-k}{k}[e_1;e_0;e_k]\in W(3;k+1)$.
 Comparing those two linear combinations of $[e_2;e_0;e_{k-1}], [e_1;e_0;e_k]$,
 we find $[e_1;e_0;e_k]\in W(3;k+1)$ unless $k=3$.
 The latter case will be considered in the greater generality below.

\noindent
 $\bullet$ \emph{Case $n>3$}.

 Analogously to the previous case, it suffices to show that $[e_1;e_0;\ldots;e_0;e_{N-1}]_n\in W(n;N)$,
 which is equivalent to  $[e_1;\ldots;e_0;e_{N-1}]_n\equiv \beta_{N,n}\cdot [e_0;\ldots;e_0;e_N]_n$
 with $\beta_{N,n}=\frac{N-2n+2}{N}$.

 We will need the following generalization of~(\ref{y6}),
 which follows from Proposition~\ref{alt_gener_a}:
\begin{equation}\tag{y7n}\label{y7n}
 [e_0;\ldots;e_0;e_{n-2}]_n=0.
\end{equation}

 Now we proceed to the proof of the aforementioned result by an induction on $N$.

\noindent
 $\circ$ \emph{Case $N\leq n-1$}.

 If $N<n-1$, then $[e_0;\ldots;e_0;e_{N-1}]_{n-1}=0=[e_0;\ldots;e_0;e_N]_n$.

 Applying $A_1$ to $[e_0;\ldots;e_0;e_{n-2}]_n=0$, we find $[e_1;\ldots;e_0;e_{n-2}]_n\in W(n;n-1)$.

\noindent
 $\circ$ \emph{Case $N=k+1, k>n-2$}.

 Applying $A_1$ to $[e_1;\ldots;e_0;e_{k-1}]_n\equiv \beta_{k,n}[e_0;\ldots;e_0;e_k]_n$, we find
   $$[e_2;\ldots;e_0;e_{k-1}]_n+[e_1,A_1([e_0;\ldots;e_{k-1}]_{n-1})]\equiv
     \beta_{k,n}([e_1;\ldots;e_0;e_k]_n+[e_0,A_1([e_0;\ldots;e_k]_{n-1})]).$$
 Combining this with the induction assumption for length $n-1$ commutators, we get
\begin{equation}\label{auxilary}
  k[e_2;e_0;\ldots;e_0;e_{k-1}]_n+((n-2)k-(n-1)(n-4))[e_1;e_0;\ldots;e_0;e_k]_n\in W(n;k+1).
\end{equation}
 Applying $A_2$ to $[e_1;\ldots;e_0;e_{k-2}]_n\equiv \beta_{k-1,n}[e_0;\ldots;e_0;e_{k-1}]_n$ and
 using the induction assumption, we analogously find
  $P[e_2;e_0;\ldots;e_0;e_{k-1}]_n+Q[e_1;e_0;\ldots;e_0;e_k]_n\in W(n;k+1),$
 where $P=\frac{(n-1)(k-n+1)}{k-1},Q=\frac{n-1}{2k(k-1)}(k^2(n-4)-k(2n^2-13n+12)+(n^3-9n^2+18n-8))$.
 Comparing those two linear combinations, we get
  $[e_1;e_0;\ldots;e_0;e_k]_n\in W(n;k+1)$ for $k\ne n$.

 It remains to consider the case $k=n$.
 Define $\hh_3:=\frac{\psi_5}{20h_0^2}-\frac{\psi_3}{12}$, so that $[\hh_3,e_j]=e_{j+3}$.
 Applying $\ad(\hh_1)\ad(\hh_2)-\ad(\hh_3)$ to $[e_0;\ldots;e_0;e_{n-2}]_n=0$, we find
  $$n[e_2;\ldots;e_0;e_{n-1}]_n+2[e_1;\ldots;e_0;e_n]_n\in W(n;n+1).$$
 Combining this with~(\ref{auxilary}), we get $[e_1;\ldots;e_0;e_n]_n\in W(n;n+1)$.
 This completes our proof.
\end{proof}


\section{Proof of Proposition~\ref{alt_gener_m}}

  Proposition~\ref{alt_gener_m} follows from Theorem~\ref{comm_subalg_m}(c) and the following two lemmas:

\begin{lem}\label{step 1.1}
  The elements $\{L^m_n\}_{n\in \NN}$ belong to $\A^m$.
\end{lem}

\begin{lem}\label{step 2.1}
 The elements $\{L^m_n\}_{n\in \NN}$ are algebraically independent.
\end{lem}

\begin{proof}[Proof of Lemma~\ref{step 1.1}]$\ $

 According to Theorem~\ref{comm_subalg_m}(a), it suffices to prove that
 $\partial^{(\infty,k)}L^m_n$ exist for all $k$. We have
\begin{equation}\label{Lmn}
  L^m_n=\mathrm{Sym}_{\mathfrak{S}_n}
  \left\{\left(\sum_{l=0}^{n-2}(-1)^l\binom{n-2}{l}\frac{x_1}{x_{n-l}}-\sum_{l=0}^{n-2}(-1)^l\binom{n-2}{l}\frac{x_n}{x_{n-l-1}}\right)
         \prod_{i<j}\omega^m(x_j,x_i)\right\}.
\end{equation}
 Our goal is to show that the RHS of~(\ref{Lmn}) has a finite limit
 as $x_{n-k+1}\mapsto \xi\cdot x_{n-k+1},\ldots,x_n\mapsto \xi\cdot x_n$ with $\xi\to \infty$.
 Note that $\frac{x_{\sigma(i)}}{x_{\sigma(j)}}$ has a finite limit as $\xi \to \infty$
 unless $\sigma(j)\leq n-k<\sigma(i)$, while it has a linear growth in the latter case.
 On the other hand, $\omega^m(x_j,x_i)$ has a finite limit as $\xi\to \infty$ for any $i,j$.
 Moreover:
  $\omega^m(\xi\cdot x,y)=1+O(\xi^{-1}),\  \omega^m(y,\xi\cdot x)=1+O(\xi^{-1})\ \mathrm{as}\ \xi\to \infty.$
 Therefore, it remains to prove the equality $A_1=A_2$, where $A_1, A_2$ are given by
  $$A_1:=\sum_{s=n-k+1}^n \sum_{\sigma\in \mathfrak{S}_n}^{\sigma(1)=s}\sum_l^{\sigma(n-l)\leq n-k}(-1)^l\binom{n-2}{l}
         \frac{x_s}{x_{\sigma(n-l)}} \prod_{i<j}^{j\leq n-k}\omega_\sigma^m(x_j,x_i)\prod_{i<j}^{n-k<i}\omega_\sigma^m(x_j,x_i),$$
  $$A_2:=\sum_{s=n-k+1}^n \sum_{\sigma\in \mathfrak{S}_n}^{\sigma(n)=s}\sum_l^{\sigma(n-l-1)\leq n-k}(-1)^l\binom{n-2}{l}
         \frac{x_s}{x_{\sigma(n-l-1)}} \prod_{i<j}^{j\leq n-k}\omega_\sigma^m(x_j,x_i)\prod_{i<j}^{n-k<i}\omega_\sigma^m(x_j,x_i).$$
 Here we set
  $\omega_\sigma^m(x_j,x_i)=\omega^m(x_j,x_i)$ if $\sigma^{-1}(i)<\sigma^{-1}(j)$ and
  $\omega_\sigma^m(x_j,x_i)=\omega^m(x_i,x_j)$ otherwise.

 If $k=1$, then $s=n$ in both sums and the map $(\sigma,l)\mapsto (\sigma',l)$ with
 $\sigma'(i):=\sigma(i+1)$ (for $1\leq i\leq n-1$) establishes
 a bijection between equal summands in $A_1$ and $A_2$, so that $A_1=A_2$.

 For $k>1$, there is no such bijection. Instead, we prove $A_1=0$ (the proof of $A_2=0$ is analogous).
 Let us group the summands in $A_1$ according to $s$, $\sigma(n-l)$ and also the ordering of
 $\{\sigma^{-1}(1),\ldots,\sigma^{-1}(n-k)\}$ and $\{\sigma^{-1}(n-k+1),\ldots,\sigma^{-1}(n)\}$,
 which are given by elements $\sigma_1^{-1}\in \mathfrak{S}_{n-k}$ and $\sigma_2^{-1}\in \mathfrak{S}_k$.
 Define
   $$\omega^m_{\sigma_1,\sigma_2}(x_1;\ldots;x_n):=
     \prod_{i<j}^{j\leq n-k}\omega_{\sigma_1}^m(x_j,x_i)\cdot \prod_{i<j}^{n-k<i}\omega_{\sigma_2}^m(x_j,x_i).$$
 Then $A_1$ can be written in the form
   $$A_1=\sum_{t\leq n-k}\sum_{\sigma_1\in \mathfrak{S}_{n-k}}\sum_{\sigma_2\in \mathfrak{S}_k}
         A_{t,\sigma_1,\sigma_2}\frac{x_{\sigma_2(1)}}{x_t}\omega^m_{\sigma_1,\sigma_2}(x_1;\ldots;x_n)\
         \mathrm{with}\ A_{t,\sigma_1,\sigma_2}\in \ZZ.$$
 We claim that all these constants $A_{t,\sigma_1,\sigma_2}$ are zero.
 As an example, we compute $A_{t, 1_{n-k}, 1_k}$:
   $$A_{t, 1_{n-k},1_k}=\sum_{l=n-k-t}^{n-t-1}(-1)^l\binom{n-2}{l}\binom{l}{n-k-t}\binom{n-l-2}{t-1}=
     \frac{(-1)^{n-k-t}(n-2)!(1-1)^{k-1}}{(t-1)!(k-1)!(n-k-t)!}.$$
 Thus $A_{t, 1_{n-k},1_k}=0$ since $k>1$.
 Analogously $A_{t,\sigma_1,\sigma_2}=0$ for any $t,\sigma_1,\sigma_2$.
 Hence, $A_1=0$.
\end{proof}

\begin{proof}[Proof of Lemma~\ref{step 2.1}]$\ $

 The elements $L^m_n$ correspond to nonzero multiples of $\theta_{n,0}$
 via  $S^m\simeq \wt{\E}^+$. An algebraic independence of $\{\theta_{n,0}\}_{n\in \NN}$ follows
 from an analogue of Proposition~\ref{triangular_m}(b) for $\ddot{\bU}_{q_1,q_2,q_3}$.

 Alternatively, note that $\{e_0\}\cup\{[e_1;e_0;\ldots;e_0;e_{-1}]_n\}_{n=2}^\infty$ correspond
 to nonzero multiples of $\{D^n\}_{n=1}^\infty$ via $\ddot{U}_h(\gl_1)\simeq U(\bar{\dd}_q^0)$.
 The result follows from the PBW theorem applied to $U(\bar{\dd}_q^0)$.
\end{proof}


\section{Proofs of Theorems~\ref{Whittaker_m},~\ref{Whittaker_a}}

\subsection{Proof of Theorem~\ref{Whittaker_m}}$\ $

 In the fixed point basis, we have
   $v^K_r=\sum_{\bar{\lambda}} a_{\bar{\lambda}}\cdot [\bar{\lambda}]\
    \mathrm{with}\ a_{\bar{\lambda}}=\prod_{w\in T_{\bar{\lambda}}M(r,|\bar{\lambda}|)}(1-w)^{-1}$.
  Hence, it suffices to prove the following equality for any $r$-tuple of diagrams $\bar{\lambda}$:
\begin{equation}\label{C_in}
   C_{j,-n}=\sum_{\bar{\lambda}'}\frac{a_{\bar{\lambda}'}}{a_{\bar{\lambda}}}\cdot {K^{(m;j)}_{-n}}_{\mid{[\bar{\lambda}',\bar{\lambda}]}},
\end{equation}
 where the sum is over all $r$-tuples of diagrams $\bar{\lambda}'$ such that
 $\bar{\lambda}\subset \bar{\lambda}'$ and  $|\bar{\lambda}'|=|\bar{\lambda}|+n$.

 For such a pair $(\bar{\lambda}, \bar{\lambda}')$, define a collection of positive integers
\begin{equation}\label{indices}
  j_{1,1}\leq j_{1,2}\leq \cdots\leq j_{1,l_1},\
  j_{2,1}\leq j_{2,2}\leq \cdots\leq j_{2,l_2},\ \ldots,\
  j_{r,1}\leq j_{r,2}\leq \cdots\leq j_{r,l_r}
\end{equation}
 (with $\sum_{a=1}^r l_a=n$) via the following equality:
  $$\bar{\lambda}'=\bar{\lambda}+\square^1_{j_{1,1}}+\cdots+\square^1_{j_{1,l_1}}+\square^2_{j_{2,1}}+\cdots+
    \square^2_{j_{2,l_2}}+\cdots+\square^r_{j_{r,1}}+\cdots+\square^r_{j_{r,l_r}}.$$
 We also consider the sequence of $r$-tuples of diagrams
   $\bar{\lambda}=\bar{\lambda}^{[0]}\subset \bar{\lambda}^{[1]}\subset \cdots\subset \bar{\lambda}^{[n]}=\bar{\lambda}',$
 where $\bar{\lambda}^{[\Qq]}$ is obtained from $\bar{\lambda}$
 by adding the first $\Qq$ boxes from above.
 For $1\leq \Qq\leq n$, the $\Qq$th box from above has a form
 $\square^{s_{\Qq}}_{j_{s_{\Qq},i_{\Qq}}}$ and we denote its character by $\chi(\Qq)$.

 For any $F\in (S^m_n)^{\mathrm{opp}}$, we have the following formula
 for the matrix coefficient $F_{\mid{[\bar{\lambda}',\bar{\lambda}]}}$:
  $$F_{\mid{[\bar{\lambda}',\bar{\lambda}]}}=\frac{F(\chi(1),\ldots,\chi(n))}{\prod_{a<b}\omega^m(\chi(a),\chi(b))}
    \cdot \prod_{\Qq=1}^n f_{0\mid[\bar{\lambda}^{[\Qq]},\bar{\lambda}^{[\Qq-1]}]}.$$
 In particular, we get
  $${K^{(m;j)}_{-n}}_{\mid{[\bar{\lambda}',\bar{\lambda}]}}=
    \prod_{1\leq a<b\leq n}\frac{(\chi(a)-\chi(b))(\chi(b)-t_1\chi(a))}{(\chi(a)-t_2\chi(b))(\chi(a)-t_3\chi(b))}\cdot
    \prod_{\Qq=1}^n\chi(\Qq)^j\cdot \prod_{\Qq=1}^n f_{0\mid[\bar{\lambda}^{[\Qq]},\bar{\lambda}^{[\Qq-1]}]}.$$
 As an immediate consequence, we find
  ${K^{(m;j)}_{-n}}_{\mid{[\bar{\lambda}',\bar{\lambda}]}}=0$
 if $\bar{\lambda}'\backslash\bar{\lambda}$ contains two boxes
 in the same row of its $i$th component, $1\leq i \leq r$.
 Therefore, the sum in~(\ref{C_in}) should be taken only over those $\bar{\lambda}'$
 which correspond to collections $\{j_{1,1},\ldots,j_{r,l_r}\}$ from~(\ref{indices}) with strict inequalities.

 We also split $\frac{a_{\bar{\lambda}'}}{a_{\bar{\lambda}}}$ into the product over consequent pairs:
   $\frac{a_{\bar{\lambda}'}}{a_{\bar{\lambda}}}=\prod_{\Qq=1}^n \frac{a_{\bar{\lambda}^{[\Qq]}}}{a_{\bar{\lambda}^{[\Qq-1]}}}.$
 According to the Bott-Lefschetz fixed point formula, we have
  $$\frac{a_{\bar{\lambda}^{[\Qq]}}}{a_{\bar{\lambda}^{[\Qq-1]}}}\cdot f_{0\mid[\bar{\lambda}^{[\Qq]},\bar{\lambda}^{[\Qq-1]}]}=
    T\cdot e_{-r\mid[\bar{\lambda}^{[\Qq-1]},\bar{\lambda}^{[\Qq]}]}, \mathrm{where}\ T=(-t)^{r-2}\chi_1^{-1}\cdots \chi_r^{-1}.$$
 For two $r$-tuples of diagrams $(\bar{\mu},\bar{\mu}')$ such that
 $\bar{\mu}'=\bar{\mu}+\square^l_j$, the matrix coefficient $e_{-r\mid[\bar{\mu},\bar{\mu}']}$
 is computed by Lemma~\ref{matrix_coeff_Gis_m}(a):
  $$e_{-r\mid[\bar{\mu},\bar{\mu}']}=\frac{1}{1-t_1}\prod_{a=1}^r \frac{1}{t_1\chi^{(l)}_j-t_2^{L_a}\chi_a^{-1}}
    \prod_{(a,k)\ne (l,j)}^{k\leq L_a}\frac{\chi^{(l)}_j-t_2\chi^{(a)}_k}{\chi^{(l)}_j-\chi^{(a)}_k},$$
 where $\{L_a\}_{a=1}^r$ are chosen to satisfy $L_a\geq (\mu^{a*})_1+1$.

 Combining these formulas together, we finally get
  $$\frac{a_{\bar{\lambda}'}}{a_{\bar{\lambda}}}\cdot {K^{(m;j)}_{-n}}_{\mid{[\bar{\lambda}',\bar{\lambda}]}}=
    T^n \cdot \prod_{1\leq \Qq\leq n}\left\{\frac{(-t_1t_2)^{\Qq-1}}{1-t_1}\cdot
    \prod_{a=1}^r\frac{1}{\chi(\Qq)-t_2^{L_a}\chi_a^{-1}}\cdot
    \prod \frac{\chi(\Qq)-t_1t_2\chi^{(a)}_k}{\chi(\Qq)-t_1\chi^{(a)}_k}\cdot \chi(\Qq)^j  \right\},$$
 with the last product taken over pairs
 $(a,k)\notin \{(s_\Qq,j_{s_\Qq,i_\Qq})\}_{\Qq=1}^n,\ k\leq L_a$ with $L_a\geq (\lambda^{a*})_1+n$.

 Let us denote the RHS of this equality by $C_\bj$, where
 $\bj=\{j_{1,1},\ldots,j_{r,l_r}\}$ is defined in~(\ref{indices}).
 Note that $C_\bj=0$ if the corresponding $\bar{\lambda}'$ fails to be a
 collection of $r$ Young diagrams. Hence,~(\ref{C_in}) is reduced to $C_{j,-n}=\sum C_\bj$,
 where the sum is over all $\bj$ from ~(\ref{indices}) with strict inequalities.

 It is easy to check that $\sum C_\bj$ is a rational function in $\chi^{(a)}_k$ with no poles for $0\leq j\leq r$.
 The degree estimate implies that $\sum C_\bj$ is an element of $\mathbb{F}_r$ independent of $\bar{\lambda}$.
 Thus $v^K_r$ is indeed an eigenvector with respect to ${K^{(m;j)}_{-n}}$.
 To compute its eigenvalue, we evaluate $\sum C_\bj$ at $\bar{\lambda}=\bar{\emptyset}$.
 This sum is actually taken over all partitions $(l_1,\ldots,l_r)$ of $n$ with $j_{a,b}=b$, and it equals
 $$\frac{T^n(-t_1t_2)^{n(n-1)/2}}{(1-t_1)^n}
   \sum_{l_1+\ldots+l_r=n}\prod_{a,b=1}^r \frac{1}{(\chi_b^{-1}-t_2^{l_a}\chi_a^{-1})\cdots(t_2^{l_b-1}\chi_b^{-1}-t_2^{l_a}\chi_a^{-1})}
   \prod_{b=1}^r(t_2^{\frac{l_b(l_b-1)}{2}}\chi_b^{-l_b})^j=$$
 $$\frac{(-t)^{(r-2)n}(-t_1t_2)^{\frac{n(n-1)}{2}}}{(1-t_1)^n}
   \sum_{l_1+\ldots+l_r=n}\prod_{a,b=1}^r \frac{1}{(\chi_a-t_2^{l_a}\chi_b)\cdots(\chi_a-t_2^{l_a-l_b+1}\chi_b)}
   \prod_{b=1}^r(t_2^{-\frac{l_b(l_b-1)}{2}}\chi_b^{l_b})^{r-j}.$$

 It is straightforward to check that latter expression
 is a rational function in $\chi_a$ with no poles for $0\leq j\leq r-1$.
 Together with the degree estimate, we see that it is independent of $\chi_a$.
 To compute this constant, we let $\chi_1\to \infty$.
 Then the only nonzero contribution comes from the collection
 $(l_1,l_2,\ldots,l_r)=(n,0,\ldots,0)$ and the result equals $C_{j,-n}$.

 For $j=r$, the product of the above expression and $(\chi_1\cdots \chi_r)^n$ is a rational
 function in $\chi_a$ with no poles and of total degree $0$, hence, it is independent of $\chi_a$.
 To compute this constant, we let $\chi_1\to \infty$.
 The only nonzero contributions come from $(l_1,\ldots,l_r)$ with $l_1=0$.
 For these terms, we let $\chi_2\to \infty$, etc.
 The result follows from straightforward computations.

\subsection{Sketch of the proof of Theorem~\ref{Whittaker_a}}$\ $

 In the fixed point basis, we have
   $v^H_r=\sum_{\bar{\lambda}} b_{\bar{\lambda}}\cdot [\bar{\lambda}]\
    \mathrm{with}\ b_{\bar{\lambda}}=\prod_{w\in T_{\bar{\lambda}}M(r,|\bar{\lambda}|)}w^{-1}$.
 Hence, it suffices to prove the following equality for any $r$-tuple of diagrams $\bar{\lambda}$:
\begin{equation}\label{D_in}
   D_{j,-n}=\sum_{\bar{\lambda}'}\frac{b_{\bar{\lambda}'}}{b_{\bar{\lambda}}}\cdot {K^{(a;j)}_{-n}}_{\mid{[\bar{\lambda}',\bar{\lambda}]}},
\end{equation}
 where the sum is over all $r$-tuples of diagrams $\bar{\lambda}'$ such that
  $\bar{\lambda}\subset \bar{\lambda}'$ and  $|\bar{\lambda}'|=|\bar{\lambda}|+n$.

 Analogously to the K-theoretical case, we have
 $$\frac{b_{\bar{\lambda}'}}{b_{\bar{\lambda}}}\cdot {K^{(a;j)}_{-n}}_{\mid{[\bar{\lambda}',\bar{\lambda}]}}=
    \prod_{1\leq \Qq\leq n}\left\{\frac{(-1)^{\Qq}}{s_1}\cdot \prod_{a=1}^r\frac{1}{\chi(\Qq)-L_as_2+x_a}\cdot
    \prod \frac{\chi(\Qq)-x^{(a)}_k-s_1-s_2}{\chi(\Qq)-x^{(a)}_k-s_1}\cdot \chi(\Qq)^j  \right\},$$
 with the last product taken over pairs $(a,k)\notin \{(s_\Qq,j_{s_\Qq,i_\Qq})\}_{\Qq=1}^n,\ k\leq L_a$
 with $L_a\geq (\lambda^{a*})_1+n$.

 Let us denote the RHS of this equality by $D_\bj$.
 Then $\sum_\bj D_\bj$ is a rational function in $x^{(a)}_k$ and it has no poles for $j\geq 0$.
 The degree estimate implies that it is independent of $x^{(a)}_k$ for $0\leq j\leq r$.
 Thus $v^H_r$ is indeed an eigenvector with respect to $\{K^{(a;j)}_{-n}\}_{0\leq j\leq r}^{n>0}$.
 To compute the corresponding eigenvalues, we evaluate $\sum_\bj D_\bj$ at $\bar{\lambda}=\bar{\emptyset}$.
 This sum equals
   $$\frac{(-1)^{n(n+1)/2}}{s_1^n}\sum_{l_1+\ldots+l_r=n}
     \left\{\prod_{a,b=1}^r \prod_{k=1}^{l_b}(x_a-x_b-(l_a-k+1)s_2)^{-1}\prod_{b=1}^r\prod_{k=1}^{l_b}((k-1)s_2-x_b)^j\right\}.$$

 It is straightforward to check that this sum is a rational function in $x_a$ with no poles.
 Together with the degree estimate for $j\leq r-1$, we see that it is independent of $x_a$.
 To compute this constant, we let $x_1\to \infty$. For $j<r-1$, all the summands tend to $0$.
 For $j=r-1$, the only nonzero contribution comes from
 $(l_1,l_2,\ldots,l_r)=(n,0,\ldots,0)$ and equals $D_{r-1,-n}$.

\begin{rem}
 An explicit formula for the eigenvalues $D_{r,-n}\ (n>1)$ was first obtained in~\cite{SV3}.
\end{rem}


\end{document}